\documentclass[12pt,a4paper]{amsart}
\pdfoutput=1
\usepackage{mathtools}
\usepackage{xcolor}
\usepackage{amsmath,amstext,amssymb,amsthm,amsfonts}
\usepackage[all]{xy} \xyoption{poly}

\usepackage{float}
\usepackage{array, cellspace}
\usepackage{tablefootnote}
\usepackage{placeins}
\usepackage{pdfpages}

\newcommand{\nc}{\newcommand}

\nc{\one}{\mbox{\bf 1}}
\nc{\invtensor}{\underset{\leftarrow}{\otimes}}
\nc{\const}{\operatorname{const}}

\nc{\ad}{\operatorname{ad}}
\nc{\tr}{\operatorname{tr}}

\nc{\tp}{\operatorname{top}}
\nc{\rank}{\operatorname{rank}}
\nc{\corank}{\operatorname{corank}}
\nc{\codim}{\operatorname{codim}}
\nc{\sdim}{\operatorname{sdim}}
\nc{\mult}{\operatorname{mult}}

\nc{\fspn}{\operatorname{span}}
\nc{\Sym}{\operatorname{Sym}}
\nc{\sym}{\operatorname{sym}}
\nc{\id}{\operatorname{id}}
\nc{\Id}{\operatorname{Id}}
\nc{\Ree}{\operatorname{Re}}

\nc{\htt}{\operatorname{ht}}
\nc{\an}{\operatorname{an}}
\nc{\pr}{\operatorname{pr}}
\nc{\str}{\operatorname{str}}

\nc{\Ker}{\operatorname{Ker}}
\nc{\rker}{\operatorname{rKer}}
\nc{\im}{\operatorname{Im}}
\nc{\osp}{\mathfrak{osp}}
\nc{\sgn}{\operatorname{sgn}}
\nc{\F}{\operatorname{F}}
\nc{\Mod}{\operatorname{Mod}}
\nc{\DS}{\operatorname{DS}}
\nc{\Soc}{\operatorname{Soc}}
\nc{\Inj}{\operatorname{Inj}}
\nc{\fhom}{\operatorname{Hom}}
\nc{\End}{\operatorname{End}}
\nc{\supp}{\operatorname{supp}}
\nc{\Card}{\operatorname{Card}}
\nc{\Ann}{\operatorname{Ann}}
\nc{\Ind}{\operatorname{Ind}}
\nc{\Coind}{\operatorname{Coind}}
\nc{\wt}{\widetilde}
\nc{\ch}{\operatorname{ch}}
\nc{\sch}{\operatorname{sch}}

\nc{\Stab}{\operatorname{Stab}}
\nc{\Sch}{{\mathcal S}\mbox{\em ch}}
\nc{\Irr}{\operatorname{Irr}}
\nc{\fspec}{\operatorname{Spec}}
\nc{\Res}{\operatorname{Res}}
\nc{\Aut}{\operatorname{Aut}}
\nc{\Ext}{\operatorname{Ext}}
\nc{\chr}{\operatorname{char}}
\nc{\Fract}{\operatorname{Fract}}
\nc{\gr}{\operatorname{gr}}

\nc{\deff}{\operatorname{def}}
\nc{\fhC}{\operatorname{HC}}
\nc{\Sk}{\operatorname{Sk}}
\nc{\Sp}{\operatorname{Sp}}
\nc{\Isos}{\operatorname{Isos}}
\nc{\red}{\operatorname{red}}

\newcommand{\re}{\mathit{re}}
\newcommand{\is}{\mathit{iso}}
\newcommand{\nr}{\mathit{nr}}

\nc{\ima}{\operatorname{im}}

\nc{\U}{\operatorname{U}}
\nc{\ttC}{\operatorname{C}}

\nc{\GL}{\operatorname{GL}}
\nc{\wdchi}{\widetilde{\chi}}
\nc{\wdH}{\widetilde{H}}
\nc{\wdN}{\widetilde{N}}
\nc{\wdM}{\widetilde{M}}
\nc{\wdO}{\widetilde{O}}
\nc{\wdR}{\widetilde{R}}

\nc{\wdV}{\widetilde{V}}

\nc{\wdC}{\widetilde{C}}

\nc{\Obj}{\operatorname{Obj}}
\nc{\Dglie}{\operatorname{{\mathcal D}glie}}
\nc{\Fin}{\operatorname{{\mathcal F}in}}

\nc{\Adm}{\operatorname{\mathcal{A}dm}}

\nc{\Sg}{{\cS(\fg)}}
\nc{\Shg}{{\cS(\fhg)}}
\nc{\Ug}{{\cU(\fg)}}
\nc{\Uhg}{{\cU(\fhg)}}
\nc{\Sh}{{\cS(\fh)}}
\nc{\Uh}{{\cU(\fh)}}
\nc{\Uhh}{{\cU(\fhh)}}
\nc{\Zg}{{{\mathcal{Z}}(\fg)}}

\nc{\Vir}{{\mathcal{V}ir}}
\nc{\NS}{{\mathcal{N}S}}

\nc{\tZg}{{\widetilde{\mathcal Z}({\mathfrak g})}}
\nc{\Zk}{{\mathcal Z}({\mathfrak k})}

\newcommand{\ZZ}{\mathbb{Z}}

\nc{\Up}{{\mathcal U}({\mathfrak p})}
\nc{\Ah}{{\mathcal A}({\mathfrak h})}
\nc{\Ag}{{\mathcal A}({\mathfrak g})}
\nc{\Ap}{{\mathcal A}({\mathfrak p})}
\nc{\Zp}{{\mathcal Z}({\mathfrak p})}
\nc{\cR}{\mathcal R}
\nc{\cS}{\mathcal S}
\nc{\cT}{\mathcal{T}}
\nc{\cC}{\mathcal C}
\nc{\cA}{\mathcal A}
\nc{\cB}{\mathcal B}
\nc{\cU}{\mathcal U}
\nc{\cZ}{\mathcal Z}
\nc{\cM}{\mathcal M}
\nc{\cL}{\mathcal L}
\nc{\cF}{\mathcal F}
\nc{\fg}{\mathfrak g}

\nc{\cK}{\mathcal{K}}
\nc{\CO}{\mathcal O}
\nc{\CR}{\mathcal R}

\nc{\cW}{\mathcal{W}}
\nc{\bM}{\mathbf{M}}
\nc{\bL}{\mathbf{L}}
\nc{\bN}{\mathbf{N}}

\nc{\zq}{\mathpzc q}

\nc{\fo}{\mathfrak o}
\nc{\fl}{\mathfrak l}
\nc{\fn}{\mathfrak n}
\nc{\fm}{\mathfrak m}
\nc{\fp}{\mathfrak p}
\nc{\fh}{\mathfrak h}
\nc{\ft}{\mathfrak t}
\nc{\fk}{\mathfrak k}
\nc{\fb}{\mathfrak b}
\nc{\fs}{\mathfrak s}
\nc{\fB}{\mathfrak B}

\nc{\vareps}{\varepsilon}
\nc{\varesp}{\varepsilon}
\nc{\veps}{\varepsilon}

\nc{\fsl}{\mathfrak{sl}}
\nc{\fgl}{\mathfrak{gl}}
\nc{\fso}{\mathfrak{so}}
\nc{\fosp}{\mathfrak{osp}}
\nc{\fsp}{\mathfrak{sp}}
\nc{\fq}{\mathfrak q}
\nc{\fr}{\mathfrak r}
\nc{\fsq}{\mathfrak{sq}}
\nc{\fpsl}{\mathfrak{psl}}

\nc{\fhg}{\fhat{\fg}}
\nc{\fhn}{\fhat{\fn}}
\nc{\fhh}{\fhat{\fh}}
\nc{\fhb}{\fhat{\fb}}
\nc{\fhrho}{\fhat{\rho}}

\nc{\fhsl}{\fhat{\fsl}}
\nc{\fpo}{\mathfrak{po}}
\nc{\dirlim}{\underset{\rightarrow}{\lim}\,}
\nc{\nen}{\newenvironment}
\nc{\ol}{\overline}
\nc{\ul}{\underline}
\nc{\ra}{\rightarrow}
\nc{\lra}{\longrightarrow}
\nc{\Lra}{\Longrightarrow}
\nc{\bo}{\bar{1}}
\nc{\Lla}{\Longleftarrow}

\nc{\Llra}{\Longleftrightarrow}

\nc{\thla}{\twoheadleftarrow}

\nc{\lang}{(}
\nc{\rang}{)}

\nc{\fhra}{\fhookrightarrow}

\nc{\iso}{\overset{\sim}{\lra}}

\nc{\ssubset}{\underset{\not=}{\subset}}

\nc{\vac}{|0\rangle}

\nc{\Thm}[1]{Theorem~\ref{#1}}
\nc{\Prop}[1]{Proposition~\ref{#1}}
\nc{\Lem}[1]{Lemma~\ref{#1}}
\nc{\Cor}[1]{Corollary~\ref{#1}}
\nc{\Conj}[1]{Conjecture~\ref{#1}}
\nc{\Claim}[1]{Claim~\ref{#1}}
\nc{\Defn}[1]{Definition~\ref{#1}}
\nc{\Exa}[1]{Example~\ref{#1}}
\nc{\Rem}[1]{Remark~\ref{#1}}
\nc{\Note}[1]{Note~\ref{#1}}
\nc{\Quest}[1]{Question~\ref{#1}}
\nc{\fhyp}[1]{Hypoth\`ese~\ref{#1}}
\nen{thm}[1]{\label{#1}{\bf Theorem.\ } \em}{}
\nen{prop}[1]{\label{#1}{\bf Proposition.\ } \em}{}
\nen{lem}[1]{\label{#1}{\bf Lemma.\ } \em}{}
\nen{cor}[1]{\label{#1}{\bf Corollary.\ } \em}{}
\nen{conj}[1]{\label{#1}{\bf Conjecture.\ } \em}{}

\nen{claim}[1]{\label{#1}{\bf Claim.\ } \em}{}

\nen{defn}[1]{\label{#1}{\bf Definition.\ } }{}
\nen{exa}[1]{\label{#1}{\bf Example.\ } }{}
\nen{rem}[1]{\label{#1}{\em Remark.\ } }{}
\nen{note}[1]{\label{#1}{\em Note.\ } }{}
\nen{exer}[1]{\label{#1}{\em Exercise.\ } }{}
\nen{sket}[1]{\label{#1}{\em Sketch of proof.\ } }{}
\nen{quest}[1]{\label{#1}{\bf Question.\ } \em}{}

\nen{hyp}[1]{\label{#1}{\bf Hypoth\`ese.\ } \em}{}
\begin{document}
\setcounter{section}{-1}
\setcounter{tocdepth}{1}
	
\title{On the  root system of a Kac-Moody superalgebra }
\author{Maria Gorelik, Shay Kinamon Kerbis}
\address{Department of Mathematics, Weizmann Institute of Science, Rehovot 761001, Israel; 
maria.gorelik@weizmann.ac.il, shay-kinamon.kerbis@weizmann.ac.il}
%

\date{}

\begin{abstract} 
In this paper we extend several results about root systems of Kac-Moody algebras to superalgebra context.
In particular, we describe the root bases and the sets of imaginary roots.
\end{abstract}


\medskip


\maketitle
\begin{center}
{Dedicated to V.~G.~Kac in the occassion of his 80th birthday.}
\end{center}

\section{Introduction}
In this paper we study the root systems of Kac-Moody superalgebras.
The root systems of Kac-Moody algebras were studied by Vinberg, Kac, Moody, Peterson and others.

A contragredient Lie algebra is a Lie algebra constructed from a complex 
square matrix by a 
standard procedure described in~\cite{Kbook}, Chapter I.
Kac-Moody algebras correspond to the matrices satisfying the following conditions
for all indices $i,j$:
\begin{itemize}
\item[1.]$a_{ij}=0\ \Longrightarrow\ a_{ji}=0$;
\item[2.] $a_{ii}\not=0$;
\item[3.]  $\frac{2a_{ij}}{a_{ii}}\in\mathbb{Z}_{\leq 0}$ for $j\not=i$.
\end{itemize}
A similar procedure assigns a  certain Lie superalgebra to any complex $n\times n$-matrix $A$
and a parity function $p: X\to \mathbb{Z}_2$, where $X:=\{1,2,\ldots,n\}$.
We call $(A,p)$ the {\em Cartan datum}. Kac-Moody superalgebras 
correspond to the Cartan data satisfying certain conditions 
(the condition (1) for all $i,j$, the condition  (2) for all $i$ with $p(i)=\ol{0}$, and a certain modification of
 the condition  (3))
which, unlike the Lie algebra
case,   should be imposed not only on the Cartan datum $(A,p)$ but
on all Cartan data obtained from $(A,p)$ by so-called
''isotropic reflexions''. 

A big part of the theory described in the  Kac's book ``Infinite-dimensional Lie Algebras'' 
can be developed in the context of  Kac-Moody superalgebras. 
The construction (Chapter I) was developed in~\cite{S},\cite{GHS}.
The results of Chapter II (``Invariant bilinear form'') and many results of Chapter III
can be extended straightforwardly.  Chapter IV is devoted to classification results; for
 Kac-Moody superalgebras
the corresponding classification  was obtained  by I.~van de Leur, C.~Hoyt and V.~Serganova (\cite{vdL}, \cite{Hoyt}, \cite{S}), see~\ref{classif} below.
In this paper we extend most of results of Chapter V (``Real and Imaginary Roots'').
Some of these results can be found in~\cite{S}, we give precise references in the text.

Different manifestations of the Weyl group in the context of Kac-Moody algebras lead
to three different notions for superalgebras: the skeleton (an analogue of the Cayley
graph of the Weyl group),  the group $\Sk^D(v)$
and the Weyl group. In this paper we study these objects and 
establish a decomposition of the group $\Sk^D(v)$ analogous
to~\cite{Kbook}, Proposition 6.5. A version of Matsumoto Theorem for skeleta in recently proven
in~\cite{GHS2}.

\subsection{Skeleton}\label{skeleton}
By construction, a  Kac-Moody superalgebra $\fg$ possesses a triangular decomposition
$\fg=\fn^-\oplus \fh\oplus \fn^+$ where  the Cartan subalgebra $\fh$
is a commutative even Lie algebra which acts diagonally in the adjoint representation.
The subalgebras $\fn^{\pm}$ are generated by the root spaces $\fg_{\pm\alpha}$ where  $\alpha$ 
runs through the set of simple roots $\Sigma$.

Starting from the original triangular decomposition, one can construct many more, using ``reflexions'' $r_x$ numbered by the indices $x\in X$.
We call such triangular decompositions ``attainable''.
The attainable triangular decompositions are parametrized by a regular graph, which we call
a {\em skeleton}: its vertices
correpond to different triangular decompositions
and the edges are marked by $r_x$ with $x\in X$.
 Each chain
of indices $x_1,\ldots, x_s\in X$ represents  a path between the corresponding triangular decompositions.
We denote the skeleton by $\Sk(v)$ (where the vertex $v$ indicates the original triangular decomposition).
For each vertex $u\in \Sk(v)$ we denote by $\Sp(u)$ the full subgraph corresponding
to the triangular decompositions which induce the same triangular decomposition of $\fg_{\ol{0}}$.
By~\cite{S}, $\Sp(u)$ is a connected subgraph of $\Sk(v)$. 
Similar objects appear in the Weyl groupoids
introduced in~\cite{HY08} and developed in~\cite{HS20}.

The Weyl group $W(\fg)$ acts faithfully on the set of vertices of the skeleton;
this action preserves the edges and the markings of the edges.
In the case of Kac-Moody algebras, this action is transitive:
the skeleton is naturally isomorphic to the Cayley graph of 
 $W(\fg)$ and the spine is a trivial graph $\Sp(v)=\{v\}$.
In general, the spine is a connected subgraph of $\Sk(v)$ and
for any vertex $u\in\Sk(v)$ there exists a unique $w\in W(\fg)$ and $v'\in \Sp(v)$ such
that $u=wv'$.  
The Weyl group can be naturally embedded in another group,
$\Sk^D(v)$, which also  acts faithfully on the set of vertices of the skeleton,
preserving the edges and the marks on the edges.
The elements of  $\Sk^D(v)$ map each triangular
decomposition $\fg=\fn^-\oplus \fh\oplus \fn^+$ to an isomorphic triangular decomposition
$\fg=(\fn')^-\oplus \fh\oplus (\fn')^+$.  
 
 \subsection{Root systems}\label{introroot}
 We denote by $Q_v$ the root lattice (=$\mathbb{Z}$-span of $\Sigma$).
 For each vertex $u$ in the skeleton we denote by $\Sigma_u$
 the corresponding set of simple roots and by $Q^+_u$ the semilattice spanned by $\Sigma_u$.

Let $\Delta$ be the root system of $\fg$ (i.e., $\Delta\subset \fh^*$ is the set of non-zero eginevalues
of $\fh$ in the adjoint representation).  Then $\Delta$ spans $Q_v$ and 
for each $u\in \Sk(v)$ one has 
$\Delta=-\Delta^+_u\coprod \Delta^+_u$ where $\Delta^+_u=\Delta\cap Q^+_u$. Moreover,
$\Sigma_{wu}=w\Sigma_u$ for each $w\in\Sk^D(v)$ (in particular, for $w\in W(\fg)$).

 We call a root $\alpha\in\Delta$ {\em real }
if $\alpha\in\Sigma_u$ for some $u\in \Sk(v)$ and
{\em imaginary} if $\alpha$ is not proportional to a real root.
We denote by $\pi$ the set of even roots $\alpha$ with the property
that $\alpha$ or $\alpha/2$ is a simple root at some vertex $u\in\Sp(v)$.
 For each $\alpha\in\pi$
one defines a coroot $\alpha^{\vee}$ in the usual way.
The Weyl group is a Coxeter group generated by the reflections with 
respect to the roots in $\pi$.

For the Kac-Moody algebras $\Sp(v)=\{v\}$ and $\Sigma_v=\pi$. 
For the finite dimensional Kac-Moody superalgebras $\pi$ is the set of simple roots for $\fg_{\ol{0}}$.

\subsection{Classification}\label{classif}
For a set $S\subset\fh^*$ we denote by $C(S)$ the real  convex cone 
generated by $S$.

We divide indecomposable Kac-Moody superalgebras into three classes: 
\begin{itemize}
\item[(Fin)] $\fg$ is finite-dimenional;
\item[(Aff)]  $\fg$ is infinite-dimensional with a finite Gelfand-Kirillov dimension;
\item[(Ind)] $\fg$ is of infinite Gelfand-Kirillov dimension. 
\end{itemize}
In the Lie algebra case the corresponding  classification of  the Cartan matrices is given by Theorem 4.3
in~\cite{Kbook}. In general, this classification can be described in terms of the {\em totally positive cone}
$$Q^{++}_{\mathbb{R}}:=\bigcap_{u\in \Sk(v)} C(\Sigma_u).$$
(denoted as $D^+_{\Pi}$ in~\cite{S}, Section 4). 
Let $r$ be the (real) dimension of the convex cone
 $Q^{++}_{\mathbb{R}}$.
By~\cite{GHS} we have  $r=0$ for (Fin), $r=1$ for (Aff) and  $r>1$ for (Ind).

The finite-dimensional Kac-Moody superalgebras (type (Fin)) were described by V.~Kac in~\cite{Ksuper}.
The symmetrizable Kac-Moody superalgebras of type (Aff) were described by I.~van de Leur in~\cite{vdL}.
The non-symmetrizable Kac-Moody superalgebras of type (Aff) and the Kac-Moody superalgebras of type (Ind) 
were described by C.~Hoyt and V.~Serganova in~\cite{Hoyt}, \cite{S}. These results
imply that in the case (Aff) the Gelfand-Kirillov dimension equals to one.

\subsection{Root systems of Kac-Moody superalgebras}\label{intro1}
In the first part of our paper we study root systems of Kac-Moody superalgebras.
We extend the main results of Chapter V in~\cite{Kbook} to the Kac-Moody superalgebras.
For the Kac-Moody algebras the closure of the Tits
cone is the set $\{h\in\fh_{\mathbb{R}}|\ \langle h,\alpha\rangle\geq 0\ \ \text{ for all }
\alpha\in Q^{++}_{\mathbb{R}}\}$ (see~\cite{Kbook}, Proposition 5.8).
In this paper we do not introduce $\fh_{\mathbb{R}}$ and the Tits cone
(since Cartan matrices of Kac-Moody superalgebras may have non-real entries), but 
 work with $Q^{++}_{\mathbb{R}}$ instead.  We do not dicuss the results
 for symmetrizable Kac-Moody algebras which appear in Chapter V in~\cite{Kbook};
 it is easy to see that most of these results   do not hold in superalgebra context.
 Another result which does not hold for superalgebras is 
Corollary 5.10 in~\cite{Kbook} describing the group of automorphisms of the root lattice
preserving the root system. It would be interesting to have such  a descrption for Kac-Moody superalgebras.

\subsubsection{Root bases} 
A linearly independent subset $\Sigma'\subset \Delta$ is called
a {\em root basis} if for 
each $\alpha\in\Delta$  either $\alpha$ or $-\alpha$ is a non-negative integral linear combination
of $\Sigma'$.  Each set of simple roots
is a root basis.  By~\cite{K1978} for an indecomposable Kac-Moody algebra   each root basis $\Sigma'$
 is $W$-conjugated to
 $\Sigma$ or to $-\Sigma$ (see~\cite{Kbook}, Proposition 5.9). 
We denote by  ``$\operatorname{Root\ Bases}$''  the set of root bases.
 
We call $h\in\fh^*$ {\em generic} if $\Ree\langle \alpha,h\rangle \not=0$ for all $\alpha\in\Delta$.
For a generic $h\in\fh^*$ we set
$$\Delta_{>0}(h):=\{\alpha\in\Delta|\ \Ree\langle \alpha,h\rangle >0\}$$
and denote by $\tilde{\Sigma}_{h}$ the set of indecomposable elements in $\Delta_{>0}(h)$.

The following theorem is a slight modification of Theorem 8.3 in~\cite{S}; it
extends Proposition 5.9 in~\cite{Kbook} to Kac-Moody superalgebras.

  \subsubsection{} 
  \begin{thm}{thmiii}
\begin{enumerate}
\item For a  Kac-Moody component of type (Fin) we have 
$$\{\Sigma_u\}_{u\in \Sk(v)}=\operatorname{Root\ Bases}=
\{\tilde{\Sigma}_{h}\}_{h\text{ is generic in } \fh}.$$
\item For an indecomposable Kac-Moody component of type (Aff) we have
$$\{\Sigma_u\}_{u\in \Sk(v)}\coprod\{-\Sigma_u\}_{u\in \Sk(v)}
=\operatorname{Root\ Bases}=
\{\tilde{\Sigma}_{h}\}_{h\text{ is generic in } \fh}.$$
\item For an indecomposable Kac-Moody component of type (Ind) we have
$$\{\Sigma_u\}_{u\in \Sk(v)}\coprod\{-\Sigma_u\}_{u\in \Sk(v)}
=\operatorname{Root\ Bases}\subsetneq
\{\tilde{\Sigma}_{h}\}_{h\text{ is generic in } \fh}.$$
\end{enumerate}\end{thm}

We say that  an indecomposable root system $\Delta$ is of {\em type II}  if any root of $\fg$ lies
in the $\mathbb{R}$-span of the set of even roots,
and is of type I otherwise. 
 Using~\Thm{thmiii} we prove the following 
  \subsubsection{}
\begin{thm}{corintrospine}
For an  indecomposable Kac-Moody superalgebra we have
\begin{enumerate}
\item One has $\bigcap\limits_{u\in \Sp(v)} C(\Sigma_u)=C(\pi)$ and $\pi$ is finite.

\item One has $Q^{++}_{\mathbb{R}}=\bigcap\limits_{w\in W} wC(\pi)=\bigcup\limits_{w\in W} wK_{\mathbb{R}}$, 
where 
$$K_{\mathbb{R}}:=\{ \mu\in C(\pi)|\ \langle \mu,\alpha^{\vee}\rangle \leq 0\ 
\text{ for all }\alpha\in\pi\}.$$
Moreover, any element in $Q^{++}_{\mathbb{R}}$ is $W$-conjugated to a unique element in $K_{\mathbb{R}}$.

\item For an  isotropic\footnote{In this paper we use the definitions of~\cite{GHS} and
all isotropic roots are real.}
 root $\alpha$ the following conditions are equivalent:
\begin{itemize}
\item[(a)] $\alpha\in \Sigma_u$ for some $u\in\ Sp(v)$;
\item [(b)] $\alpha\not\in \bigl(-C(\pi)\cup C(\pi)\bigr)$.
\end{itemize}

\item The spine is infinite if and only if $\Delta$ is infinite and of type I.
\end{enumerate}
\end{thm}

Note that all assertions in~\Thm{corintrospine} are trivial if $\Sp(v)=v$ (in particular, in
the case of Kac-Moody algebras). 

For the cases (Fin) and (Aff)  \Thm{thmiii} follows from usual properties of root bases
 and the assertions (i)---(iii) of~\Thm{corintrospine} follow from this theorem.
 In the case when the spine is finite~\Thm{corintrospine} (i) is established in~\cite{S} Lemma 3.10.

\subsubsection{Imaginary roots}\label{introimaginary}
Set
$$Q^{++}:=Q^{++}_{\mathbb{R}}\cap Q_v.$$
Let $(\Delta^{\ima})^+$ be the set of positive imaginary roots. 
We prove that
\begin{equation}\label{suppconnected}
(\Delta^{\ima})^+=\{\mu\in Q^{++}|\ \ \supp\mu\ \text{ is connected}\}
\end{equation}
where for $\mu=\sum\limits_{\alpha\in\Sigma} k_{\alpha}\alpha$
we set $\supp\mu:=\{\alpha\in\Sigma|\ k_{\alpha}\not=0\}$ and $\supp\mu$
is connected if the corresponding Dynkin diagram is connected. 
Combining the above formula with~\Thm{corintrospine} (ii),
we obtain a  generalization of~\cite{Kbook},  Proposition 5.2 (a) and Theorem 5.4.

As in the case of Kac-Moody algebras, this implies  that 
for the Cartan datum of a Kac-Moody superalgebra, all ``root superalgebras''
have the same set of imaginary roots (the multiplicities may vary);
this extends Corollary 5.12 in~\cite{Kbook} to Kac-Moody superalgebras.

In the case when $\Delta$ is not ``purely anisotropic'' (this means that
$\Sigma$ contains an isotropic root) we have
\begin{equation}\label{eqimagine}
(\Delta^{\ima})^+=Q^{++}.
\end{equation}

\subsubsection{Algebra $\fg(B_{\pi})$}
By~\cite{S}, Lemma 3.7, Lie algebra $\fg_{\ol{0}}$ contains a subalgebra $\fg'$ which can be seen as
``the best approximation to the largest almost contragredient Lie subalgebra''.
The corresponding contragredient Lie algebra has the Cartan matrix $B_{\pi}$
with the entries $(\langle \beta,\alpha^{\vee}\rangle)_{\alpha,\beta\in\pi}$. For example, if 
$\fg$ is finite-dimensional, then $\fg'=\fg(B_{\pi})=[\fg_{\ol{0}},\fg_{\ol{0}}]$.
The Weyl group $W(\fg)$ is naturally isomorphic to the Weyl group of $\fg(B_{\pi})$.

Denote by $\tilde{\pi}$ the set of simple roots for $\fg(B_{\pi})$. 
We have a linear map  which sends $\tilde{\pi}$
to $\pi$ and the root system of $\fg(B_{\pi})$ to $\Delta_{\ol{0}}$
(this map is injective in the types (Fin) and (Ind)). Denote by $\Delta^{\ima}_{\pi}$
the image of the set of imaginary roots of $\fg(B_{\pi})$.
Then $\Delta^{\ima}_{\pi}\subset \Delta^{\ima}_{\ol{0}}$.
Since  $\fg(B_{\pi})$ is a Kac-Moody algebra, $\Delta^{\ima}_{\pi}$ is given
by the formula~(\ref{suppconnected}). 
Our proofs of~\Thm{thmiii} (iii) and the formula~(\ref{suppconnected}) are based on 
this description of $\Delta^{\ima}_{\pi}$ and
the Hoyt's classification
of the Kac-Moody superalgebras of type (Ind).

Our results imply  the following formula for an indecomposable Kac-Moody superalgebra
$$\Delta^{\ima}=\{\alpha\in Q_v|\ 2\alpha\in \Delta^{\ima}_{\pi}\},\ \ \ \ 
Q^{++}=\{\alpha\in Q_v|\ 2\alpha\in Q^{++}_{\pi}\}.$$

\subsection{Group $\Sk^D(v)$}
As in the case of Kac-Moody algebras, part of the theory can be naturally developed in a greater generality. In this paper we use the framework of root groupoids developed
 in~\cite{S},~\cite{GHS}. Lie superalgebras in this approach are described by 
certain connected components of so-called {\sl root groupoid} $\cR$.
The root groupoid has a distinguished collection of generators, similarly
to the standard generators of a Coxeter group. 
In~\cite{GHS} it was shown that the group $\Aut_\cR(v)$, the automorphism group of the vertex $v$ of 
the root groupoid, contains two commuting normal subgroups: the Weyl group $W$ and the group
of ``irrelevant automorphisms''.  Any element in  $\Aut_\cR(v)$ lifts to an automorphism
of the corresponding Kac-Moody superalgebra and this automorphism preserves $\fh$.
Irrelevant automorphisms preserve all root spaces and
 the quotient $\Aut_\cR(v)/K(v)$ identifies with $\Sk^D(v)$. 

The group $\Sk^D(v)$ can be seen as
 a subgroup of the group of automorphisms of the root lattice preserving the root system.
Denote by $\Sp^D(v)$ the stabilizer of the spine $\Sp(v)$ under the action of $\Sk^D(v)$ (induced by the action on the root lattice). One has $\Sk^D(v)=W\rtimes \Sp^D(v)$; moreover,
$\Aut_\cR(v)/(W\times K(v))$ identifies with $\Sp^D(v)$. 

The group $\Sp^D(v)$ stabilizes the set $\pi$. This defines a group  homomorphism from
$\Sp^D(v)$ to the automorphisms of the Dynkin diagram of $\pi$.
This map is injective if $\pi$ spans the root lattice (in particular, for the Kac-Moody superalgebras of
 the type II).

\subsubsection{The group $\Sp^D(v)$ for the Kac-Moody superalgebras} 
Combining the results
of~\cite{GHS}, \cite{Shay} and this paper we obtain the following description
of the group $\Sp^D(v)$ for the indecomposable Kac-Moody superalgebras:
 $$\Sp^D(v)\cong \left\{\begin{array}{ll}
 \mathbb{Z}_2\ & \text{ for }A(n|n),\ q_{2n}^{(2)}\\
 \mathbb{Z}\ & \text{ for }C(n+1)^{(1)}, \ A(m|n)^{(1)} \ (m\not=n),\\
 \mathbb{Z}\rtimes \mathbb{Z}_2\ & \text{ for }A(n|n)^{(1)},\\
 1 & \text{ for other indecomposable Kac-Moody superalgebras}.
 \end{array}\right.$$

 In all cases, except $A(n|n)$ and $A(n|n)^{(1)}$, 
  the image of $\Sp^D(v)$ in the automorphism group of the Dynkin diagram of $\pi$
  is trivial; for  $A(n|n)$ and $A(n|n)^{(1)}$ 
  the image contains the identity and the involution 
  which switches two connected components of the Dynkin diagram of $\pi$.

 \subsubsection{The group $\Sk^D(v)$ for type (Aff)}
 By above, $\Sk^D(v)=W(\fg)$ except for the cases $A(n|n)$,
 $A(m|n)^{(1)}$, $\fq_{2n}^{(2)}$  and $C(n+1)^{(1)}$. 
 
  By~\cite{Kbook}, Proposition 6.15,  for
  each affine Kac-Moody algebra the Weyl group admits a decomposition
  $W(\fg)={W}(\dot{\Delta})\ltimes T$, where the root system
  $\dot{\Delta}$ is   the image of the set of real roots  in $\fh^*/\mathbb{C}\delta$
  and $T$ is an abelian group which is isomorphic to $\mathbb{Z}\dot{\Delta}$.
 For the affine Kac-Moody superalgebras  we have  $W(\fg)=W(\fg_{\pi})$ and
 the above description gives a decomposition 
 $W(\fg)={W}(\dot{\Delta}_{\pi})\ltimes T_{\pi}$.

For the cases $A(m|n)^{(1)}$ the  following formula was established in~\cite{GHS}, 10.3:
\begin{equation}\label{eqintroSkDv}
\Sk^D(v)={\Sk}^D(\dot{v})\ltimes T 
\end{equation}
${\Sk}^D(\dot{v})$ is the corresponding group for the root system $\dot{\Delta}$
and $T$ is naturally isomorphic to the root lattice of $\dot{\Delta}$.
In Section 9 we prove this formula for the case $C(n+1)^{(1)}$.  
For the remaining case $\fq_{2n}^{(2)}$ a smilar decomposition
is obtained in~\cite{GHS}, 10.3, but the isomorphism between $T$ and $\mathbb{Z}\dot{\Delta}$ is not ``natural''.
One has $T/T_{\pi}\cong \Sp^D(v)/{\Sp}^D(\dot{v})$ (so $T/T_{\pi}$ is isomorphic 
to $\mathbb{Z}$ for $A(m|n)^{(1)}$, $C(n+1)^{(1)}$ and to
$\mathbb{Z}_2$ for $\fq_{2n}^{(2)}$).

\subsubsection{}
In this paper we introduce $\Sk^D(v)$  as
 a subgroup of the group of automorphisms of the root lattice.
Using the formula~(\ref{eqintroSkDv}) we describe, for the Kac-Moody superalgebras,
 an embedding of $\Sk^D(v)$ into $GL(\fh^*)$.
 
By~\cite{GHS} one has $\Aut_{\cR}(v)=W\times K(v)$ if $\Sp^D(v)=1$ and 
$\Sk^D(v)\cong \Aut_{\cR}(v)/K(v)$. The embedding of $\Sk^D(v)$ into $GL(\fh^*)$ induces
a group isomorphism  
$$\Aut_{\cR}(v) \cong \Sk^D(v)\times K(v)$$
 for all Kac-Moody superalgebras, see~\cite{Shay}.

\subsection{Content of the paper}

In Section 1 we present a preliminary information on root groupoids, mostly borrowed from~\cite{GHS}. 

In Section 2 we discuss root bases in the general setup and compare root systems of type I and type II;
in particular, we prove~\Prop{lempis}. In addition, we
 prove Theorems~\ref{thmiii}  and~\ref{corintrospine}
for the types (Fin) and (Aff). 
 
In Section 3 we verify that $B_{\pi}$ is a  Cartan matrix and study the algebra $\fg(B_{\pi})$;
the results of this section can be easily deduced from~\cite{S}, but we give the proofs for the sake 
of completeness.

In Section 4 we prove~\Thm{thmiii} (iii). The proof is very similar to the proof of Theorem 8.3 in~\cite{S};
we present it for the sake of completness.

In Section 5 we study $Q^{\pm}(m,n,t)$,
which are the only superalgebras of type (Ind) that are not purely anisotropic.
For this case
we prove  Theorems~\ref{thmiii}, \ref{corintrospine}  and the formula
$(\Delta^{\ima})^+=Q^{++}_\mathbb{R}\cap Q_v$.

In Section 6 we complete the proof of~\Thm{corintrospine}
and of the formula~(\ref{suppconnected}).

In Section 7 we study the group $\Sp^D(v)$ in a more general context introduced in~\cite{GHS}.
We prove the existence of a group  homomorphism from
$\Sp^D(v)$ to the automorphisms of the Dynkin diagram of $\pi$. As an application, we obtain another proof of~\Thm{corintrospine} (iv). In~\ref{examples7} we consider the examples
$A(m|n)$ and  $A(m|n)^{(1)}$.

In Section 8 we study some additional examples. We briefly recall a desrciption of spines
for the type (Fin) and
describe $\Sp(v)$ and $\Sp^D(v)$ for the component  of type $C(n+1)^{(1)}$. We compare two algebras $G(3)^{(1)}$ and $G(3)^{(2)}$ and prove that they are isomorpic, so correcting the van de Leur's classification in~\cite{vdL}. In~\ref{infinitepiS} we consider an example of spine for a component
which is not Kac-Moody.

In Section 9 we study the group $\Sk^D(v)$ in the type (Aff) and prove the formula~(\ref{eqintroSkDv}).

\subsection{Acknowledgement}
We are grateful to Victor Kac for his guidance.
We would like to thank Vladimir Hinich and  Vera Serganova for their patience and valuable conversations.
We would like to thank Inna Entova-Aizenbud, Ian M.~Musson, Shifra Reif, Alex Sherman and
 Lior David Silberberg for their support and  numerous helpful discussions.
The research was 
supported by the Minerva foundation with funding from the Federal German Ministry for Education and Research.

\subsection{Index of definitions and notation} \label{sec:app-index}
Throughout the paper the ground field is $\mathbb{C}$; 
$\mathbb{N}$ stands 
 for the set of non-negative integers. 
 
\begin{center}
\begin{tabular}{lcl} 
$\Sk(v)$, $\Sp(v)$ & & \ref{spineandskeleton}\\
$\Sigma_u, \Delta^\re,\Delta_{\is},\Delta_{\an}, \Delta_{\nr}$ & & \ref{deltasdef}\\
 $\Sigma_{{pr}}$ &  & \ref{prindef}\\
$\Delta(\fg(v))$ & & \ref{decomposition}\\
 $\pi,\Delta,\Delta_{\ol{0}},\Delta_{\ol{1}}$ & &  \ref{pidef}\\
  $\Delta^{\ima}$ & & \ref{-Delta}\\
$\Delta^+_v, \Delta_{\an}^+,  (\Delta^{\ima})^+,\Delta_{\ol{0}}^+$ & & \ref{Delta+}\\
$W$ & & \ref{weylgroup}\\
$\pi_S$ & &  \ref{defnpiS}\\
$\Delta^+(\Sigma')$ & & \ref{rootbasis}\\
$\Delta_{>0}(h)$ & &  \ref{rootbases}\\
 ${\Sk}^D(v),{\Sp}^D(v)$ & & \ref{SpDvset}\\
\end{tabular}
\end{center}
  
\section{Preliminaries }
We work over the field \(\mathbb{C}\). We fix a finite set $X$.

\subsection{Main definitions}

Recall that a groupoid is a small category in which every arrow is invertible.
In this section we recall a construction of the root groupoid given in Section 2 of~\cite{GHS}.

\subsubsection{}
\begin{defn}{}\label{injectivemap}
Let \(V\) be a complex vector space. We call a map \(\phi:X\rightarrow V\) {\em injective}  if the induced linear map \(\phi:\operatorname{Span}_{\mathbb{C}}(X)\rightarrow V\) is injective.

For a finite set $X$, a root datum of $X$ is,
by definition, a quadruple \((\mathfrak{h},a,b,p)\) where \(\mathfrak{h}\) is a finite dimensional complex vector space, \(a:X\rightarrow\mathfrak{h},b:X\rightarrow\mathfrak{h}^*\) are injective maps and \(p:X\rightarrow\mathbb{Z}_2\).
\end{defn}
The cardinality of \(X\) is called the {\em rank}  of the root datum.

\subsubsection{}
\begin{defn}{} \label{rootdatum}
For a root datum 
\(v=(\mathfrak{h},a,b,p)\)  the {\em Cartan matrix} \(A^v\) is given by \(A^v=(a^v_{x,y})_{x,y\in X}\) where 
\(a^v_{x,y}:=\langle a(x),b(y)\rangle \).	The pair $(A^v,p^v)$ is called the {\em Cartan datum } at $v$.
  Two Cartan data \((A,p)\) and \((A',p')\) are called {\em \(D\)-equivalent} if  \(p=p'\)
and there exist an invertible diagonal matrix \(D\) such that \(A=DA'\).

For \(v=(\fh,a,b,p)\) a root datum of \(X\) and \(x\in X\) we call \(b(x)\) a \emph{simple root at \(v\)}.
\end{defn}

\subsubsection{}
\begin{defn}{}\label{refdef}
	We call \(x\in X\) {\em reflectable} at \((\mathfrak{h},a,b,p)\) if one of the following conditions holds:
	\begin{enumerate}
		\item \(a_{xx}=0\) and \(p(x)=1\).
		\item \(a_{xx}\neq0, \; p(x)=0\) and \(\frac{2a_{xy}}{a_{xx}} \in \ZZ_{\leq 0}\) for every \(y\neq x\in X\) .
		\item \(a_{xx}\neq0, \; p(x)=1\) and \(\frac{a_{xy}}{a_{xx}} \in \ZZ_{\leq 0}\) for every \(y\neq x\in X\) .
	\end{enumerate}

If \(x\in X\) is reflectable and \(a_{xx}\neq 0 \) we define 
\(r_x:(\mathfrak{h},a,b,p)\rightarrow(\mathfrak{h},a',b',p')\) by  \(p=p'\) and
 \[a'(y)\coloneqq a(y)-2\frac{a_{yx}}{a_{xx}}a(x),\;b'(y)\coloneqq b(y)-2\frac{a_{xy}}{a_{xx}}b(x);\] we call 
$r_x$ {\em anisotropic reflexion}.
	
	If \(x\in X\) is reflectable and \(a_{xx}= 0 \) we define
	 \(r_x:(\mathfrak{h},a,b,p)\rightarrow(\mathfrak{h},a',b',p')\) by \[(a'(y),b'(y),p'(y))\coloneqq \begin{cases}
		(-a(y),-b(y),p(y)) &\; \text{if } x=y\\
		(a(y),b(y),p(y)) &\; \text{if } x\neq y,\; a_{xy}=0\\
		(a(y)+\frac{a_{yx}}{a_{xy}}a(x),b(y)+b(x),p(y)+1) &\; \text{if } x\neq y,\; a_{xy}\neq 0;\\
	\end{cases}\] 
	we call 
	$r_x$ {\em isotropic reflexion}.
\end{defn}

\subsubsection{}\label{rootgroupoid}
Let us now define the root groupoid \(\cR\).\\
Fix a finite set \(X\). The objects of \(\cR\) are all root data of \(X\). 
The arrows of \(\cR\) will be defined by generators and relations. The generators are arrows of the following types:
	\begin{enumerate}
\item {\em Reflexions}. For  \(x\in X\) reflectable at \((\mathfrak{h},a,b,p)\)
 we have the arrow \(r_x:(\mathfrak{h},a,b,p)\rightarrow(\mathfrak{h},a',b',p')\)  defined in~\ref{refdef}).
\item {\em Tautological arrows}.  If \(\theta: \mathfrak{h} \rightarrow \mathfrak{h}'\) 
is an isomorphism of vector spaces, then \(\theta\) gives rise to the following arrow
 \(t_\theta:(\mathfrak{h},a,b,p)\rightarrow(\mathfrak{h}',a',b',p)\) 
where \(a' \coloneqq \theta\circ a\) and \(b'\coloneqq (\theta^*)^{-1}\circ b\).
\item {\em Homotheties}. Let \(\lambda: X \rightarrow \mathbb{C}^*\), \(\lambda\) 
gives rise to the following arrow \(h_\lambda:(\mathfrak{h},a,b,p)\rightarrow(\mathfrak{h},a',b,p)\) where 
\(a'(x)=\lambda(x)a(x)\).
	\end{enumerate}
We temporarily denote by \(\mathcal{F}\) the free category whose objects are the objects of 
\(\cR\) and whose arrows are generated by the three types of arrows described above. 
	
Let \(\mathfrak{h}:\mathcal{F}\rightarrow \mathsf{Vect}\) be the forgetful functor from
 \(\mathcal{F}\) to the category of finite dimensional vector spaces over  $\mathbb{C}$,
 sending \((\mathfrak{h},a,b,p)\) to \(\mathfrak{h}\). Notice that this functor	
 sends reflexions and homothety arrows to the 
 identity map and tautological arrows  to the isomorphism that is used to define them. 
 We now define an equivalence relation on arrows of \(\mathcal{F}\):
 two arrows from \(v=(\mathfrak{h},a,b,p)\) to \(v'=(\mathfrak{h}',a',b',p')\) are equivalent if they induce the same isomorphism from \(\mathfrak{h}\) to \(\mathfrak{h}'\).
	
The arrows of \(\cR\) are defined as the equivalence class of the arrows of the category \(\mathcal{F}\).
 By~Proposition 2.3 of \cite{GHS},  the category \(\cR\) is  a groupoid.
%

\subsubsection{Skeleton}\label{skeleton}
\begin{defn}{}
 Let \(\operatorname{Sk}\subset \cR\) be the subgroupoid whose 
 objects are the same objects as the objects of \(\cR\) and an arrow \(\gamma:v\rightarrow v'\) in \(\cR\) will be in the skeleton if \(\gamma\) can be presented as a composition of reflexions.  
 \end{defn}

 Let $V_a(v)\subset \fh(v)$ and $V_b(v)\subset\fh^*(v)$
 be the spans of $\{a_v(x)\}_{x\in X}$ and $\{b_v(x)\}_{x\in X}$ respectively.  
 Since the elements of $\{a_v(x)\}_{x\in X}$(resp., of $\{b_v(x)\}_{x\in X}$) are linearly independent,
 one has  $\dim V_a(v)=\dim V_b(v)=|X|$.
 If $u\in\Sk(v)$, then $\fh(v)=\fh(u)$, $V_a(u)=V_a(v)$ and $V_b(u)=V_b(v)$.
 In particular, $\corank A^v=\corank A^u$ (this implies $\corank A^v=\corank A^{v'}$
 for any $v'\in\cR_0$). We usually fix a connected component of $\Sk$ and
 omit $v$ in the notation (writing $\fh$, $V_a$, $V_b$).

 \subsubsection{Spine}\label{spine}
\begin{defn}{}
 Let \(\operatorname{Sp} \subset \cR\) be the subgroupoid whose 
 objects are the same objects as the objects of \(\cR\) and an arrow
  \(\gamma:v\rightarrow v'\) in \(\cR\) will be in the spine if \(\gamma\) can be 
  presented as a composition of isotropic reflexions. 
\end{defn} 
	\subsubsection{}
\label{spineandskeleton}
 For \(v\in\cR\) we denote by $\Sk(v)$, $\Sp(v)$ the connected component of 
\(\operatorname{Sk}\), $\Sp$ which contains \(v\). We will call $\Sk(v)$, $\Sp(v)$ 
the skeleton and the spine  at \(v\) or only the skeleton, 
the spine when it is clear which \(v\) we are talking about. 

\subsubsection{Remark}\label{groupiediscontractible}
It is important to notice that $\Sk(v)$, $\Sp(v)$  are contractible groupoids, 
since by definition of reflexion every arrow in $\Sk(v)$ induces the identity map on $\fh$ 
and thus any two arrows with the same beginning and end coincide.

\subsubsection{}
\begin{defn}{}
A connected component \(\cR_0\) of \(\cR\) is called
{\em fully reflectable} if all \(x\in X\) are reflectable at all \(v \in \cR_0\).
\end{defn}

\subsubsection{Remark} \label{dequiv reflectable}
Let \(v_1,v_2 \in \cR\) have \(D\)-equivalent Cartan data. 
If \(v_1\) is reflectable at \(x\) then \(v_2\) is also reflectable at \(x\). 
Moreover, if \(v_1 \xrightarrow{r_x}v_1'\) and \(v_2 \xrightarrow{r_x}v_2'\) 
then the Cartan data of \(v_1',v_2'\) are also \(D\)-equivalent.

\subsubsection{}
\begin{defn}{} 
A matrix \(A\) is called {\em symmetrizable} if there exists an invertible
 diagonal matrix \(D\) and symmetric matrix \(B\) such that \(A=DB\).
\end{defn}

\subsubsection{}
The anisotropic reflexions and tautological arrows do not change the Cartan datum. The homotheties
do not change $p$ and multiply the Cartan matrix on an invertible diagonal matrix.
Isotropic reflexions preserve the symmetricity of a Cartan matrix.

Thus, if for a connected component $\cR_0$ 
 there exists \(v \in \cR_0\) with a  symmetrizable Cartan matrix \(A_v\),  then for every \(u \in \cR_0\)
 the Cartan matrix  \(A_u\) is symmetrizable. Such  connected component  is called {\em symmetrizable}.
 
  \subsubsection{}\label{bilin}
  Observe that the formulae in~\ref{refdef}
imply that for any $u\in\Sk(v)$ we have $a_u(x)=\sum k_y a_v(y)$ 
and $b_u(x)=\sum k'_y b_v(y)$  for some $k_y, k'_y\in\mathbb{C}$.  If $A_v$ is symmetric, then
$k_y=k'_y$.

 \subsection{Root (super)algebras} \label{rootalgebras}
The root groupoid $\cR$ consists  of connected 
components some of which will lead to  
``root superalgebras'' introduced in Section 3 of~\cite{GHS}.  Root superalgebras
are Lie superalgebras constructed in the following manner.
For each $v\in\cR$ one defines a ``half-baked'' Lie superalgebra $\wt\fg(v)$ 
 in a classical way, see~\cite{Kbook}, Chapter I:
this  is a Lie superalgebra generated by  $\fh$
and the  elements $\tilde{e}_{\pm b(x)}$, $x\in X$ with the parity
$p(\tilde{e}_{\pm b(x)}):=p(x)$
  the relations are
\begin{equation}\label{relationshalfbaked}
[\tilde{e}_{b(x)},\tilde{e}_{-b(y)}]=\delta_{x,y} a(x), \ \ \ [h,\tilde{e}_{\pm b(x)}]=\pm 
\langle h, b(x) \rangle\ \ \text{ for all }x,y\in X,\ h \in\fh. \end{equation}

 For an arrow $\gamma:v\to v'$ in $\cR$ the 
isomorphism
$\fh(\gamma):\fh(v)\to \fh(v')$ does not always extend to a homomorphism of  $\wt\fg(v)\to \wt\fg(v')$.
A root algebra supported on a component $\cR_0$ of $\cR$ is a collection of 
quotients $\fg(v)$ of $\wt\fg(v)$ such that for any $\gamma:v\to v'$ the isomorphism 
$\fh(\gamma)$ extends to an isomorphism $\fg(v)\to\fg(v')$.

A connected component $\cR_0$ of $\cR$ is called {\em  admissible} if it admits a root algebra. 
In this paper we consider only admissible components. By Theorem 3.3.1 in~\cite{GHS}, 
$\cR_0$, the connected component is admissible if and only if for each $v\in \cR_0$ the Cartan datum
$(A_v, p_v)$ is {\em weakly symmetrizable} which means 
$$a^v_{xy}=0\ \ \Longrightarrow\ \ a^v_{yx}=0\ \text{ for all $x$ reflectable at $v$.}$$

For an admissible component $\cR_0$ the  category of root algebras
has an initial object, a  root algebra $\fg^{\U}$ which is called {\em universal},
and a final object, a root algebra $\fg^{\ttC}$ is called {\em contragredient}.  
The algebra $\fg^{\ttC}$  is 
the quotient of $\wt\fg(v)$ by the maximal
ideal having zero intersection with $\fh$. 
The universal root algebra $\fg^{\U}$ is obtained by imposing on $\wt\fg(v)$ reflected Chevalley relations --- so generalizing the classical Serre relations.

\subsubsection{}\label{sllocnil}
Let $\fg(v)$ be a root algebra. 
Assume that $x\in X$ is reflectable at $v$ and $a_{xx}\not=0$. Set $\alpha:=b_v(x)$.
The root spaces $\fg(v)_{\pm \alpha}$ act locally nilpotently in the adjoint representation and
generate a subalgebra which is isomorphic to $\fsl_2$
if  $p(x)=\ol{0}$, and  to $\mathfrak{osp}(1|2)$ if 
 $p(x)=\ol{1}$. 
Note that  $[\fg_{\alpha},\fg_{-\alpha}]$ is one-dimensional. We denote by $\alpha^{\vee}$ a unique
element in $[\fg_{\alpha},\fg_{-\alpha}]$ satisfying $\langle \alpha^{\vee},\alpha\rangle=2$.

\subsubsection{}\label{122}
If two vertices $v,v'$ have $D$-equivalent Cartan data
and $\dim\fh=\dim\fh'$, then there exists a homothety
$h_{\lambda}: (\fh,a,b,p)\to (\fh',a',b',p')$.
By the definition of a root algebra, this homothety
(as all other generating arrows) can be extended to isomorphisms
$\fg^{\U}\iso (\fg')^{\U}$ and $\fg^{\ttC}\iso (\fg')^{\ttC}$.

\subsection{Lie algebra case}\label{Liealgcase}
Consider the case when $p\equiv \ol{0}$ for some $v\in\cR_0$.  Then 
$p'\equiv \ol{0}$  and  $(A^v,p)$, $(A^{v'},p_{v'})$ 
are $D$-equivalent for any $v'\in\cR_0$.
 Applying a suitable homothety we can 
choose $v\in\cR_0$ such that $a_{xx}\in\{0,2\}$ for all $x\in X$. 
Note that $x$ is reflectable at $v$
(and at any $v'\in\cR_0$) if and only if $a_{xx}\not=0$ and 
$\frac{2a_{xy}}{a_{xx}}\in\mathbb{Z}_{\leq 0}$
for all $y\in X$.

Let $A:=A^v$ be the Cartan matrix at $v$.
If $\cR_0$ is admissible, then  corresponding root superalgebras are  Lie algebras.
In the fully reflectable case ($a_{xx}\not=0$, 
$\frac{2a_{xy}}{a_{xx}}\in\mathbb{Z}_{\leq 0}$
for all $x,y$) the admissibility is equivalent to 
\begin{equation}\label{C2}
a_{xy}=0\ \Longrightarrow a_{yx}=0\end{equation}
for all $x,y\in X$.

In the fully reflectable admissible case, $\fg^{\ttC}$ is the direct product of the Kac-Moody algebra with the Cartan matrix $A$ and a commutative algebra
of the dimension $\dim\fh-|X|-\corank A$.
If $A$ is symmetrizable, then $\fg^{\U}=\fg^{\ttC}$, see~\cite{GabberKac}.

If the Cartan matrix $A$ is real symmetrizable
  and satisfies additional
conditions (see~\cite{Wak}, 2.2), then $\fg^{\ttC}=\fg^{\U}$ is a Borcherds-Kac-Moody algebra. For example, if $A=(0)$ and $\dim\fh=2$, then  $\fg^{\ttC}=\fg^{\U}$
is the product of the Heisenberg algebra and a one-dimensional Lie algebra.

\subsection{Roots}\label{roots}
From now on we fix an admissible component $\Sk(v)$ of $\Sk$; we
write $\fh:=\fh(v)$ and so on.

\subsubsection{}
\begin{defn}{defsimpleroots}
\begin{enumerate}
		\item A simple root $b(x)\in\fh^*$ will be called {\em isotropic} if $x$ is reflectable at $v$ and $\langle a(x),b(x)\rangle=0$.
		\item A simple root $b(x)\in\fh^*$ will be called {\em anisotropic} 
		if $x$ is reflectable at $v$ and $\langle a(x),b(x)\rangle\ne 0$. 
\end{enumerate}
\end{defn}

\subsubsection{}
\begin{defn}{deltasdef} An element $\alpha\in \fh^*$ is called a {\em real root} if there exists
	$v'\in \Sk(v)$
	and $x\in X$ so that $\alpha=b_{v'}(x)$.
\end{defn}

\subsubsection{}	
We denote by $\Sigma_v$ the set of simple roots at \(v\), i.e.
$\Sigma_v:=\{b(x)\}_{x\in X}$.

By~\cite{GHS}, Prop. 4.2.2 and Corollary 5.3.10, Definition~\ref{defsimpleroots} can be extended to 
the real roots. Moreover we have
\begin{equation}\label{eq:reunion}
\Delta^\re=\Delta_{\is}\sqcup\Delta_{\an}
\sqcup\Delta_{\nr}, \ \ \ \ \ \Delta_{\is}\sqcup\Delta_{\an}=\Delta^\re\cap (-\Delta^\re)
\end{equation}
where
\begin{itemize}
\item[] $\Delta_{\is}$ is the set of {\em isotropic real } roots that 
are reflectable simple isotropic roots at some $v'\in\Sk(v)$.
\item[] $\Delta_{\an}$ is the set of {\em anisotropic real} roots that
are reflectable anisotorpic simple roots at some $v'\in\Sk(v)$. Any anisotropic
real root $\alpha\in\Delta_{\an}$ defines a coroot 
$\alpha^\vee\in\fh$.
\item[] $\Delta_{\nr}$ is the set of {\em non-reflectable real}  roots
that are non-reflectable simple roots at some $v'\in\Sk(v)$.
\end{itemize}

By~\ref{sllocnil}, for any $\alpha\in\Delta_{\an}$ 
the root spaces $\fg_{\pm \alpha}$ generate a subalgebra $\fg(\alpha)$ which is isomorphic to $\fsl_2$
if  $p(\alpha)=\ol{0}$, and  to $\mathfrak{osp}(1|2)$ if 
 $p(\alpha)=\ol{1}$. One has $\fg(\alpha)\cap \fh=\mathbb{C}\alpha^{\vee}$
where $\alpha^{\vee}$ is normalized by the condition 
$\langle \alpha^{\vee},\alpha\rangle=2$.

\subsubsection{Remark}
In the above definition $2\alpha\not\in
 \Delta^\re$ if $\alpha\in\Delta^\re$ (even if $\alpha$ is odd and non-isotropic).
 Moreover, 
isotropic roots are necessarily real.
In another tradition, a root of a symmetrizable Lie superalgebra is
called isotropic if it has length zero. For the real roots both notions 
of isotropicity coincide. 

\subsubsection{}
\begin{defn}{} \label{prindef}	An element $\alpha\in \Delta_{\an}$ is called a {\em principal root} if there exists
	$v'\in \Sp(v)$
	and $x\in X$ so that $\alpha=b'(x)$\footnote{We use the definition in~\cite{GHS}; in~\cite{S}
principal roots are defined differently.}. We denote the set of principal roots
	by $\Sigma_{\pr}$. 
\end{defn}

\subsubsection{} \label{decomposition}
	We set
	$$Q_v:=\sum\limits_{x\in X}\mathbb{Z} b(x)\subset\fh^*,\ \ 
	Q^+_v:=\sum\limits_{x\in X}\mathbb{Z}_{\geq 0} b(x)\subset\fh^*.$$
		The parity function $p:X\to\mathbb{Z}_2$ extends to a group homomorphism
	$p:Q_v\to\mathbb{Z}_2$ that we denote by the same letter $p$. 
	
Any root algebra $\fg(v)$ admits a root decomposition
	$$\fg(v)=\fh(v)\oplus\bigoplus_{\mu\in\Delta(\fg(v))}\fg(v)_\mu$$
	where $\Delta(\fg(v)):=\{\mu\in\fh^*|\ \fg(v)_{\mu}\not=0,\ \ \mu\not=0\}$.
	By the defining relation of the root algebra one can see that 
	$\Delta(\fg(v))\subset Q^+_v\coprod (-Q^+_v)$. 
	
	\subsubsection{} \label{pidef}
	We set 
	$\Delta:=\Delta(\fg^{\ttC})$ and $\Delta_{\ol{i}}:=\{\alpha\in \Delta(\fg^{\ttC})|\ p(\alpha)=\ol{i}\}$.
	The following formulae easily follows from definitions:  
	$\Delta\subset \Delta(\fg)$, 
	$\Delta^\re\subset\Delta$, $\fg_{\mu}\subset \fg_{p(\mu)}$ and

	\begin{equation}\label{2alpha}\begin{array}{ll}
\dim \fg_{j\alpha}\leq 1\ & \text{ for } \alpha \in\Delta^{\re}, j\in\mathbb{Z},\\
\mathbb{Z}\alpha\cap \Delta(\fg)=\{\pm\alpha\}\ & \text{ for } 	\alpha\in\Delta^{\re}_{\ol{0}},\\
\mathbb{Z}\alpha\cap \Delta(\fg)=\{\pm\alpha, \pm 2\alpha\}\ & \text{ for } \alpha\in\Delta^{\re}_{\ol{1}}\setminus \Delta_{\is},\\
\mathbb{Z}\alpha\cap \Delta=\{\pm\alpha\}\ & \text{ for } 	\alpha\in\Delta_{\is}.
\end{array}\end{equation}

We set
$$\pi:=\{\alpha\in \Sigma_{\pr}|\ p(\alpha)=\ol{0}\}\cup \{2\alpha|\ \alpha\in \Sigma_{\pr}|\ p(\alpha)=\ol{1}\}$$
and define $\alpha^{\vee}$ for each $\alpha\in \pi$ by
setting $(2\alpha)^{\vee}:=\frac{\alpha}{2}$ for each odd principal root $\alpha$.

By~\ref{sllocnil},
\begin{equation}\label{betaalphavee}
\langle \beta,\alpha^{\vee}\rangle \in\mathbb{Z}\ \ \text{ for all }\beta\in Q_v,\  \alpha\in\pi.
\end{equation}

\subsubsection{}\label{2bx}
By~\cite{GHS}, Corollary 3.2.8 (4) one has 
$$\mathbb{C} b_v(x)\cap \Delta(\fg)=\{\pm b_v(x)\}\ \ \text{ if }
a_{xx}=0\ \text{ and }a_{xy} \not=0\ \text{ for some }y\in X.$$

\subsubsection{}\label{-Delta}
By~\cite{GHS}, 3.1.4  $\Delta=-\Delta$. A root $\beta\in\Delta$ is called {\em imaginary}
	if $\beta\not\in\Delta^{\re}$ and $\frac{\beta}{2}\not\in\Delta^{\an}$. We denote the set of imaginary roots
	by $\Delta^{\ima}$.

\subsubsection{}\label{Delta+}
	We set
$\Delta^+_v:=\Delta\cap Q^+_v$, $\Delta_{\an}^+:=\Delta_{\an}\cap Q^+_v$,  $(\Delta^{\ima})^+:=\Delta^{\ima}\cap Q^+_v$, $\Delta^+_{\overline{i}}:=\Delta_{\overline{i}}\cap Q^+_v$.

Take a reflexion $r_x: v\to u$ and set $\alpha:=b_v(x)$. Then $-\alpha=b_u(x)$ and
\begin{equation}\label{Delta+uv}
\Delta^+_u=\left\{\begin{array}{ll}(\Delta^+_v\setminus \{\alpha\})\cup\{-\alpha\}\ \text{ if
} p(\alpha)=\ol{0}\ \text{ or $\alpha$ is isotropic} ,\\
(\Delta^+_v\setminus \{\alpha, 2\alpha\})\cup\{-\alpha,-2\alpha\}\ \text{ otherwise}.
\end{array}\right.\end{equation}

In particular, the  sets $\Delta^+_{\an}$, $\Sigma_{\pr}$ and $\Delta^+_{\ol{0}}$ are independent of the choice of $v$
in the same component of $\Sp$ (since any isotropic reflection does not change the set
of positive roots which are not isotropic). One has $\pi\subset  \Delta^+_{\ol{0}}$ and
	$\Sigma_{\pr}\subset \Delta_{\an}^+$.
	Moreover, any reflection does not change the set of positive imaginary roots, so
	 $(\Delta^{\ima})^+$  are independent of the choice of $v$ in the same component of $\Sk$.

\subsection{Weyl group}\label{Weylgr}
The notion of the Weyl group can be naturally extended to the case of admissible 
components in the following way.

For $\alpha\in\Delta_{\an}$  we define a reflection $s_\alpha$
	acting on $\fh(v)$ and on $\fh^*(v)$ by the formulae
	\begin{equation*}
		\label{eq:salpha}
		s_\alpha(\beta)=\beta-\langle \beta,\alpha^\vee\rangle\alpha,\
		s_\alpha(h)=h-\langle \alpha,h\rangle\alpha^\vee.
	\end{equation*}

\subsubsection{}
\begin{defn}{}
	\label{weylgroup}  
	The Weyl group $W$  is the subgroup of 
	 $\GL(\fh)=\GL(\fh^*)$ generated by the reflections with respect to $\Delta_{\an}$.
\end{defn}

\subsubsection{}\label{l(w)}
The  Weyl group acts on $\Sk(v)$ mapping a vertex $u$ to the vertex $wu$ such that 
$b_{wu}(x)=w b_u(x)$, $a_{wu}(x)=wa_u(x)$ and $p_{wv}(x)=p_v(x)$ for all $x\in X$
(see~\cite{GHS}, Proposition 4.3.6).   
Moreover, this action is faithful and 
for any $u\in \Sk(v)$ there exists a
unique $w\in W$ such that $wu\in \Sp(v)$ (see~\cite{GHS}, Section 5.2).
Note that $A^{wv}=A^v$, so
the Cartan data at $wv$ and at $v$ are equal.

By~\cite{GHS}, 4.3.12, $W$ is the Coxeter group with the generators 
$\{s_{\alpha}\}_{\alpha\in\pi}=
\{s_{\alpha}\}_{\alpha\in\Sigma_{\pr}}$. 
The length of $w$ (as an element of the Coxeter group) equals to 
 the cardinality of the set $\Delta^+_{\an}(u)\setminus \Delta^+_{\an}(v)$.

\subsubsection{}\label{111}
By~\cite{GHS}, 4.3.13, the group $W$ acts faithfully on $\Delta_{\an}$ and 
any element in $\Delta_{\an}$ is $W$-conjugated to a principal root, that is
 $$\Delta_{\an}=W \Sigma_{\pr}.$$

By~\ref{l(w)} the skeleton $\Sk(v)$ contains a vertex $wv$ with $b_{wu}(x)=w b_u(x)$  for all $x\in X$.
Using~\ref{Delta+} we conclude that
$(\Delta^{\ima})^+$ is $W$-invariant.

\subsection{Remark}\label{Dequivalence}
Let $u,v\in\cR$ be two vertices with $D$-equivalent Cartan data, i.e. $p_u=p_v$ and  $A_u=DA_v$
for an invertible diagonal matrix $D$. In this case
we have a natural bijection $\sigma$ between the corresponding connected components 
which maps $v$ to $u$ and the arrows to their ``namesakes'': this means that a generating arrow
$\gamma: v_1\to v_2$ maps to a generating arrow $\gamma': u_1\to u_2$, 
where $\gamma'=r_x$ if $\gamma=r_x$, $\gamma'=h_{\lambda}$ if $\gamma=h_{\lambda}$
and $\gamma'$ is tautological if $\gamma$ is tautological. The Cartan data at $v$ and at  $\sigma(v)$ are $D$-equivalent (with the same matrix $D$). The map
$\sigma$ maps $\Sk(v)$ to $\Sk(u)$ and $\Sp(v)$ to $\Sp(u)$, see~\ref{SpDvset}
for details.

The map $b_v(x)\mapsto b_u(x)$
extends to a linear isomorphism $\sigma_Q: Q_v\iso Q_u$ which maps
 $\Delta_{\an}(v)$ to $\Delta_{\an}(u)$, $\Delta_{\an}^+(v)$ to $\Delta_{\an}^+(u)$
and $\Sigma_{\pr}(v)$ to $\Sigma_{\pr}(u)$. 

This defines an isomorphism of the commutators of half-baked Lie superalgebras 
$$[\wt\fg(v),\wt\fg(v)]\to[\wt\fg(u),\wt\fg(u)]$$ that induces an isomorphism 
$[\fg^{\U}(v),\fg^{\U}(v)]\to[\fg^{\U}(u),\fg^{\U}(u)]$.
Factoring both sides by the maximal graded ideal having zero intersection with the Cartan, 
we get an isomorphism $[\fg^{\ttC}(v),\fg^{\ttC}(v)]\to[\fg^{\ttC}(u),\fg^{\ttC}(u)]$ of the commutators of the corresponding contragredient Lie
superalgebras. In particular, $\sigma_Q(\Delta)=\Delta$.

Taking these isomorphisms into account, we will always assume that $\fh$ is chosen to have
a minimal possible dimension that allows to realize a given Cartan matrix. By~\cite{Kbook} 
$\dim\fh=|X|+\corank(A)$.

\subsection{Indecomposable components}\label{indecomosable}
We call a Cartan matrix $A$ {\em indecomposable} if $A=(0)$ or
the elements of $X$ can be ordered in such a way
that $X=\{x_1,\ldots,x_s\}$ and $a_{x_j x_{j+1}}\not=0$ for $j=1,\ldots, s$  
where $x_{s+1}:=x_1$.
We call an admissible connected component $\Sk(v)$   {\em indecomposable} if all Cartan matrices $A^u$ are indecomposable and call a component {\em decomposable} otherwise.   Note that  
$\Sk(v)$ is indecomposable if and only if $A^u$ is indecomposable for all 
$u\in\Sp(v)$.

A contragredient superalgebra $\fg^{\ttC}$ is called {\em quasisimple} if for any ideal $\mathfrak{i}\subset \fg^{\ttC}$
either $\mathfrak{i}+\fh=\fg^{\ttC}$ or $\mathfrak{i}\subset \fh$ (in the latter case $\mathfrak{i}$ lies in the centre
of $\fg^{\ttC}$). It is easy to show that the superalgebra $\fg^{\ttC}$ is simple if and only if it is quasisimple with   $\dim\fh=|X|$ (this implies $\det A\not=0$). 
 One readily sees that $\fg^{\ttC}$ corresponding to an admissible Cartan datum 
is quasisimple if $\Sk(v)$ is an indecomposable component.

\subsection{Kac-Moody components}\label{KMcomponents}
A fully reflectable component with $\dim\fh=|X|+\corank(A)$ is called a {\em Kac-Moody component}.
In this case $\fg^{\ttC}$ is a Kac-Moody superalgebra.In the terminology of~\cite{S} 
such superalgebras are called ``quasisimple regular Kac-Moody superalgebras''; 
``Kac-Moody superalgebras'' in~\cite{S} include even and Heisenberg superalgebras.

Since a Kac-Moody component is fully reflectable, the set $\Delta_{\nr}$ is empty, so 
$$\Delta^{\re}=-\Delta^{\re}=\Delta_{\an}\cup\Delta_{\is},\ \ \ 
\Delta=\Delta_{\an}\cup\Delta_{\is}\cup\Delta^{\ima}\cup \{2\alpha|\ \alpha\in \Delta_{\an}\cap \Delta_{\ol{1}}\}.$$

By~\cite{GHS}, 3.2.7 any root algebra for a decomposable Kac-Moody component
is a product of root algebras for the corresponding indecomposable Kac-Moody components.
Using the standard argument (see~\cite{Kbook}, Lemma 1.6) one deduces that $\supp\beta$ is connected for any $\beta\in\Delta$
(see~\ref{introimaginary} for notation).

 \subsubsection{} 
 \Thm{thmiii} is a slight modification of Theorem 8.3 in~\cite{S}  which implies 
 the following important corollary (see 
Theorem 4.14 in~\cite{S}):

  \begin{cor}{corbase}
  Let $(A,p)$ and $(A',p')$ be two Cartan data 
  and $\fg$, $\fg'$ be the corresponding Kac-Moody superalgebras
  with the Cartan subalgebras $\fh$ and $\fh'$. If
   $\fg\iso\fg'$ is an algebra isomorphism which maps $\fh$ to $\fh'$, then  
$(A',p')$ can be obtained from $(A,p)$ by a sequence of isotropic reflections,
a multiplication of $A$ to an invertible diagonal matrix and   
a permutation  of indices (=the elements of $X$).\end{cor}

\subsubsection{}
We will use the classification of Kac-Moody component described in~\ref{classif}. 
We give some additional details below.

\subsubsection{Purely anisotropic case}\label{Deltareisoempty} 
We say that $\Sk(v)$ is {\em purely anisotropic} if $a_{xx}\not=0$ for all $x$.
In this case all reflections are anisotropic,
the Cartan data is the same in all vertices of the skeleton and the spine is trivial.
In \cite{GHS} it is shown that in this case, the classes (Fin), (Aff), and (Ind) correspond to (FIN), (AFF), and (IND) as defined in \cite{Kbook}, Theorem 4.3. 

It is not hard to check that many proofs in~\cite{Kbook} are valid in purely anisotropic case.  
In particular,  all propositions, theorems and corollaries in Chapters II--V hold 
in this case (some lemmas should be modified: for example, in Lemma 3.7,
if  $\alpha_i$ is a simple root and $\alpha\in\Delta^+$ is such that
$s_{\alpha_i}\alpha\not\in\Delta^+$ , then $\alpha=\alpha_i$
if $\alpha_i$ is even and $\alpha\in\{\alpha_i,2\alpha_i\}$ if $\alpha_i$ is odd).
Main results of Chapters IX and X (including the classification
of integrable highest weight modules and the Weyl-Kac  formula
for their characters) are also valid.

\subsubsection{Type (Fin)}
The corresponding Lie superalgebras were classified in~\cite{Ksuper}.
Except for the case $A(0|0)$ with $\fg^{\ttC}=\fgl(1|1)\not=\fg^{\U}$, one has
$\fg^{\ttC}=\fg^{\U}$ and this  superalgebra is either a basic classical Lie superalgebra which differs from
$\mathfrak{psl}(n|n)$ 
or $\fgl(n|n)$.

\subsubsection{Type (Aff)}
By~\cite{vdL}, for any indecomposable symmetrizable Kac-Moody component of type (Aff)
the corresponding Kac-Moody  superalgebra $\fg^{\ttC}$ is  the affinization or a twisted affinization
of  a  basic classical Lie superalgebra. By~\cite{GHS}, one has $\fg^{\ttC}=\fg^{\U}$ except for the 
case when  $\fg^{\ttC}$ is  the affinization or a twisted affinization of $\mathfrak{psl}(n|n)$.

By~\cite{Hoyt}  the non-symmetrizable components of type (Aff)
are $S(1|2,b)$ (for $b\in\mathbb{C}\setminus\mathbb{Z}$)  and $q_n^{(2)}$ (for $n\geq 3$).
The component $S(1|2,b)$ has rank $3$; 
the corresponding Kac-Moody superalgebra $\fg^{\ttC}=\fg^{\U}$ 
can be seen as a deformation of $\fsl(1|2)^{(1)}$. The component $q_n^{(2)}$ has rank $n$; 
the corresponding Kac-Moody superalgebra is the twisted affinization
of a simple strange superalgebra $\mathfrak{psq}_n$ (while the universal superalgebra is the twisted affinization of the  strange superalgebra $\mathfrak{q}_n$, see~\cite{GHS}).
We will use the following property of the components of types (Aff):

\begin{equation}
\label{affmodulodelta}\begin{array}{c}
(\Delta^{\ima})^+=\mathbb{Z}_{>0}\delta\ \ \text{and 
 the image of $\Delta$ in $\fh^*/\mathbb{C}\delta$ is finite.}
\end{array}\end{equation}
For all cases except $S(1|2;b)$ the above property follows from the fact that
$\fg^{\ttC}$ can be realized as the affinization or a twisted affinization
of  a  finite-dimensional Lie superalgebra.
For $S(1|2,b)$ this follows from the description in~\cite{S}.

\subsubsection{Type (Ind)}
The components of type (Ind) which  are  not purely anisotropic were described in~\cite{Hoyt};
these components are $Q^{\pm}(m,n,t)$ where $m,n, p$ are positive integers with $mnp>1$.
 The purely anisotropic components of type (Ind) correspond to the pairs
 $(A,p)$ where $A=(a_{ij})_{i,j=1}^n$ is the usual Cartan matrix of type (Ind) 
 (with $a_{ii}=2$) 
 and  $p:\{1,2,\ldots,n\}\to\mathbb{Z}_2$ is such that $p(i)=\ol{1}$ implies
 $a_{ij}\in 2\mathbb{Z}_{\leq 0}$. It is not known whether 
 $\mathfrak{g}^{\ttC}=\mathfrak{g}^{U}$.

\subsection{Symmetrizable case}\label{bilinearform}
Retain notation of~\ref{skeleton}.
In the symmetrizable case we choose the spine in such a way that
the Cartan matrices are symmetric (this property is independent
of the choice of
$v$ in the  spine) and we fix a bilinear form
on $V_b$ by setting $(b_v(x),b_v(y)):=\langle a_v(x), b_v(y)\rangle$.
The observation~\ref{bilin} implies that 
this form is independent on the choice of $v$ in the skeleton and
that for any vertex $u\in \Sk(v)$  the Cartan matrix 
$A_u$ coincides with the Gram matrix of $\{b_u(x)\}$ (i.e.,
$\langle a_u(x),b_u(y)\rangle=(b_u(x),b_u(y))$).  
This form is $W$-invariant, since for $w\in W$ one has
$$(wb_v(x),wb_v(y))=(b_{wv}(x),b_{wv}(y))=\langle a_{wv}(x),b_{wv}(y)\rangle=
\langle a_v(x),b_v(y)\rangle=(b_v(x),b_v(y)).$$

Note that this form is nondegenerate if and only if $\det A_v\not=0$.
Using the argument in~\cite{Kbook}, Ch. II this form can be extended
(not uniquely) 
 to a nondegenerate $W$-invariant bilinear form on $\fh^*$. 
 For $\mu\in \fh^*$ satisfying $(\mu,\mu)\not=0$ we define the reflection $s_{\mu}\in GL(\fh^*)$  by the usual formula $s_{\mu}(\nu):=\nu-\frac{2(\nu,\mu)}{(\mu,\mu)}\mu$; for $\alpha\in\Delta_{\an}$
 the reflection $s_{\alpha}$  coincides with the reflection defined in~\ref{eq:salpha}.

\subsection{Notation}\label{VaVb}
 Recall that  $\Sigma_{v}:=\{b_v(x)\}_{x\in X}$ denotes the set of simple roots at $v$.
We will always fix an admissible component of $\Sp$ (i.e., a connected 
component of spine which lies in an admiisible component of $\cR_0$). 
From now on ``the spine'' will refer to this component and ``the 
skeleton'' to the corresponding component in $\Sk$.
We will use notation  $W$, $\Delta^+_{\an}$, $\Sigma_{\pr}$
and  $(\Delta^{\ima})^+$ since these objects are independent of the choice of
$v$ in the same component of $\Sp$ (see~\ref{Delta+}).
We denote by $V_{\mathbb{R}}$ the $\mathbb{R}$-span of $\Sigma_v$.
By~(\ref{betaalphavee}), $\langle \mu,\alpha^{\vee}\rangle\in\mathbb{R}$
for all $\mu\in V_{\mathbb{R}}$ and $\alpha\in\pi$.


\section{Bases of $\Delta$}
Fix an admissible component $\Sk(v)$ of $\Sk$.

\subsection{The sets $\Sigma_u$}
Recall that $\Sk(v)$ 
contains vertices which are obtained from $v$ by a chain of reflections.
For each reflection $r_x: v\to v'$ the set
$\Delta^+_{v}(\fg)\setminus \Delta^+_{v'}(\fg)$ equals to $\{\alpha\}$ if $\alpha$ is even, 
to $\{\alpha,2\alpha\}$ if $\alpha$ is odd and non-isotropic and to one 
of these sets if $\alpha$ is odd isotropic (one has
$\Delta^+_{v}\setminus \Delta^+_{v'}=\{\alpha\}$  if $\alpha$ is odd isotropic).
In particular, for $v'\in\Sk(v)$ the set $\Delta^+_{v}(\fg)\setminus \Delta^+_{v'}(\fg)$
is finite.

\subsubsection{}\label{attain}
By~\cite{GHS}, Corollary 5.3.7 for 
$u\in\Sk(v)$ one has
\begin{equation}\label{Sigmau=Sigmav}
\Sigma_u=\Sigma_v \ \Longrightarrow\ u=v.
\end{equation}
 Following~\cite{GHS}, 4.3.1
we call the sets $\Sigma_u$ for $u\in\Sk(v)$  {\em 
 attainable sets of simple roots} for $v$.   By above, $\Sk(v)$ classifies the attainable sets of simple roots.

\subsubsection{}\label{spineDelta0}
Recall that $\Sp(v)$ 
contains vertices which are obtained from $v$ by a chain of isotropic reflections. 

By~\ref{Delta+}, for $u\in\Sp(v)$ we have 
$\Delta_{\an}\cap \Delta^+_u=\Delta^+_{\an}$ and 
$\Delta_{\ol{0}}\cap \Delta^+_u=\Delta^+_{\ol{0}}$.
We claim that
$$\Sp(v)=
\{u\in \Sk(v)|\ \pi\subset \Delta^+_u\}.$$
Indeed, the inclusion $\subset$ holds since $\pi\subset \Delta_{\ol{0}}^+$.
For the inverse inclusion assume that $\pi\subset \Delta^+_{u}$ for some
$u\in \Sk(v)$. By above, $u=wu'$ for some $u'\in\Sp(v)$ and $w\in W$.
Then $\pi\subset  w\Delta^+_{u'}$, so $w^{-1}\pi\subset  (\Delta_{\ol{0}}\cap \Delta^+_{u'})$.
Since $u'\in \Sp(v)$ this implies $w^{-1}\pi\subset  \Delta^+_{\ol{0}}$.
Recall that $W$ is the Coxeter group generated by $s_{\alpha}$ with
$\alpha\in\pi$. If  $w\not=1$, then there exists
$\alpha\in\pi$ such that $w^{-1}\alpha$ is a non-positive linear combination
of the elements in $\pi$, so $w^{-1}\alpha\in -Q^+_v$, a contradiction.

We conclude that  for $u\in Sk(v)$ we have 
$$\Delta_{\an}\cap \Delta^+_u=\Delta_{\an}\cap \Delta^+_v\ \ \Longleftrightarrow\ \ \
\Delta_{\ol{0}}\cap \Delta^+_u=\Delta_{\ol{0}}\cap \Delta^+_v\ \ \Longleftrightarrow\ \ \
u\in\Sp(v).$$

In other words, $\Sp(v)$ classifies the attainable sets of simple roots $\Sigma'$ with the property
$\Delta_{\ol{0}}\cap \mathbb{Z}_{\geq 0}\Sigma'=\Delta^+_{\ol{0}}$.
By~\ref{l(w)}, for any attainable set of simple roots $\Sigma''$ there exists a unique element $w\in W$ such that
$w\Sigma''=\Sigma_u$ for some $u\in\Sp(v)$.

 \subsection{Type I and type II}
Let $\Sp(v)$ be an indecomposable admissible component  of $\Sp$.

\subsubsection{}
\begin{defn}{defnpiS}
We call $\alpha\in\fh^*$  {\em $S$-principal} 
\footnote{these elements are called principal in~\cite{S}.} 
 if for some $u\in\Sp(v)$ and $x\in X$ we have $\alpha=b_u(x)$ and \(p_u(x)=\ol{0}\) or $\alpha=2b_u(x)$ and \(p_u(x)=\ol{1}\) with \(a^u_{xx}\neq 0\).
 We denote  the set of $S$-principal elements by $\pi_S$. 
 \end{defn}

\subsubsection{Remark}
 For any $u\in\Sp(v)$ and any $\alpha\in\Sigma_u$ exactly one of the following properties hold:
 either $\alpha\in \pi_S$ or $2\alpha\in\pi_S$ or $\alpha\in  \Delta_{\is}$. 
 Using~(\ref{2alpha}) and~(\ref{Delta+uv}) we obtain
 $\pi\subset\pi_S\subset \Delta^+_u$ for all $u\in\Sp(v)$.

\subsubsection{}
\begin{prop}{lempis}
If $\Sp(v)$ is an indecomposable admissible component, then
\begin{equation}\label{lempiseq}
\mathbb{C}\pi_S\cap \Delta_{\is}\not=\emptyset\ \ \ \Longrightarrow\ \ \ \mathbb{C}\Delta=\mathbb{C}\pi_S.
\end{equation}
\end{prop}
\begin{proof}
For any $u\in\Sp(v)$ we introduce
$$X'_u:=\{x\in X|\ b_u(x)\in \mathbb{C}\pi_S\},\ \ \ X''_u:=\{x\in X|\ b_u(x)\in\Delta_{\is}\}$$
and observe that $X'_u\cup X''_u=X$. 

Assume that the formula~(\ref{lempiseq}) does not hold, i.e.
$\mathbb{C}\pi_S\cap \Delta_{\is}\not=\emptyset$ and  
$\mathbb{C}\Delta\not=\mathbb{C}\pi_S$. 

Since $\mathbb{C}\pi_S\cap \Delta_{\is}\not=\emptyset$ one has
$b_{u_0}(x)\in \mathbb{C}\pi_S$ for some $u_0\in \Sk(v)$ and $x\in X$ such that $a_{xx}^{u_0}=0$.
By~\ref{l(w)}, $u_0=wu$ for some $w\in W$ and $u\in \Sp(v)$. Since $W$ is generated by the reflections
with respect to the elements in $\pi\subset \pi_S$, the set $\mathbb{C}\pi_S$ is $W$-invariant, so
 $b_{u}(x)=w^{-1} b_{u_0}(x)\in \mathbb{C}\pi_S$. Since $\Delta_{\is}$ is $W$-invariant,
 $b_u(x)\in \Delta_{\is}$. Thus $x\in X'_u\cap X''_u$, so  $X'_u\cap X''_u\not=\emptyset$.

For any $u\in\Sp(v)$ we define
 a ``simplified Dynkin diagram'' $\Gamma(\Sigma_u)$: this is a graph with the set of vertices
$X$ where $x$ and $y$ are connected by an edge $x\to y$ if $\langle a_u(x), b_u(y)\rangle\not=0$. Since $\Sp(v)$
is indecomposable, $\Gamma(\Sigma_u)$ is a connected graph for each $u\in\Sp(v)$.

Since $\mathbb{C}\Delta\not=\mathbb{C}\pi_S$ we have $X\setminus X'_u\not=\emptyset$ for all
$u\in\Sp(v)$; by above, the set $X'_u\cap X''_u$  is non-empty for some $u\in\Sp(v)$.
Consider all paths in all  graphs $\Gamma(\Sigma_u)$  which connect
a vertex in $X\setminus X'_u$ with a vertex in $X'_u\cap X''_u$ (by above, this set is non-empty for some $u\in\Sp(v)$).
Let  $u\in\Sp(v)$ be such that $\Gamma(\Sigma_u)$ contains a path of the minimal possible length and
let $x_1\to x_2\to \ldots\to x_s$ be such a path. Then 
\begin{equation}\label{xxx}\begin{array}{lcl}
\langle a_u(x_{j-1}), b_u(x_{j})\rangle\not=0\   & & \text{ for }\ 1<j\leq s;\\
x_1\in X\setminus X'_u,\ \ x_s\in X'_u\cap X''_u,\ \ x_i\in X'_u\setminus  X''_u \    & & \text{ for }\ 1<i<s\\
\langle a_u(x_{i}), b_u(x_{j})\rangle=0\    & & \text{ for }\ 1\leq i<j-1,\ \ j\leq s.\\
\end{array}
\end{equation} 
(The last two lines follow from the minimality of the length).

Since $x_1\in X\setminus X'_u$, the root $b_u(x_1)$ is isotropic. Then $\Sp(v)$ contains a vertex $u_1$
with the reflexion  $r_{x_1}: u\to u_1$. By~(\ref{xxx}) one has $a^u_{x_1 x_2}\not=0$  and $a^u_{x_1 x_i}=0$ 
for $2<i\leq s$. This gives
\begin{equation}\label{bu1}
b_{u_1}(x_2)=b_u(x_1)+b_u(x_2),\ \ \ b_{u_1}(x_j)=b_u(x_j),\ \ a_{u_1}(x_j)=a_u(x_j)\ \ \text{ for }2<j\leq s.
\end{equation} 
Since $x_1\not\in X'_u$ and $x_2\in X'_u$, we have 
$b_{u_1}(x_2)\not\in \mathbb{C}\pi_S$, so $x_2\in X \setminus X'_{u_1}$. 

If $s=2$, then 
$b_u(x_1),b_u(x_2)\in\Delta_{\is}$, so $b_{u_1}(x_2)$ is even and thus $b_{u_1}(x_2)\in\pi_S$, a contradiction.
Hence $s\geq 3$ and
$$a^{u_1}_{x_2 x_3}=\langle a_{u_1}(x_2), b_{u_1}(x_3)\rangle=
\langle a_u(x_2)+ \frac{a^u_{x_2 x_1}}{a^u_{x_1 x_2}}a_u(x_1), b_u(x_3)\rangle=\langle a_u(x_2),
b_u(x_3)\rangle\not=0.$$
Using~(\ref{bu1}) we conclude that $\Gamma(\Sigma_{u_1})$
contains a path $x_2\to x_3\to \ldots \to x_s$ and $x_s\in X'_{u_1}\cap X''_{u_1}$. By above,
$x_2\in X\setminus X'_{u_1}$. This contradicts to the assumption that 
the path $x_1\to \ldots\to x_s$ has the minimal possible length.
\end{proof}

\subsubsection{}\label{TypeIandII}
Let $Q_0$ be the lattice spanned by $\Delta_{\ol{0}}$. 
By~\cite{S}, Theorem 2.5 the group $Q_v/Q_0$ is either $1$ or $\mathbb{Z}$ or $\mathbb{Z}_2$; moreover,
if $Q_v/Q_0\cong \mathbb{Z}$, then
$\fg^{\ttC}$ admits a $\mathbb{Z}$-grading $\fg^{\ttC}=\fg_{-1}\oplus \fg_0\oplus \fg_1$
such that $\fg_0=\fg_{\ol{0}}$ and $\fg_{-1}\oplus \fg_1=\fg_{\ol{1}}$. The superalgebra
$\fg^{\ttC}$  is called of {\em type I} if  $Q_v/Q_0\cong \mathbb{Z}$ and of {\em type
 II} otherwise.   We have 
 \begin{equation}\label{eqtypeIandII}
  \mathbb{C}\Delta_{\ol{0}}=\mathbb{C}\Delta \ \ \Longleftrightarrow\ \ \ \mathbb{Q}\Delta_{\ol{0}}=\mathbb{Q}\Delta \ \ \Longleftrightarrow\ \ \ \text{ type II}.
 \end{equation}

 \subsubsection{}
 \begin{cor}{cortypeI}
  \begin{equation}\label{typeIimplication}
  \text{ type I}\ \ \ \Longrightarrow\ \ \ 
 \mathbb{C}\pi_S\cap \Delta_{\is}=\emptyset.
\end{equation}
 \end{cor}
 \begin{proof}
 Let $\Sp(v)$ be of the type I. Then $\mathbb{C}\Delta_{\ol{0}}\subsetneq\mathbb{C}\Delta$
 so  $\mathbb{C}\pi_S\not=\mathbb{C}\Delta$ (since $\pi_S\subset Q_0$). \Prop{lempis} gives
  $\mathbb{C}\pi_S\cap \Delta_{\is}=\emptyset$ as required.\end{proof}

 \subsubsection{Kac-Moody components}\label{KMtypeIandII}
 For a Kac-Moody component we have $\pi_S=\pi$.
 For a purely non-isotropic Kac-Moody component we have 
 $\mathbb{Q}\pi_S=\mathbb{Q}\Delta$ (in particular, such component is   of type II).
 From the van de Leur-Hoyt classification 
it is easy to see that  the only  Kac-Moody components of type I are $A(m|n)$,   
$C(k+1)$, $A(m|n)^{(1)}$ , $C(k+1)^{(1)}$,
 and $S(1|2;b)$ (with $m,n\geq 0, mn>0,
  k\geq 1, b\in\mathbb{C}\setminus\mathbb{Z}$).  
  
 Since
 $\mathbb{Q}\Delta_{\ol{0}}=\mathbb{Q}\pi+\mathbb{Q}\Delta^{\ima}_{\ol{0}}$, we have
 $\mathbb{Q}\Delta=\mathbb{Q}\pi$ for the type (Fin) II and 
$\mathbb{Q}\Delta=\mathbb{Q}\pi+\mathbb{Q}\delta$
for the type (Aff) II. From the  van de Leur-Hoyt classification it follows that for the type (Aff) we have
 $\delta\in \mathbb{Q}\pi$. Hence $\mathbb{Q}\Delta=\mathbb{Q}\pi$ for the types (Fin) II
 and (Aff) II.

\subsection{Root bases}
The following definition extends the definition in~\cite{Kbook}, 5.9  given in the context of Kac-Moody algebra.
\subsubsection{}
\begin{defn}{rootbasis} 
A linearly independent set of roots $\Sigma'$ is called a \emph{root basis of $\Delta$} if
$$\Delta\subset (\sum_{\alpha\in\Sigma'} \mathbb{Z}_{\geq 0}\alpha) \bigcup (\sum_{\alpha\in\Sigma'} \mathbb{Z}_{\leq 0}\alpha).$$
If  $\Sigma'$ is  a root basis of $\Delta$ we set 
$\Delta^+(\Sigma'):=(\sum_{\alpha\in\Sigma'} \mathbb{Z}_{\geq 0}\alpha)\cap \Delta$.
Then 
$$\Delta=-\Delta^+(\Sigma')\coprod \Delta^+(\Sigma').$$
\end{defn}

\subsubsection{}
\begin{defn}{rootbases} 
	 We call $h\in\fh^*$ {\em generic} if $\Ree\langle \alpha,h\rangle \not=0$ for all $\alpha\in\Delta$.
For a generic $h\in\fh^*$ we set
$$\Delta_{>0}(h):=\{\alpha\in\Delta|\ \Ree\langle \alpha,h\rangle >0\}.$$
We denote by $\tilde{\Sigma}_{h}$ the set of indecomposable elements in $\Delta_{>0}(h)$.
By definition, any element in $\Delta_{>0}(h)$ is a non-negative integral linear combination of elements
in  $\tilde{\Sigma}_{h}$.
\end{defn}

\subsubsection{}
Recall that  ``$\operatorname{Root\ Bases}$''  denotes  the set of root bases. One has
\begin{equation}\label{basesinclusion}
\{\Sigma_u\}_{u\in \Sk(v)} \subset  \operatorname{Root\ Bases}\subset 
\{\tilde{\Sigma}_{h}\}_{h\text{ is generic in }\fh}.\end{equation}

\subsection{}
\begin{lem}{lemfindif}
\begin{enumerate}
\item  For any $u\in\Sk(v)$ the set  $\Delta^+_u\setminus \Delta^+_v$ is finite.
\item Let $\Sk(v)$ be a Kac-Moody component and $h_0\in\fh^*$  
be a generic element. The following conditions are equivalent
\begin{enumerate}
\item the set $\Delta_{>0}(h_0)\setminus \Delta^+_v$ is finite;
\item the set $(\Delta_{>0}(h_0)\cap \Delta_{\re})\setminus \Delta^+_v$ is finite;
\item  $\Delta_{>0}(h_0)=\Delta^+_u$ for some $u\in\Sk(v)$.
\end{enumerate}
\end{enumerate}
\end{lem}	
\begin{proof}
By~(\ref{Delta+uv}), the set $\Delta^+_u\setminus\Delta^+_v$ is finite for any $u\in\Sk(v)$.
This gives (i) and the implication (c) $\Longrightarrow $ (a) in (ii).
Now it is enough to verify  the  implication (b) $\Longrightarrow $ (c).
For $u\in\Sk(v)$ we set 
$$D_u:=(\Delta_{>0}(h_0)\cap \Delta_{\re})\setminus \Delta^+_u.
$$
and proceed by by induction on the cardinality of $D_v$.
If $D_v$ is empty, then $\Sigma_v\subset \Delta_{>0}(h_0)$; using
$$\Delta=-\Delta_{>0}(h_0)\coprod \Delta_{>0}(h_0)= -\Delta^+_v\coprod\Delta^+_v,$$
we obtain $\Delta^+_v=\Delta_{>0}(h_0)$.
Assume that $D_v$ is a non-empty finite set. Then  
$\Sigma_v\not\subset \Delta_{>0}(h_0)$.
Take $\alpha=b_v(x)$ such that
$\alpha\not\in  \Delta_{>0}(h_0)$ and consider the reflexion  $r_x: v\to u$.
By~(\ref{Delta+uv}), $D_u\subset (D_v\setminus\{-\alpha\})$, so the cardinality of $D_u$ is less than the cardinality of $D_v$.
By induction, $\Sk(v)=\Sk(u)$  contains a vertex $u'$ such that $\Delta_{>0}(h_0)=\Delta^+_{u'}$.
\end{proof}

\subsubsection{}
By~\cite{Kbook}, for an indecomposable infinite-dimensional Kac-Moody algebra
 for any root basis $\Sigma'$  there exists $w\in W$ such that
 $\Sigma'=w\Sigma$ or $-\Sigma'=w\Sigma$.
 In this case $\Sk(v)=Wv$, so this can be reformulated in the following way:
 for any root basis $\Sigma'$  there exists $u\in \Sk(v)$ such that
 $\Sigma'=\Sigma_u$ or $-\Sigma'=\Sigma_u$.
The following proposition extends this result to the types (Fin) and (Aff).
For the type (Ind) see~\Prop{propiii}.
 
\subsection{}
\begin{prop}{proprootbasisfinaff}
\begin{enumerate}
\item For a  Kac-Moody component of type (Fin) we have 
$$\{\Sigma_u\}_{u\in \Sk(v)}=\operatorname{Root\ Bases} =
\{\tilde{\Sigma}_{h}\}_{h\text{ is generic in } \fh}.$$
\item For an indecomposable Kac-Moody component of type (Aff) we have
$$\{\Sigma_u\}_{u\in \Sk(v)}\coprod\{-\Sigma_u\}_{u\in \Sk(v)}=\operatorname{Root\ Bases}=
\{\tilde{\Sigma}_{h}\}_{h\text{ is generic in } \fh}.$$
\end{enumerate}
\end{prop}
\begin{proof}
\Lem{lemfindif} implies (i) and the fact that
 $\{\Sigma_u\}_{u\in \Sk(v)}\cap\{-\Sigma_u\}_{u\in \Sk(v)}=\emptyset$ if 
 $\Delta$ is infinite. 
 
 For (ii) assume that $\Delta$ is of type (Aff). In light of~(\ref{basesinclusion}) it is enough
 to check that for any generic $h$ one of the sets
 $\Delta_{>0}(h)\setminus \Delta^+_v$ or $-\Delta_{>0}(h)\setminus \Delta^+_v$ is finite.  
 By~(\ref{affmodulodelta}) the set $\Delta_v$ contains a finite set $U$ such that
 $$\Delta^+_v\subset U+\mathbb{Z}\delta.$$
  Without loss of generality we can (and will) assume that $\Ree\langle \delta,h\rangle>0$.
 Then for each $\alpha\in U$ almost all roots of the form $\alpha+j\delta$ lie in $\Delta_{>0}(h)$.
 Hence $\Delta_{>0}(h)\setminus \Delta^+_v$ is finite. This completes the proof of (ii).
\end{proof}

\subsubsection{}
\begin{lem}{Sigmapraff}
If $\ \Sk(v)$ is an indecomposable Kac-Moody component of the type  (Aff), then
$\pi$ is finite.
\end{lem}
\begin{proof}
Denote by $V$ the $\mathbb{R}$-span of $\Delta$.
By~(\ref{affmodulodelta})
one has $(\Delta^{\ima})^+=\mathbb{Z}_{>0}\delta$ and 
 the image of $\Delta$
 in $V/\mathbb{R}\delta$ is finite.
 Take $\alpha,\beta\in\Sigma_{\pr}$. 
 Fix $u$ such that $\alpha\in\Sigma_u$. Since $\beta,\delta\in\Delta^+_u$ we have
 $\alpha-\beta\not\in \mathbb{R}_{>0}\delta$. Similarly, 
 $\beta-\alpha\not\in \mathbb{R}_{>0}\delta$. 
 Hence the restriction of the map
 $V\to V/\mathbb{R}\delta$   to $\Sigma_{\pr}$ is injective. Since the image of $\Delta$
 is finite, $\Sigma_{\pr}$ is finite. 
\end{proof}

\subsubsection{}
\begin{cor}{corrootbasisfinaff}
Let $\Sk(v)$ be an indecomposable Kac-Moody component of the type (Fin) or (Aff).
\begin{enumerate}
\item
If $h$ is generic and $\Delta_{>0}(h)_{\ol{0}}=\Delta_{\ol{0}}^+$, then
$\Delta_{>0}(h)=\Delta^+_u$ for some  $u\in \Sp(v)$.
\item One has $\bigcap\limits_{u\in \Sp(v)} C(\Sigma_u)=C(\pi)$.
\end{enumerate}
\end{cor}
\begin{proof}
Combining~\Prop{proprootbasisfinaff} and~\ref{spineDelta0} we obtain (i).

For (ii) note that  $C:=\bigcap\limits_{u\in \Sp(v)} C(\Sigma_u)$
is a closed convex cone (since each  $C(\Sigma_u)$ is a closed convex cone).
By~\Lem{Sigmapraff}, $C(\pi)$ is also a closed convex cone.
One has $C(\pi)\subset C$, since $\pi\subset Q^+_u$ for any $u\in \Sp(v)$.
If $C(\pi)\not=C$ there exists a generic $h\in\fh$ such that $\langle C(\pi),h\rangle >0$
and  $\langle \mu,h\rangle <0$ for some $\mu\in C$. By~\Prop{proprootbasisfinaff} 
 $\Delta_{>0}(h)=\Delta^+_{u}$ for some $u\in \Sk(v)$. Since
$\pi\subset \Delta_{>0}(h)$, \ref{spineDelta0} gives $u\in\Sp(v)$.
Then $\mu\in \mathbb{R}_{\geq 0}\Sigma_u$, so $\langle \mu,h\rangle> 0$, a contradiction. 
\end{proof}

\subsubsection{}
\begin{cor}{corinf}
 If $\Sp(v)$ 
is  an indecomposable component of the type (Aff) I, then 
$\Sp(v)$ is infinite.
\end{cor}
\begin{proof}
For the type I we have $\mathbb{Q}\Delta_{\ol{0}}\not=\mathbb{Q}\Delta$, so
$\Delta_{\is}\not=\emptyset$. From the van de Leur-Hoyt classification it follows that
 $\Delta^+_{\is}$ is infinite. 
Assume that $\Sp(v)$ is finite. Since for any $u\in \Sp(v)$ the set 
$\Delta^+_u\setminus \Delta^+_v$ is finite, 
the set 
$\bigl(\bigcap\limits_{u\in \Sp(v)} \Delta^+_{\is;u}\bigr)$ is non-empty.
Clearly, this set lies in $C=C(\pi)\subset \mathbb{C}\pi$. 
Since $\pi_S=\pi$, this contradicts to~(\ref{typeIimplication}).
\end{proof}

\subsubsection{Remark}
It is not difficult to show that 
the spine is finite for all other indecomposable Kac-Moody component, see, for example, \ref{infess}
or~\Lem{lemspinefin} below. The fact that the spine is infinite for  $A(m|n)^{(1)}$
and $S(1|2;b)$ was known, see~\cite{S} and~\cite{GHS} (see~\ref{Amnaff} for details).

\subsection{}
\begin{lem}{Winvariant}
One has $Q^{++}_{\mathbb{R}}=
\bigcup\limits_{w\in W} wK'_{\mathbb{R}}$, 
where 
$$K'_{\mathbb{R}}:=\{ \mu\in \bigcap\limits_{u\in \Sp(v)} C(\Sigma_u)|\ \langle \mu,\alpha^{\vee}\rangle \leq 0\ 
\text{ for all }\alpha\in\pi\}.$$
Moreover, any element in $Q^{++}_{\mathbb{R}}$ is $W$-conjugated to a unique element in $K'_{\mathbb{R}}$.
\end{lem}
\begin{proof}
The proof is similar to the proof of Proposition 5.2 (b) in~\cite{Kbook}. 
We present it below for the sake of completeness. We set  
$$Q^{++}:=\bigcap\limits_{u\in \Sk(v)} Q^+_u=Q^{++}_{\mathbb{R}}\cap Q_v,\ \ \ 
K':=K'_{\mathbb{R}}\cap Q_v.$$

Let us verify that $Q^{++}=\bigcup\limits_{w\in W} wK'$. 
Since $\Sk(v)$ contains
$wv$, the set $Q^{++}$ is $W$-invariant.
For $\mu=\sum\limits_{\alpha\in\Sigma_v} k_{\alpha}\alpha\in Q_v$
the sum $\htt\mu:=\sum\limits_{\alpha\in\Sigma_v} k_{\alpha}$	is called the {\em height} of $\mu$.
Take $\mu\in Q^{++}$. By above,
$W\mu\in Q^{++}$, so $\htt (w\mu)\in\mathbb{Z}_{\geq 0}$.
Let $\mu_0$ be the element of minimal height in $W\mu$.
For $\alpha\in\pi$ we have  $\htt s_{\alpha}\mu_0\geq \htt\mu_0$
which gives  $\langle \mu_0, \alpha^\vee \rangle\leq 0$. Hence $\mu_0\in K'$.
This establishes the formula  $Q^{++}=\bigcup\limits_{w\in W} wK'$ which implies
 $Q^{++}_{\mathbb{R}}=\bigcup\limits_{w\in W} wK'_{\mathbb{R}}$ (since the union of the rays $\mathbb{R}_{\geq 0}\mu$ for $\mu\in Q^{++}$ is dense  in   $Q^{++}_{\mathbb{R}}$). 

The fact that each $W$-orbit has at most one intersection
with $K'_{\mathbb{R}}$ follows from~\cite{Kbook}, Proposition 3.12 (b). 
\end{proof}

\subsection{Essentially simple roots}
The following notion was used in~\cite{GKadm}. A root $\alpha$ is called {\em essentially simple}
if  $\alpha\in\Sigma_u$ for some $u\in\Sp(v)$. By~(\ref{Sigmau=Sigmav}),
the finiteness of spine is equivalent to the finiteness of the 
set of essentially simple roots.

By definition, $\Sigma_{\pr}$ is the set 
of anisotropic essentially simple roots.  By~(\ref{Delta+uv}) 
 a positive isotropic root $\alpha$ is not essentially simple if and only if
 $\alpha\in \bigcap\limits_{u\in \Sp(v)} C(\Sigma_u)$.

\subsubsection{}\label{essence}
Consider the case when $\Sk(v)$ is indecomposable of type (Fin) or (Aff).
By~\Cor{corrootbasisfinaff}, 
a positive isotropic root $\alpha$ is  not essentially simple if and only if $\alpha\in C(\pi)$. Hence the set of essentially simple isotropic roots is equal to
$$\Delta_{\is}\setminus \bigl( -C(\pi)\cup C(\pi)\bigr).$$

By~(\ref{typeIimplication}), 
 all isotropic roots  are essentially simple in the types (Fin) I and (Aff) I.

\subsubsection{}\label{infess}
Using van de Leur-Hoyt classification
it is not hard to show that the set of essentially simple roots are
 finite for all indecomposable  components of the type (Aff) II.

Indeed, for a component of type II, the span of  $\pi$ contains $\Delta$.
Moreover, any connected component $\Sigma^i_{\pr}$ of the Dynkin diagram of $\Sigma_{\pr}$
is of type (Aff) and the minimal imaginary root of this component is propotional to $\delta$,
so $\delta=\sum\limits_{\alpha\in \Sigma^i_{\pr}} k_{\alpha}\alpha$ 
for some $k_{\alpha}\in\mathbb{Q}_{>0}$. As a result, for any $\beta\in\Delta$
one has $\beta+j\delta\in C(\pi)$ for $j>>0$.
Since the image of $\Delta$ in $\fh^*/\mathbb{C}\delta$ is finite,
the set 
$$\Delta\setminus \bigl( -C(\pi)\cup C(\pi)\bigr)$$
is finite. Hence  the set of essentially simple isotropic roots is finite.
Since $\Sigma_{\pr}$ is finite, the set of essentially simple roots is finite.
This implies the finiteness of the spine for all indecomposable  components of the type (Aff) II.

\subsection{}
\begin{cor}{cor264}
\Thm{corintrospine} and the formulae~(\ref{suppconnected}), (\ref{eqimagine}) hold for the types (Fin) and (Aff).
\end{cor}
\begin{proof}
Combining~\ref{Sigmapraff}, \ref{corrootbasisfinaff} and~\ref{Winvariant} we obtain 
\Thm{corintrospine} (i) and (ii); (iii) follows from~\ref{essence} and (iv) from~~\ref{essence}, \ref{infess}.
By~\ref{Delta+} one has $(\Delta^{\ima})^+\subset Q^{++}_{\mathbb{R}}\cap Q_v$.
Since $ Q^{++}_{\mathbb{R}}$ is empty in the type (Fin)
and $ Q^{++}_{\mathbb{R}}\cap Q_v=\mathbb{Z}_{>0}\delta$
in the type (Aff), this gives  $(\Delta^{\ima})^+=Q^{++}_{\mathbb{R}}\cap Q_v$.
Finally, (\ref{suppconnected}) follows from the above formula and~\ref{KMcomponents}.
\end{proof}

\section{Principal roots}
In this section $\Sk(v)$ is an admissible component of  $\Sk$.

\subsection{Definitions and main results}\label{piandpr}
Let $\Sigma_{pr}=\{\alpha_i\}_{i\in I}$ be the set of principal roots.
We denote by $B_{pr}$ the matrix $\langle \alpha^{\vee}_i,\alpha_j\rangle$. 
For each odd root $\alpha\in \Sigma_{\pr}$  we set $(2\alpha)^{\vee}:=\frac{\alpha^{\vee}}{2}$ and introduce
$$\pi=\{\alpha\in \Sigma_{\pr}|\ p(\alpha)=\ol{0}\}\cup \{2\alpha|\ \alpha\in \Sigma_{\pr},\ p(\alpha)=\ol{1}\}.$$
  Note that $\pi\subset\Delta_{\ol{0}}$.
Set $\pi=\{\alpha'_i\}_{i\in I}$ and  denote by $B_{\pi}$ the matrix $\langle (\alpha')^{\vee}_i,\alpha'_j\rangle$.

 In~\Prop{lemBpr} below we show that 
$B_{pr}$ and $B_{\pi}$  are  generalized Cartan matrices in the sense of~\cite{Kbook}
if $\Sigma_{\pr}$ is finite, and ``infinite generalized Cartan matrix''
if $\Sigma_{\pr}$ is infinite.

We call $\alpha\in \Delta^+_{\an}$
{\em indecomposable} if $\alpha\not=\alpha_1+\alpha_2$ for 
all $\alpha_1,\alpha_2\in \Delta^+_{\an}$. In~\Prop{propprin}  we will show that 
the principal roots are precisely indecomposable elements in $\Delta^+_{\an}$.

The proofs  are very similar to the proofs in~\cite{S}. 
\subsection{}
\begin{lem}{ialphajbeta}
 Let $\alpha,\beta$ be distinct principal roots.
If $\gamma:=i\alpha-j\beta\in\Delta_{\an}$  for some $i,j>0$,  then $\gamma$ is an odd real isotropic root.
\end{lem}
\begin{proof}
Without loss of generality we can assume that $\gamma\in\Delta^+_v$.
Taking $v'\in\Sp(v)$ such that $\alpha\in\Sigma_{v'}$
	we obtain $i\alpha=j\beta+\gamma$.
	Since $\alpha,\beta$ are distinct principal roots, they are not proportional and
	$\beta\in\Delta^+_{v'}$. Since $\alpha\in\Sigma_{v'}$, this implies
	$\gamma\in\Delta^-_{v'}$. Thus $\gamma\in\Delta^+_v\cap \Delta^-_{v'}$, so
$\gamma\in\Sigma_{v''}$ for some $v''\in 	\Sp(v)$.
 \end{proof}	
	
\subsection{}
\begin{prop}{lemBpr}
$B_{pr}$ and $B_{\pi}$  a generalized Cartan matrix i.e.,  for $\alpha,\beta\in\Sigma_{\pr}$ 
or $\alpha,\beta\in\pi$ one has
 \begin{enumerate}
 \item $\langle \alpha,\alpha^{\vee}\rangle=2$.
 \item   if $\alpha\not=\beta$, then 
  $\langle \beta,\alpha^{\vee}\rangle\in\mathbb{Z}_{\leq 0}$, and
	$\langle \beta,\alpha^{\vee}\rangle$ is even if $p(\alpha)=\ol{1}$. 
	\item $\langle \alpha,\beta^{\vee}\rangle=0\ \Longrightarrow\ \ 
\langle \beta,\alpha^{\vee}\rangle=0$.
\end{enumerate}
\end{prop}
\begin{proof}
 (i) follows from the definition of $\alpha^{\vee}$. 
 
 Let $\alpha,\beta$ be distinct principal roots.
Set $j:=\langle \beta,\alpha^{\vee}\rangle$. 
By~(\ref{betaalphavee}), we have $j\in\mathbb{Z}$
and $j\in 2\mathbb{Z}$ if $p(\alpha)=\ol{1}$ (since $p(\alpha)=\ol{1}$ implies $2\alpha\in \pi$ and $(2\alpha)^{\vee}=\frac{\alpha^{\vee}}{2}$).
Without loss of generality, we can assume that $\alpha\in\Sigma_v$. Since 
	$s_{\alpha}\beta=\beta-j\alpha\in\Delta_{\an}$, \Lem{ialphajbeta} implies 
	$j\leq 0$. This proves (ii).

Assume that $\langle \alpha,\beta^{\vee}\rangle=0$. Then $s_{\beta}\alpha=\alpha$ and 
$$s_{\beta}s_{\alpha}\beta=s_{\beta}(\beta-j\alpha)=-\beta-j\alpha.$$
By above, $j\leq 0$. Since $s_{\beta}s_{\alpha}\beta\in\Delta_{\an}$, \Lem{ialphajbeta} gives
$j=0$ and establishes (iii).
\end{proof}

\subsection{}
\begin{prop}{propprin} 
The set of principal roots coincides with the set of indecomposable elements in $\Delta^+_{\an}$.
\end{prop}
\begin{proof}
Let $\beta$ be a  principal root. Then $\beta\in\Sigma_v$ for some $v$ in the spine. Since 
$\Delta^+_{\an}\subset \Delta^+_v$, any decomposable element
in $\Delta^+_{\an}$ is decomposable in $\Delta^+_v$. Since $\beta\in\Sigma_v$,
$\beta$ is indecomposable in $\Delta^+_{\an}$.   

It remains to verify that any $\gamma\in\Delta^+_{\an}$ can be  written as
	a linear combination of finitely many elements in $\Sigma_{\pr}$ with a non-negative coefficients.
Take $\gamma\in\Delta_{\an}$. By~\cite{GHS} Remark 4.3.13,
 $\gamma=w\alpha$ for some $\alpha\in\Sigma_{pr}$ and $w\in W$; moreover,
 $w=s_{\alpha_{i_1}}\ldots s_{\alpha_{i_m}}$ for some $i_1,\ldots,i_m\in I$.
	Choose $w$ of the minimal length (i.e. $m$ is minimal) and set $\Sigma':=\{\alpha\}\cup
	\{\alpha_{i_j}\}_{j=1}^m$. 
	Denote by $B'_{\pr}$ the submatrix of $B_{\pr}$ which corresponds 
	to $\Sigma'$. Then  $B'_{\pr}$ is a generalized Cartan matrix (of finite size)
	and we can construct a Kac-Moody algebra for this Cartan matrix. Let $W'$ be the corresponding Weyl group. Clearly, $w\in W'$. Then
	$W'\Sigma'$ is the set of real roots for this  Kac-Moody Lie algebra, so  
	$$\gamma=w\alpha=\sum_{\beta\in \Sigma'} m_{\beta}\beta$$
	for some coefficients $m_{\beta}\in\mathbb{Z}$ and all $m_{\beta}$
	are either non-positive or non-negative. Since $\Sigma'\subset\Delta^+_{\an}$, 
	the coefficients are non-negative.  Hence any $\gamma\in\Delta^+_{\an}$ can be  written as
	a linear combination of finitely many elements in $\Sigma_{\pr}$ with a non-negative coefficients.
\end{proof}

\subsection{The algebra $\fg(B_{\pi})$}\label{gpr}
Let $\fg$ be any root algebra for an admissible component $\Sk(v)$. 
We denote by  $\fg(B_{\pi})$ the Kac-Moody algebra corresponding to the matrix $B_{\pi}$.
Denote by $\fg'$ the subalgebra of $\fg_{\ol{0}}$
which is generated by $\fg_{\pm \alpha}$ for $\alpha\in \pi$.

The following construction is similar to one appeared  in Lemma 3.7 in~\cite{S}.

\subsubsection{}\label{tildehab}
Let $(\tilde{\fh}, \tilde{a}, \tilde{b})$ be a realization of $B_{\pi}$. This means that $\tilde{\fh}^*$
contains linearly independent set $\{\tilde{\alpha}\}_{\alpha\in\pi}$ and
$\tilde{\fh}$ contains linearly independent set $\{h_{\alpha}\}_{\alpha\in\pi}$ 
such that 
$$\langle h_{\alpha},\tilde{\beta}\rangle=\langle \alpha^{\vee},\beta\rangle\ \ \text{ for all }\alpha,\beta\in\pi.$$
We choose $\tilde{\fh}$ of the minimal possible dimension
(i.e., $\dim \tilde{\fh}=|\Sigma_{\pr}|+\corank B_{\pi}$, see~\cite{Kbook} Proposition 1.1). Let
$$\tilde{\fg}(B_{\pi})=\tilde{\fn}^{-}\oplus\tilde{\fh}\oplus\tilde{\fn}^{+}$$
be the half-baked algebra for $B_{\pi}$: this  is a Lie algebra generated by  $\tilde{\fh}$
and the even elements $\tilde{e}_{\pm\alpha}$ for $\alpha\in\pi$  subject to the relations
(\ref{relationshalfbaked}). 
One has
$$ [\tilde{\fg}(B_{\pi}),\tilde{\fg}(B_{\pi})]=\tilde{\fn}^{-}\oplus\tilde{\fh}'\oplus\tilde{\fn}^{+}$$
where $\tilde{\fh}'$ is spanned by $h_{\alpha}$, $\alpha\in\pi$.
By~\cite{Kbook}, Theorem 1.2, $\tilde{\fn}^+$ (resp., $\tilde{\fn}^{-}$) is a
 free Lie algebra generated by $\tilde{e}_{\alpha}$, $\alpha\in\pi$  (resp., $\tilde{e}_{-\alpha}$, $\alpha\in\pi$). 
 Using this result, it is not hard to show that $[\tilde{\fg}(B_{\pi}),\tilde{\fg}(B_{\pi})]$ is generated by  
 $h_{\alpha},\tilde{e}_{\pm\alpha}$ for $\alpha\in\pi$ subject to the relations~(\ref{relationshalfbaked}) for $h\in\tilde{\fh}'$
 (the freeness result implies that there are no additional relations).
Combining Lemmatta~\ref{ialphajbeta} and~\ref{lemBpr} we conclude
that there exists a surjective homomorphism 
$$s:\ [\tilde{\fg}(B_{\pi}),\tilde{\fg}(B_{\pi})]\twoheadrightarrow \fg'$$
which maps $\tilde{\fn}^{\pm}_{\pm\alpha}$ to $\fg_{\pm\alpha}$
for all $\alpha\in\pi$.  Since $\Sigma_{\pr}\subset\Delta^+$
there exists $h\in\fh$ such that $\langle h,\alpha\rangle\in \mathbb{R}_{>0}$
for  all $\alpha\in\pi$. Take $\tilde{h}\in \tilde{\fh}$ such that
$\langle \tilde{h},\alpha\rangle=\langle h,\alpha\rangle$.
The actions of $h$ and $\tilde{h}$ define  $\mathbb{R}$-gradings
on $[\tilde{\fg}(B_{\pi}),\tilde{\fg}(B_{\pi})]$ and on
$\fg'$. Then $s$ is the homomorphism of the graded algebras, so $\Ker s$ is a graded ideal. Therefore
$$\Ker s=\fm^-\oplus \fm_0\oplus \fm^+$$
where $\fm_0\subset \tilde{\fh}$ and $\fm^{\pm}\subset \tilde{\fn}^{\pm}$.
For any $\alpha\in\Sigma_{\pr}$ we have
 $\tilde{\fg}(B_{\pi})_{\pm\alpha}\not\in\Ker s$, so
 $$[\tilde{\fg}(B_{\pi})_{\alpha}, \fm_0]=0,\ \ \ 
 [\tilde{\fg}(B_{\pi})_{-\alpha},\fm^{+}]\subset \fm^+,\ \ 
   [\tilde{\fg}(B_{\pi})_{\alpha},\fm^{-}]\subset \fm^-.$$
 Therefore $\fm_0$ lies
in the centre of $\tilde{\fg}(B_{\pi})$ and
$\fm^-\oplus  \fm^+$ is an ideal of $\tilde{\fg}(B_{\pi})$.
Let
$$q_{\U}:\tilde{\fg}(B_{\pi})\to \fg^{\U}(B_{\pi}),\ \ \ \ \ q_{\ttC}:\tilde{\fg}(B_{\pi})\to \fg^{\ttC}(B_{\pi})=\fg(B_{\pi})$$
be the canonical epimorphisms. 
Recall that $\Ker q_{\ttC}$ is the maximal ideal having zero intersection with $\tilde{\fh}$; 
by above, $\fm^-\oplus  \fm^+\subset \Ker q_{\ttC}$.
On the other hand, $\Ker q_{\U}$ is the minimal ideal $I$ with property that 
for each $\alpha\in\pi$ the root space 
$\tilde{\fg}(B_{\pi})_{\pm\alpha}$ does not lie in $I$ and
 acts locally nilpotently
on the quotient $\tilde{\fg}(B_{\pi})_{\pm\alpha}/I$.
Since $\fg_{\pm\alpha}$ acts locally nipotently on $\fg'$ we conclude that
$\Ker q_{\U}$ lie in $\fm^-\oplus  \fm^+$, Hence
$$\Ker q_{\U}\subset (\fm^-\oplus  \fm^+)\subset \Ker q_{\ttC}.$$
 This gives a surjective map
\begin{equation}\label{s'}
s':\ \fg'\twoheadrightarrow [{\fg}(B_{\pi}),{\fg}(B_{\pi})]/\mathfrak{c},\end{equation}
where  $\mathfrak{c}$, which is the image of $\fm_0$, is
a subalgebra of  the centre of ${\fg}(B_{\pi})$.

\subsubsection{Remark}\label{c=0}
It is easy to see that $\mathfrak{c}=0$  if and only if $\pi$
is linearly independent.

\subsubsection{Remark}
The above construction does not work
if we substitute $\pi$ by $\Sigma_{\pr}$. For example, $\fg:=\fosp(3|2)$  does not contain a subalgebra
isomorphic to $\fg(B_{\pr})/\mathfrak{c}$.

\subsubsection{}\label{Bpisymm}
Consider the case when  $B_{\pi}$ is symmetrizable. By~\cite{GabberKac} symmetrizability of $B_{\pi}$ implies that for any finite subset
$\pi'\subset \pi$ the canonical map
$\fg^{\U}(B_{\pi'})\to\fg (B_{\pi'})$ is bijective.
This gives  $\fg^{\U}(B_{\pi})\iso\fg^{\ttC}(B_{\pi})$, so 
$\Ker q_{\U}= \Ker q_{\ttC}$. Then 
$\fm^-\oplus  \fm^+=\Ker q_{\U}$ and  $s$ induces an algebra isomorphism
$[{\fg}(B_{\pi}),{\fg}(B_{\pi})]/\mathfrak{c}\cong \fg'$.

\subsubsection{Remark}
By construction, $\fg'\subset [\fg_{\ol{0}},\fg_{\ol{0}}]$. We have $\fg'=[\fg_{\ol{0}},\fg_{\ol{0}}]$
in the type (Fin). The following example demonstrates that $\fg'$ 
can be a proper subalgebra of $[\fg_{\ol{0}},\fg_{\ol{0}}]$. Take  
$\fg=\fsl(m|n)^{(1)}$. In this case ${\fg}(B_{\pi})=\fsl_m^{(1)}\times\fsl_n^{(1)}$, so
$$[\fg_{\ol{0}},\fg_{\ol{0}}]_{\delta}=\dim \fg_{\delta}=m+n-1,\ \ \ 
\dim {\fg}'_{\delta}=m+n-2.$$
Other  examples include $D(m|n+1)^{(2)}$, see~\ref{remC32} below.

\subsection{Applications to $\Delta$}\label{Deltapi}
The construction described in~\ref{tildehab} is very useful for a description of $\Delta$. 

Recall that $\tilde{\pi}:=\{\tilde{\alpha}\}_{\alpha\in\pi}\subset \tilde{\fh}^*$ form a set of simple roots for $\fg(B_{\pi})$.
We denote by $\tilde{Q}^{++}_{\mathbb{R}}$ the totally positive cone for $\fg(B_{\pi})$.

For each 
$\tilde{\beta}:=\sum\limits_{\alpha\in \pi} k_{\alpha}\tilde{\alpha}\in \tilde{\fh}^*$
we assign the element $\beta:=\sum\limits_{\alpha\in \pi} k_{\alpha} {\alpha}\in \fh^*$.
The Weyl group of  $\fg(B_{\pi})$ is the Coxeter group generated by $s_{\tilde{\alpha}}$, $\alpha\in \pi$.
This group is isomorphic to $W$ and for any $w\in W$ one has $\widetilde{w\beta}=w\tilde{\beta}$. 
 Since $\fg(B_{\pi})$ is a Kac-Moody algebra, the totally positive cone is given 
$$\tilde{Q}^{++}_{\mathbb{R}}=\bigcap_{w\in W} C(w\tilde{\pi}).$$

We set
$$\begin{array}{l}
\Delta^+_{\pi}:=\{\beta|\ \tilde{\beta}\ \text{ is a positive root  of }\fg(B_{\pi})\},\\
(\Delta^{\ima})^+_{\pi}:=\{\beta|\ \tilde{\beta}\ \text{ is a positive imaginary root  of }\fg(B_{\pi})\},\\
Q^{++}_{\mathbb{R},\pi}=\{\beta|\ \tilde{\beta}\in  
\tilde{Q}^{++}_{\mathbb{R}}\}.
\end{array}$$

The construction described in~\ref{tildehab} implies 
$ \dim\fg_{\beta}\geq \dim \fg(B_{\pi})_{\tilde{\beta}}$. In particular,
$$\Delta^+_{\pi}\subset \Delta^+_{\ol{0}}.$$

\subsubsection{}
\begin{cor}{corimaginary}
One has
\begin{enumerate}
\item $\tilde{\beta}$ is a  real root of 
$\fg(B_{\pi})$ if and only if either $\beta\in (\Delta_{\an}\cap \Delta_{\ol{0}})$  or $\frac{\beta}{2}\in (\Delta_{\an}\cap \Delta_{\ol{1}})$;
\item
$ (\Delta^{\ima})^+_{\pi}\subset (\Delta^{\ima})^+_{\ol{0}}$;
\item $Q^{++}_{\mathbb{R},\pi}\subset Q^{++}_{\mathbb{R}}$.
\end{enumerate}  
\end{cor}
\begin{proof}
Now (i) follows from the fact that 
 the real roots of $\fg(B_{\pi})$  are $W$-conjugated to simple roots.
 By (i), $\tilde{\beta}$ is real if $\beta$ is not imaginary. This gives (ii).
 For (iii) recall that
  $$Q^{++}_{\mathbb{R}}=\bigcap_{u\in\Sk(v)} C(\Sigma_u) =\bigcap_{u\in\Sp(v), w\in W} 
 C( w\Sigma_u).$$
 Since $\pi\subset \Delta^+_u$ for any $u\in\Sp(v)$  one has
 $Q^{++}_{\mathbb{R},\pi}\subset Q^{++}_{\mathbb{R}}$.
 \end{proof}

\subsubsection{Remark}\label{remC32}
The converse of (ii) does not hold:
for  $D(m+1|n)^{(2)}$,
one has $\delta\in\Delta^{\ima}_{\ol{0}}$, but the minimal imaginary root of ${\fg}(B_{\pi})$ is $2\delta$.

\subsection{Example: $\fg(B_{\pi})$ for an indecomposable Kac-Moody component}\label{BpiKM}
Let $\fg$ be an indecomposable Kac-Moody superalgebra.

If $\fg$ is of type (Fin), then $\dim\fg<\infty$, so ${\fg}(B_{\pi})$ is of type (Fin).
By~\Cor{corimaginary} (iii), if $\fg$ is of type (Aff), then each indecomposable component 
of $\fg(B_{\pi})$ is of the types (Fin) or (Aff). The results of~\cite{vdL} and~\cite{Hoyt} imply the following
\begin{itemize}
\item  Each indecomposable component 
of $\fg(B_{\pi})$ is of same type (i.e., (Fin), (Aff), or (Ind)) as $\fg$.

\item If the component is not purely anisotropic of type (Ind), then 
the matrix $B_{\pi}$ is symmetrizable,
so $\fg'\cong [{\fg}(B_{\pi}),{\fg}(B_{\pi})]/\mathfrak{c}$ by~\ref{Bpisymm}.

\item In the  type (Ind) the Dynkin diagram of $\pi$ is connected
 and the elements of $\pi$ are linearly independent;
in particular, $\fg'\cong [{\fg}(B_{\pi}),{\fg}(B_{\pi})]$.
\end{itemize}

If the component is purely anisotropic, then the elements of $\pi$ are linearly independent
(in particular, $\fg'\cong [{\fg}(B_{\pi}),{\fg}(B_{\pi})]$)
and  $B_{\pi}=AD$, where $D$ is the diagonal matrix with $d_{xx}=1$
 if  $p(x)=\ol{0}$ and $d_{xx}=2$ otherwise. In this case $Q^{++}_{\mathbb{R},\pi}=Q^{++}_{\mathbb{R}}$ 
 since $\Sk(v)=Wv$.

\section{Root bases for the Kac-Moody components}
In this section we prove  Theorem~\ref{thmiii} (iii).
The proof is a variation of the proof of Theorem 8.3 in~\cite{S}.

\subsection{}
\begin{prop}{propiii}
For an indecomposable Kac-Moody component of type (Ind) we have
$$\{\Sigma_u\}_{u\in \Sk(v)}\coprod\{-\Sigma_u\}_{u\in \Sk(v)}
=\operatorname{Root\ Bases}\subsetneq
\{\tilde{\Sigma}_{h}\}_{h\text{ is generic in } \fh}.$$
\end{prop}

 \subsection{Proof}\label{Sigma'h}
 Let $\Sk(v)$ be of type (Ind).  Recall that 
 $Q^{++}_{\mathbb{R}}$ contains two non-proportional vectors
 $\nu_1$ and $\nu_2$. Take a generic $h\in \fh$ such that
 $\langle \nu_2,h\rangle >0>\langle \nu_1,h\rangle$.
 Then $\tilde{\Sigma}_h\not=\Sigma_u$ and
 $\tilde{\Sigma}_h\not=-\Sigma_u$ for any $u\in\Sk(v)$. This implies
 $$\{\Sigma_u\}_{u\in \Sk(v)}\coprod\{-\Sigma_u\}_{u\in \Sk(v)}\not=
\{\tilde{\Sigma}_{h}\}_{h\text{ is generic in } \fh}.$$ 
It remains to verify that any root basis $\Sigma'$ 
 is of the form $\Sigma_u$ or $-\Sigma_u$ for some $u\in \Sk(v)$.

 \subsubsection{Notation}\label{Indnotat}
 We fix  a root basis $\Sigma'$ such that 
\begin{equation}\label{Sigma'assm}
(\Delta^{\ima}_{\pi})^+\cap \Delta^+(\Sigma')\not=\emptyset\end{equation}
and fix $h'\in\fh$ such that
$\langle \alpha, h'\rangle=1$ for $\alpha\in\Sigma'$. 
We set 
$$Q(\pi):=\mathbb{Z}\pi,\ \  \ \ Q^+(\pi)=\mathbb{Z}_{\geq 0}\pi,\ \ \ \  \theta:=\sum\limits_{\alpha\in\pi}\alpha.$$
Let $V:=\mathbb{R}\Sigma$ be the $\mathbb{R}$-span of $\Sigma$. 
  Consider the unit spere
$$S:=\{\sum\limits_{\alpha\in\Sigma'} c_{\alpha}\alpha\in V|\ \sum\limits_{\alpha\in\Sigma'} c^2_{\alpha}=1\}.$$
For any non-zero $\mu\in V$ denote by $z_{\mu}$ the intersection
of the ray $\mathbb{R}_{\geq 0}\mu$ with the unit spere.
We set
$$\begin{array}{l}
K:=\{\mu\in V|\ \  
\langle \mu,\beta^{\vee}\rangle \leq 0\ \text{ for all }\beta\in \pi\},\\
 \mathring{K}:=\{\mu\in V|\ \ 
\langle \mu,\beta^{\vee}\rangle<0\ \text{ for all }\beta\in \pi\}.\end{array}$$

Recall that $B_{\pi}$ is an indecomposable matrix of type (IND).  Using
Theorems 5.4 and 5.6 (iii)  in~\cite{Kbook} we  get
\begin{equation}\label{Kim}
K\cap (\theta+Q^+(\pi))\subset (\Delta^{\ima}_{\pi})^+,\ \ \  
\mathring{K}\cap (\theta+Q^+(\pi))\not=\emptyset.
 \end{equation}
We fix $\nu_0\in \mathring{K}\cap (\theta+Q^+(\pi))$.

\subsubsection{}
\begin{lem}{lemV+}
The set
$\{z_{\mu}|\ \mu\in  Q^{++}_{\mathbb{R},\pi}\}$ is the closure of the set $\{z_{\alpha}| \alpha\in(\Delta^{\ima}_{\pi})^+\}$
in the metric topology.
 \end{lem}
 \begin{proof}
  Recall that $(\Delta^{\ima}_{\pi})^+\subset Q^{++}_{\mathbb{R},\pi}$. 
  Observe that 
 $Q^{++}_{\mathbb{R},\pi}=\bigcap\limits_{w\in W} C(w\pi)$
 is a $W$-invariant closed convex cone (since $C(w\pi)$ is a closed convex cone).
The set $\{z_{\mu}|\ \mu\in Q(\pi)\}$ is dense in $S$, so the set 
 $\{z_{\mu}|\ \mu\in Q(\pi)\cap Q^{++}_{\mathbb{R},\pi}\}$
 is dense in $S\cap Q^{++}_{\mathbb{R},\pi}$.

 We introduce a partial order on $\fh^*$ by setting $\mu\leq \nu$ if $\nu-\mu\in Q^+(\pi)$.
 Note that any subset in $Q^+(\pi)$ contains a minimal element with respect to this order.
 Since $(\Delta^{\ima}_{\pi})^+$ is $W$-invariant, it is enough to verify that for
any $\mu\in   Q(\pi)\cap Q^{++}_{\mathbb{R},\pi}$, which is minimal in its $W$-orbit $W\mu$, the point $z_{\mu}$
 is a limit point for the set $\{z_{\alpha}| \alpha\in(\Delta^{\ima}_{\pi})^+\}$.
 Fix $\mu$ as above. 
 
 For any $\beta\in\pi$ one has $\langle \mu,\beta^{\vee}\rangle\in\mathbb{Z}$
 (since $\mu\in Q(\pi)$), so the condition $s_{\beta}\mu\not\leq \mu$  implies
 $\langle \mu,\alpha^{\vee}\rangle \leq 0$. Hence $\mu\in K\cap Q^+(\pi)$. 
 Take $\nu_0\in \mathring{K}\cap (\theta+Q^+(\pi))$ as in~\ref{Indnotat}.
 For any $s\in\mathbb{N}$ the element 
 $\mu_s:=s^2\mu+s\nu_0+\theta$ lies in $ (\theta+Q^+(\pi))$. For $s>>0$ one has 
$\mu_s \in {K}$ and thus   $\mu_s\in (\Delta^{\ima}_{\pi})^+$ by~(\ref{Kim}) .
From the formula
$$\mathbb{R}_{\geq 0}\mu_s=\mathbb{R}_{\geq 0} (\mu+\frac{\nu_0}{s}+\frac{\theta}{s^2})$$
 we see that $z_{ \mu}$ is the limit point for $z_{\mu_s}$.
This completes the proof.
\end{proof}

\subsubsection{}
\begin{lem}{corrotbasisInd}
One has  $(\Delta^{\ima}_{\pi})^+\subset \Delta^+(\Sigma')$.
\end{lem}
\begin{proof}
 We retain notation of~\ref{Indnotat}.   Recall that  $\langle \beta,h'\rangle=1$ for all $\alpha\in\Sigma'$.
For any  $\beta=\sum\limits_{\alpha\in\Sigma'} k_{\alpha}\alpha\in\Delta^+(\Sigma')$ we have
$$\langle z_{\beta},h'\rangle=
\frac{\sum\limits_{\alpha\in\Sigma'} k_{\alpha}}{\sqrt{\sum\limits_{\alpha\in\Sigma'} k_{\alpha}^2}}\geq 1.$$
Therefore $|\langle z_{\beta},h'\rangle|\geq 1$ for any $\beta\in\Delta$ .
By~\Lem{lemV+} this implies  $|\langle z_{\mu},h'\rangle|\geq 1$
for any non-zero $\mu\in Q^{++}_{\mathbb{R},\pi}$. Since $Q^{++}_{\mathbb{R},\pi}$ is a convex cone we have
either $\langle z_{\mu},h'\rangle\geq 1$ for all non-zero $\mu\in Q^{++}_{\mathbb{R},\pi}$
 or
$\langle z_{\mu},h'\rangle\leq -1$ for all non-zero $\mu\in Q^{++}_{\mathbb{R},\pi}$.
By the assumption~(\ref{Sigma'assm})
$\langle z_{\beta},h'\rangle>0$ for some $\beta\in (\Delta^{\ima}_{\pi})^+\subset Q^{++}_{\mathbb{R},\pi}$.
Hence $\langle\mu,h\rangle>0$ for all non-zero $\mu\in Q^{++}_{\mathbb{R},\pi}$, 
so $-\Delta^+(\Sigma')\cap Q^{++}_{\mathbb{R},\pi}$ is empty and 
 $(\Delta^{\ima}_{\pi})^+\subset \Delta^+(\Sigma')$ as required.
\end{proof}

\subsubsection{}
\begin{lem}{lemrootbasisInd}
The set $-\Delta^+(\Sigma')\cap \Delta^+_{\an}$ is finite.
\end{lem}
\begin{proof}
In light of~\Cor{corimaginary} it is enough to check that
$-\Delta_{>0}(h')\cap (\Delta^{\re}_{\pi})^+$  is finite.

Let $\Delta^{\vee}_{\pi}$ be the dual root system for $\Delta_{\pi}$; this system has 
the set of simple roots
$\{\alpha^{\vee}\}_{\alpha\in\pi}$  and the Cartan matrix which is the transpose of $B_{\pi}$.

We claim  that the map $\beta\mapsto \beta^{\vee}$ is a bijection
between $(\Delta^{\re}_{\pi})^+$ and $((\Delta^{\vee}_{\pi})^{\re})^+$. 
Take $\beta\in (\Delta^{\re}_{\pi})^+$. Then $\beta=w\beta_0$ for 
some $\beta_0\in \pi$ and $w\in W$. By~\ref{l(w)} one has
$\beta^{\vee}=w\beta_0^{\vee}$. Recall that $W$ is the Coxeter group
generated by $s_{\alpha}$, $\alpha\in\pi$ and is isomorphic
to the Coxeter group generated by   $s_{\alpha^{\vee}}$, $\alpha\in\pi$.
This gives
$$w\beta_0\in\Delta_{\pi}^+\ \ \Longleftrightarrow\ \ 
\ell(ws_{\beta_0})>\ell(w)
\ \ \Longleftrightarrow\ \ w\beta_0^{\vee}\in (\Delta^{\vee}_{\pi})^+.$$
where $\ell(w)$ stands for the length of $w$.
Therefore $\beta^{\vee}\in ((\Delta^{\vee}_{\pi})^{\re})^+$. 
Since $(\Delta_{\pi}^{\vee})^{\vee}=\Delta_{\pi}$, the map
$\beta\mapsto \beta^{\vee}$ is a bijection
between $(\Delta^{\re}_{\pi})^+$ and $((\Delta^{\vee}_{\pi})^{\re})^+$.

Take $\beta\in  (\Delta^{\re}_{\pi})^+$.
We write
$\beta^{\vee}=\sum\limits_{\alpha\in\pi} k_{\alpha}\alpha^{\vee}$
and set $\htt^{\vee}\beta:=\sum\limits_{\alpha\in\pi} k_{\alpha}$.

 Recall that
$\nu_0\in (\Delta^{\ima}_{\pi})^+$ is such that
$m_{\alpha}:=\langle \nu_0,\alpha^{\vee}\rangle<0$ for any $\alpha\in\pi$.
Since $\nu_0, s_{\beta}\nu_0$ lie in 
$(\Delta^{\ima}_{\pi})^+\subset \Delta^+(\Sigma')$, we have
$\langle s_{\beta}\nu_0,h'\rangle >0$.
Since $\langle \mu,\alpha^{\vee}\rangle\in\mathbb{Z}$ for any  $\alpha\in\pi$ and $\mu\in \Delta$,
we have $m_{\alpha}\leq -1$.  Then
$$0<\langle s_{\beta}\nu_0,h'\rangle=\langle \nu_0+m\beta, h'\rangle
=\langle \nu_0, h'\rangle+m\langle \beta, h'\rangle$$
where
$$m:=-\langle \beta^{\vee}, \nu_0\rangle =-\sum\limits_{\alpha\in\pi} k_{\alpha} \langle \nu_0,\alpha^{\vee}\rangle=-\sum\limits_{\alpha\in\pi} k_{\alpha} m_{\alpha}
\geq \htt^{\vee} \beta.$$

If $\beta\in (-\Delta_{>0}(h')\cap (\Delta^{\re}_{\pi})^+)$, then
$\langle \beta,h'\rangle\leq -1$, so $\htt^{\vee} \beta< \langle \nu_0,h'\rangle$.
Since the set
$\{\alpha\in (\Delta^{\re}_{\pi})^+|\ \htt^{\vee}\alpha<s\}$ is finite for any $s$,
the set   $-\Delta_{>0}(h')\cap (\Delta^{\re}_{\pi})^+$ is finite.
\end{proof}

\subsubsection{}
Let $\Sk_{\pi}$ be the component of $\Sk$ corresponding 
to $\fg(B_{\pi})$ (one has $v'=(\tilde{\fh},\tilde{a},\tilde{b},\ol{0})$ where
the triple $(\tilde{\fh},\tilde{a},\tilde{b})$ is described in~\ref{tildehab}). 
Combining Lemmatta~\ref{lemfindif} and~\ref{lemrootbasisInd}
we conclude that $\Sk_{\pi}$
contains a vertex with the set of simple roots
$\pi'$ we have
$\Delta^+(\pi')=\Delta^+(\Sigma')\cap \Delta_{\pi}$.
Since $W$ acts transitively on $\Sk_{\pi}$ 
we have $\pi'=w^{-1}\pi$ for some $w\in W$. This gives
$w^{-1}\pi\subset \Delta^+(\Sigma')$ that is
$\pi \subset \Delta^+(w\Sigma')$.
Then $\Sigma'':=w\Sigma'$ is a root basis with the property
\begin{equation}\label{953}
\pi\subset \Delta^+(\Sigma'').\end{equation}

In the pure anisotropic case $\Sigma_{\pr}=\Sigma$, so~(\ref{953}) implies
$\Sigma=\Sigma''$, so  $\Sigma'=w^{-1}\Sigma=\Sigma_{w^{-1}v}$
as required.

Consider  the remaining case $Q^{\pm}(m,n,t)$.
Fix $v$ such that $\Sigma=\{\beta_i\}_{i=1}^3$ consists of isotropic roots, see~\ref{Qmnpcon}. Then $\Sigma_{\pr}=\pi=\{\alpha_i\}_{i=1}^3$ where
$\alpha_j=\sum\limits_{i=1}^3 \beta_i-\beta_j$.
If $\Sigma\subset \Delta^+(\Sigma'')$, then 
$\Sigma=\Sigma''$. 
Assume that $\Sigma\not\subset \Delta^+(\Sigma'')$, for example,
 $\beta_1\not\in \Delta^+(\Sigma'')$.
For $v_1:=r_{v_1}(v)$ we have
$\Sigma_{v_1}=\{-\beta_1,\beta_1+\beta_2,\beta_1+\beta_3 \}$.
Since $\beta_1+\beta_2,\beta_1+\beta_3\in \Sigma_{\pr}$
we have $\Sigma_{v_1}\subset \Delta^+(\Sigma'')$
so $\Sigma''=\Sigma_{v_1}$ and
$\Sigma'=w^{-1}\Sigma_{v_1}=\Sigma_{w^{-1}v_1}$. 
 This completes the proof of~\Prop{propiii}.\qed

\subsection{Proof of~\Cor{corbase}}
The following definition is given in Section 3 of~\cite{S}.

\subsubsection{}
\begin{defn}{}  A linearly independent set \(\Sigma'\) of roots is called a \emph{base} of $\Delta$ if one can 
find $e'_{\pm\beta}\in \fg_{\pm \beta}$ for each $\beta\in \Sigma'$ such that 
$\fh$ and the elements
$\{e'_{\pm\beta}\}_{\beta\in \Sigma'}$ generate $\fg$, and for any 
$\alpha\neq\beta\in\Sigma'$ we have $[e'_{-\alpha},e'_{\beta}]=0$. 
\end{defn}

\subsubsection{}
The sets $\Sigma_u$ and $-\Sigma_u$ are bases for any $u\in\Sk(v)$.
The PBW theorem implies that any base is a root basis. Using~\Thm{thmiii} 
we conclude that the set of bases coincides with the set of root bases.

\subsubsection{}
Recall that the Cartan datum at $u\in\Sk(v)$ (=the Cartan datum of
$\Sigma_u$) is the pair $(A_u,p_u)$, where $A_u$ is the
Cartan matrix and  $p:X\to\mathbb{Z}_2$ is the parity function.

Let $\fg'$ be a Kac-Moody superalgebra
  with the Cartan subalgebra $\fh'$ and the set of simple roots $\Sigma'$.
  Assume that 
   $\fg\iso\fg'$  and that this isomorphism  maps $\fh$ to $\fh'$. By~\Thm{thmiii}
   the preimage of $\Sigma'$     in $\Delta$ equals to $\Sigma_u$
   or $-\Sigma_u$ for some $u\in\Sk(v)$. Clearly, the Cartan datum of
   $\Sigma'$ is $D$-equivalent to the Cartan datum of $\Sigma_u$ up to a permutation
   of indexes. One has $u=wv_1$ where $w\in W$ and $v_1\in\Sp(v)$.    
 The vertices $u$ and $v_1$ have the same Cartan datum and the Cartan datum
 at $v_1$ is obtained from the 
 Cartan datum at $v$ by a sequence of isotropic reflections. This establishes~\Cor{corbase}.

\section{Example: $Q^{\pm}(m,n,t)$}\label{Qmnp}
In this section we describe the root system of any root algebra 
for the components $Q^{\pm}(m,n,t)$. 
These components are of rank three and we set $X=\{x_1,x_2,x_3\}$.

\subsection{Construction}\label{Qmnpcon}
We fix a vertex $v$ with $p(x_1)=p(x_2)=p(x_3)=\ol{1}$. The Cartan matrix is
	 $$A_v=\begin{pmatrix}
			0 & a& 1\\
			1 & 0& b\\		
			c & 1& 0\\
		\end{pmatrix}$$
where 
	$$1+a+\frac{1}{c}=-m ,\ \ 1+b+\frac{1}{a}=-n,\ \ 1+c+\frac{1}{b}=-t$$
and $m,n,t$ are positive integers satisfying $mnt>1$.
One has   $-1<a,b,c<0$ for $Q^+(m,n,t)$ and $a,b,c\leq -1$ for $Q^-(m,n,t)$.	
Since $mnt>1$ without loss of generality we will assume that $a\not=-1$.

\subsubsection{}\label{spineQmnt}
Let $\beta_1,\beta_2,\beta_3$ be the elements of $\Sigma:=\Sigma_v$. Set
$\delta:=\sum\limits_{i=1}^3\beta_i$ and	 $\alpha_i:=\delta-\beta_i$.

Set $v_i:=r_{x_i}(v)$. The set $\Sigma_{v_i}$ contains only one odd root $-\beta_i$ and two even roots
$\alpha_j,\alpha_k$ for $j,k\not=i$. Therefore the spine consists of four vertices $v,v_1,v_2,v_3$ and
	$$\Sigma_{\pr}=\pi=\{\alpha_i\}_{i=1}^3.$$

\subsubsection{Remark}\label{SpDvQmnp}
	Note that all parity functions $p_v,p_{v_1},p_{v_2},p_{v_3}$ are different ($p_{v_i}(x_j)=\delta_{ij}$
	and $p_v\equiv \ol{1}$), so $\Sp^D(v)$ is trivial.

\subsubsection{}\label{deltaalphai}
Since $\det A\not=0$ the set $\{\beta_i\}_{i=1}^3$ forms a basis of $\fh^*$. Therefore
$\pi=\{\alpha_i\}_{i=1}^3$ also forms a basis of $\fh^*$. We denote by $Q^+(\pi)$ the semilattice spanned
by  $\pi$.
Note that $2\delta=\sum\limits_{i=1}^3 \alpha_i$.

 \subsubsection{}\label{gprforQ}
 The roots $\alpha_i$ are even. The matrix $B_{\pi}=\langle \alpha_i^{\vee},\alpha_j\rangle$ takes the form
		$$B_{\pi}=\begin{pmatrix}
			2 & -n& -n\\
			 -m& 2& -m\\		
			-t& -t & 2\\
		\end{pmatrix}$$
Note that $\pi$ is linearly independent and
$B_{\pi}$ is symmetrizable. By~\ref{gpr}, $\fg_{\ol{0}}$ contains the Kac-Moody algebra $\fg(B_{\pi})$.

We define a partial order on $\fh^*$ by $\mu\leq \nu$ if 
		$\nu-\mu\in\mathbb{N}\pi$.
		
\subsection{}
\begin{lem}{lemQ}
\begin{enumerate}
\item
$(\bigcap\limits_{u\in \Sp(v)} Q^+_u)_{\ol{0}}=Q^+(\pi),\ \  (\bigcap\limits_{u\in \Sp(v)} Q^+_u)_{\ol{1}}=\delta+Q^+(\pi)$.
\item  $\pm\delta\in\Delta$.
\end{enumerate}
\end{lem}
\begin{proof}
For (i) the embeddings  $\supset$ follow from the fact that $\pi\subset  \bigcap\limits_{u\in \Sp(v)} \Delta^+_u$ .
For the inverse embedddings 
recall that $\{\alpha_i\}_{i=1}^3$  forms a basis of $\fh^*$. 
Take $\mu=\sum\limits_{i=1}^3 k_i\alpha_i\in \bigcap\limits_{u\in \Sp(v)} Q^+_u$. Then
$$\mu=-2k_1\beta_1+(k_1+k_2)\alpha_2+(k_1+k_3)\alpha_3.$$
Recall that $\Sp(v)=\{v,v_1,v_2,v_3\}$.
Since $\Sigma_{v_1}=\{-\beta_1,\alpha_2,\alpha_3\}$, we have 
$2k_1, k_1+k_2,k_1+k_3\in\mathbb{N}$. Similarly, using $v_2$ and $v_3$, we obtain $2k_2,2k_3\in\mathbb{N}$.
This gives $\mu\in Q^+(\pi)$ if $p(\mu)=\ol{0}$ and $\mu\in \delta+Q^+(\pi)$ if $p(\mu)=\ol{1}$.
This implies (i).

For (ii) recall that $-\Delta=\Delta$ (see~\ref{-Delta}).
View $\fg$ as a module over the copy of $\fsl(1|1)$ corresponding to $\beta_1$. Since
		$$\langle a_v(x_1),\beta_2+\beta_3\rangle=a_{12}+a_{13}=a+1\not=0,$$
		the set $\Delta$ contains  $\beta_2+\beta_3-\beta_1$ or $\delta=\beta_2+\beta_3+\beta_1$. However,
		$\beta_2+\beta_3-\beta_1\not\in\Delta$, so
		$\delta\in\Delta$.   
\end{proof}

\subsection{}
\begin{thm}{thmQmnt}
Let $\fg$ be any root algebra for the component $Q^{\pm}(m,n,t)$.
\begin{enumerate}
\item 
$\bigcap\limits_{u\in \Sp(v)} Q^+_u=\{\mu\in Q_v|\ 2\mu\in Q^+(\pi)\}$ and 
$Q^{++}_{\mathbb{R}}=Q^{++}_{\mathbb{R},\pi}$.

\item 
One has $\Delta(\fg)=\Delta_{\an}\coprod\Delta_{\is}\coprod \Delta(\fg)^{\ima}$, where
$\Delta_{\an}=W\pi$,  $\Delta_{\is}=W(-\Sigma_v\cup\Sigma_v)$, 
$\Delta(\fg)^{\ima}=\Delta^{\ima}$ and 
$$(\Delta^{\ima})^+=W\{\mu\in (\bigcap\limits_{u\in \Sp(v)} Q^+_u)|\ \ \langle \mu,\alpha^{\vee}\rangle\leq 0\ \text{ for all }\alpha\in\pi\}.$$
Moreover, $(\Delta^{\ima})^+_{\ol{0}}=(\Delta^{\ima}_{\pi})^+$.
\end{enumerate}
\end{thm}
\begin{proof}
Recall that any vertex in $\Sk(v)$ can be presented as $wu$ for $w\in W$ and $u\in\Sp(v)$.
\Lem{lemQ} (i) implies the first formula in (i) and the formula
$$Q^{++}_{\mathbb{R}}=\mathbb{R}_{\geq 0}\bigcap\limits_{u\in \Sk(v)} Q^+_u =\mathbb{R}_{\geq 0}
\bigcap\limits_{w\in W}(w(\bigcap\limits_{u\in \Sp(v)} Q^+_u ))=\mathbb{R}_{\geq 0}
\bigcap\limits_{w\in W}(wQ^+(\pi))=Q^{++}_{\mathbb{R},\pi}.$$
This gives (i). In (ii)
the formulae for $\Delta_{\an}$ and $\Delta_{\is}$
 follow from~\ref{l(w)} and the formula 
 $\Delta(\fg)=\Delta_{\an}\coprod\Delta_{\is}\coprod \Delta(\fg)^{\ima}$
follows from (\ref{2alpha}) and~\ref{2bx}.
It remains to verify that $\Delta(\fg)^{\ima}=\Delta^{\ima}$ and
the formula for $(\Delta^{\ima})^+$.
 
 By~\Lem{lemQ}, $(\Delta(\fg)^{\ima})^+_{\ol{0}}\subset Q^+(\pi)$ and
 $(\Delta(\fg)^{\ima})^+_{\ol{1}}\subset \delta+Q^+(\pi)$.
 As in~\ref{Indnotat} we denote by $V$ the $\mathbb{R}$-span of $\pi$ and set 
 $$K:=\{\mu\in V|\ \langle \mu,\alpha^{\vee}\rangle\leq 0\ \text{ for all }\alpha\in\pi\},\ \ \ K_0:=Q^+(\pi)\cap K,\ \ 
 K_1:=(\delta+Q^+(\pi))\cap K.$$
 Using~\Lem{Winvariant}  and~\ref{Deltapi} we obtain 
$$(\Delta^{\ima}_{\pi})^{\pm}\subset (\Delta(\fg)^{\ima})^{\pm}_{\ol{0}}\subset \pm W K_0,\ \ \ (\Delta(\fg)^{\ima})^{\pm}_{\ol{1}}\subset \pm W K_1.$$
By~\cite{Kbook}, Lemma 5.3, one has
$(\Delta(\fg)^{\ima}_{\pi})^{\pm}=\pm W K_0$. This gives 
$$(\Delta^{\ima})^{\pm}_{\ol{0}}=(\Delta^{\ima}_{\pi})^{\pm}= \pm W K_0.$$
It remains to verify that $\pm K_1\subset\Delta(\fg)^{\ima}$.
First, let  us show that
		\begin{equation}\label{K1Delta}
	\pm K_1	\subset\Delta(\fg).
		\end{equation}
		
		The proof is similar to one in~\cite{Kbook}, Lemma 5.3.  Let 
		$\alpha$
		be a minimal element in $K_1$ which does not lie
		in $\Delta(\fg)$. 
		Since $\delta\in  \Delta\subset \Delta(\fg)$ the set 
		$$\{\beta'\in \Delta(\fg)|\ \delta\leq \beta'\leq \alpha\}$$
		is non-empty; let $\beta$ be a maximal element in this set. We have
		$$\alpha =\delta+\sum\limits_{i=1}^3 k_i\alpha_i,\ \ \ \ \beta=\delta+\sum\limits_{i=1}^3 m_i\alpha_i=\sum\limits_{i=1}^3 (m_i+\frac{1}{2})\alpha_i$$
	for some coefficients $k_i,m_i\in\mathbb{N}$ with
		  $m_i\leq k_i$. Set 
$$R:=\{i| \  m_i\not=k_i\},\ \ \ \beta':=\sum\limits_{i\in R} (m_i+\frac{1}{2})\alpha_i.$$
Since $\alpha\not\in \Delta(\fg)^{\ima}$, the set
 $R$ is non-empty. Take $i\in R$. 
The root spaces $\fg_{\pm\alpha_i}$ generate a copy of 
$\fsl_2$ which we denote by $\fsl_2(\alpha_i)$. 
The algebra $\fsl_2(\alpha_i)$ acts locally finitely in the adjoint representation of $\fg$;
thus  $\langle \beta,\alpha_i^{\vee}\rangle< 0\ \ \Longrightarrow\ \ \beta+\alpha_i\in\Delta(\fg)$. Using the maximality of $\beta$ 
we obtain
\begin{equation}\label{langleQmnt}
\langle \beta,\alpha_i^{\vee}\rangle\geq 0\ \text{ for all } i\in R.
\end{equation} 
Since $\beta'=\beta-\sum\limits_{j\not\in R} (m_j+\frac{1}{2})\alpha_j$
and $\langle \alpha_j,\alpha_i^{\vee}\rangle<0$ for $j\not\in R$ and $i\in R$,
the formula (\ref{langleQmnt}) implies
$$\langle \beta',\alpha_i^{\vee}\rangle>0\ \text{ for all } i\in R.$$
Recall that $\beta'$ is a positive linear combination of $\alpha_i$ for $i\in R$. Using
 Theorem 4.3  in~\cite{Kbook} we conclude that the Dynkin diagram of $R$ is of finite type.
One has
$$\alpha-\beta=\sum\limits_{i\in R} (k_i-m_i)\alpha_i.$$
Combining  (\ref{langleQmnt}) with the condition $\alpha\in K_1$  we obtain
$$\langle \alpha-\beta,\alpha_i^{\vee}\rangle \leq 0\ \text{ for all } i\in R.$$
By above, the Dynkin diagram of $R$ is of finite type. Since $\alpha-\beta$ is
a non-negative  linear combination of $\alpha_i$ for $i\in R$, 
Theorem 4.3  in~\cite{Kbook} gives $\alpha=\beta$, a contradiction.
Therefore $K_1\subset \Delta(\fg)$. 
The proof of $-K_1\subset \Delta(\fg)$ is similar.
This  completes the proof of~(\ref{K1Delta}).

It remains to verify that $K_1\cap\Delta(\fg)^{\re}=\emptyset$.
Take $\alpha\in K_1$. Note that
 $2\alpha\in K_0$, so $2\alpha\in\Delta(\fg)^{\ima}$.
Hence $\alpha\not\in\Delta(\fg)^{\re}$. This implies
$\alpha\in (\Delta(\fg)^{\ima})^+$.
	\end{proof}

\subsubsection{}
\begin{cor}{corQmnt}
\begin{enumerate}
\item $\Delta(\fg)=\Delta$ (cf.~\cite{Kbook}, Corollary 5.12).	
\item  
$(\Delta^{\ima})^+=Q^{++}_{\mathbb{R}}\cap Q_v$.
\item 
The element $\delta$ is the unique minimal element in 
 $(\Delta^{\ima})^+_{\ol{1}}$ and $\langle \delta,\alpha^{\vee}\rangle<0$ for all $\alpha\in\pi$.
\end{enumerate}
\end{cor}
\begin{proof}
\Thm{thmQmnt} and~\Lem{Winvariant} imply (i) and (ii).

For (iii) recall that  $2\delta=\sum\limits_{i=1}^3 \alpha_i$. The formula for $B_{\pi}$ (see~\ref{gprforQ}),
gives
$$\langle \delta,\alpha_1^{\vee}\rangle=1-n,\ \ \ \langle \delta,\alpha_2^{\vee}\rangle=1-m,\ \ \ 
\langle \delta,\alpha_3^{\vee}\rangle=1-t$$
as required.
\end{proof}

\section{\Thm{corintrospine} and the formulae~(\ref{suppconnected}), (\ref{eqimagine}) in the type (Ind)}
Recall that~\Thm{corintrospine} and
 the formulae~(\ref{suppconnected}),
(\ref{eqimagine})  hold for the types (Fin) and (Aff) (see~\Cor{cor264}).
In this section we complete the proof of~\Thm{corintrospine}
and the formulae~(\ref{suppconnected}), (\ref{eqimagine}) in the remaining type (Ind).

We set $Q^{++}=\bigcap\limits_{u\in \Sk(v)} Q^+_u$ (i.e.,  $Q^{++}=Q^{++}_{\mathbb{R}}\cap Q_v$)
and $Q^{++}_{\pi}:=\bigcap\limits_{w\in W} (wQ^+_{\pi})$ (in the notation of~\ref{Deltapi},
$Q^{++}_{\pi}$ is the image of $\tilde{Q}^{++}$).

\subsection{Purely anisotropic components}\label{61}
The arguments in~\cite{Kbook}  5.2--5.4 are valid for this case. 
As a result,
$(\Delta^{\ima})^+$  is the union of $W$-orbits and each orbit has 
a unique representative in the set 
$$
K:=\{\mu\in  Q^+_v|\ \supp\mu\ \text{ is connected and }\ 
\langle\mu,\alpha^{\vee}\rangle \leq 0\ \ \text{for all }
\alpha\in\Sigma_v\}$$
where for $\mu=\sum\limits_{\alpha\in\Sigma_v} k_{\alpha}\alpha$
we set $\supp\mu:=\{\alpha\in\Sigma_v|\ k_{\alpha}\not=0\}$ (and $\supp\mu$
is connected if the corresponding Dynkin diagram is connected).

Since   for $\alpha\in\pi$ either $\alpha$ or $\alpha/2$ lies in $\Sigma_v$, 
we have $2\mu\in Q^+(\pi)$ for all $\mu\in Q^+_v$.

\subsection{}
\begin{cor}{corimaginary2}
Let $\Sk(v)$ be an  indecomposable Kac-Moody component of type (Ind).
\begin{enumerate}
\item The spine and $\pi$ are finite.
 One has $\bigcap\limits_{u\in \Sp(v)} C(\Sigma_u)=C(\pi)$.

\item 
 $(\Delta^{\ima})^+=Q^{++}$ if the component is not purely anisotropic.
\item $(\Delta^{\ima})^+=\{\mu\in Q^{++}|\ \supp \mu  \text{ is connected}\}$.
\item $(\Delta^{\ima})^+=\{\alpha\in Q_v| \ 2\alpha\in (\Delta^{\ima})^+_{\pi}\}$.
\item
If $\alpha\in (\Delta^{\ima})^+$, then $(\mathbb{R}_{\geq 0}\alpha\cap Q_v)\subset (\Delta^{\ima})^+$.

\item For an  isotropic root $\alpha$ the following conditions are equivalent:
\begin{itemize}
\item[(a)] $\alpha\in \Sigma_u$ for some $u\in\ Sp(v)$;
\item [(b)] $\alpha\not\in \bigl(-C(\pi)\cup C(\pi)\bigr)$.
\end{itemize}

\item One has $Q^{++}_{\mathbb{R}}=\bigcap\limits_{w\in W} wC(\pi)=\bigcup\limits_{w\in W} wK_{\mathbb{R}}$, 
where 
$$K_{\mathbb{R}}=\{ \mu\in C(\pi)|\ \langle \mu,\alpha^{\vee}\rangle \leq 0\ 
\text{ for all }\alpha\in\pi\}.$$
Moreover, any element in $Q^{++}_{\mathbb{R}}$ is $W$-conjugated to a unique element in $K_{\mathbb{R}}$.
\end{enumerate}
\end{cor}
\begin{proof}
Combining Section~\ref{Qmnp} and~\ref{61} we obtain (i)--(iv); (v) follows from (iv), (vi) follows
from (i) and (vii) follows from (i) and~\Lem{Winvariant}.
\end{proof}

\subsection{Remark}
From the van de Leur-Hoyt classification it follows that 
for a Kac-Moody component which is not purely anisotropic we have
 $\Delta^{\ima}_{\pi}=\Delta^{\ima}_{\ol{0}}$ except for the affine component
 $D(m+1|n)^{(2)}$, where $\Delta^{\ima}_{\pi}\subsetneq\Delta^{\ima}_{\ol{0}}$
 (since $\delta\in\Delta^{\ima}_{\ol{0}}\setminus\Delta^{\ima}_{\pi}$). 
 Note that the component $C(3)^{(2)}=D(1|2)^{(2)}$ is purely anisotropic.

\section{Skeleton  and spine as  graphs}\label{spinegraph}
We  view the skeleton as a graph with the set of vertices
$\Sk(v)$ and the marked edges $v \overset{r_x}\longleftrightarrow u$.
This graph does not have loops and the  edges adjacent to each vertex 
have different marks. This graph is studied in Section 5.3 of~\cite{GHS}.

Fix an  ordered set $(x_{i_1},\ldots , x_{i_s})$ 
of elements in $X$. For each vertex $u$ there exists at most one path 
marked by $(r_{x_{i_1}},\ldots , r_{x_{i_s}})$ starting at $u$.
We call two paths $v_1\to  v_2$ and  $u_1\to u_2$
 {\em namesakes} if they are marked by the same ordered set.

The spine is a subgraph of the skeleton. The action of $W$ described in~\ref{l(w)} 
preserves the marking of the edges;  each orbit contains exactly one vertex from $\Sp(v)$.

\subsection{Examples}
The skeleton of a Kac-Moody component is a regular graph (the degree of each vertex equals to the cardinality 
of $X$). The degree of a  vertex $u$  in the spine is equal to the number 
of  isotropic roots in $\Sigma_u$.

In the non-isotropic components the spine is trivial and the skeleton identifies the Cayley graph of 
the Weyl group.

 \subsubsection{Example}
 The skeleta of the Kac-Moody components $A_n$  
 (corresponding to  $\fsl_{n+1}$) and $A(k|n-1-k)$ 
 (corresponding to $\fsl(k+1|n-k)$ if $2k+1\not=n$ and to $\fgl(k+1|k+1)$ if $2k+1=n$)
identify with the Cayley graph of $S_{n+1}$.   For $A(k|n-1-k)$ the Weyl group
is $S_{k+1}\times S_{n-k}$ and the spine contains $\binom{n+1}{k}$ 
vertices. For example, for $A(0|n-2)$  the spine contains $n+1$ vertices and is of the form
 $\cdot \longleftrightarrow\ \cdot\ldots\cdot\longleftrightarrow \cdot \longleftrightarrow\cdot$,
 with the arrows marked by $r_1,\ldots,r_{n}$.

 \subsubsection{Example}\label{346}
By~\ref{attain}, $\Sk(v)$ classifies the attainable sets of simple roots. Thus
 $\Sk(v)$ is finite if $\Delta$ is finite. 
  The converse does not hold: for example, for the Cartan matrix $\begin{pmatrix} -1 & -1 \\-1 &-1\end{pmatrix}$
  all elements of $X$ are not reflectable at $v$, so $\Sk(v)=\{v\}$, but $\fg^{\ttC}=\fg^{\U}$ coincides
  with the ``half-baked'' algebra (see~\ref{Liealgcase}) and $\Delta$ is infinite.
  
 \subsubsection{Example}
The spines of $G(3)^{(1)}$ and $F(4)^{(1)}$  are given in
 Appendix (we use the notation of~\cite{Ksuper} for the roots).

\subsection{The groups ${\Sk}^D(v)$ and  ${\Sp}^D(v)$}\label{SpDvset} 
We denote by $V_a$ the subspace of $\fh$ spanned by
$\{a_v(x)\}_{x\in X}$ and by $V_b$ the subspace of $\fh^*$ spanned by
$\{b_v(x)\}_{x\in X}$.  

Take $u\in\Sk(v)$. 
We introduce $\sigma_a^u\in GL(V_a)$, $\sigma_b^u\in GL(V_b)$  via the formulae
$$\sigma_a^u(a_v(x)):=a_u(x),\ \ \ \sigma_b^u(b_v(x)):=b_u(x)\ \text{ for all }x\in X.$$

For any $u\in\Sk(v)$ we set 
$$\begin{array}{l}
\Sk^D(u):=\{u'\in\Sk(v)|\ \text{the Cartan data at $u$ and $u'$ are $D$-equivalent}\},\\
\Sp^D(u):=\Sk^D(u)\cap \Sp(u).\end{array}$$
By~\ref{l(w)}, $W$ acts faithfully on $\Sk^D(u)$ and 
each orbit cointains a unique element of $\Sp^D(v)$.

\subsubsection{Remark}
For each $u\in\Sk(v)$ denote by $\mathfrak{b}_u$ the Borel subalgebra of $\fg:=\fg^{\ttC}$
generated by $\fg_{\alpha}$ with $\alpha\in\Sigma_u$. By~\ref{122},
for each $u\in\Sk^D(v)$ there exists an algebra automorphism of $\fg$
which maps $\fg_{b_v(x)}$ to $\fg_{b_u(x)}$; this automorphism
maps $\fg_{\beta}$ to $\fg_{\beta'}$ where  $\beta'=\sigma_b^u(\beta)$. This gives
$$u\in \Sk^D(v)\ \ \Longrightarrow\ \ \mathfrak{b}_v\cong \mathfrak{b}_u.$$

\subsubsection{}
\begin{lem}{lemSkDv}
Take $v'\in\Sk^D(v)$. Let $D$ be an invertible matrix such that $A^{v'}=DA^v$.
\begin{enumerate}
\item
For any path $\gamma: v\to u$ in $\Sk(v)$,  the skeleton
contains a unique path $\gamma': v'\to u'$ which is namesake for $\gamma$.
\item One has $u'\in\Sk^D(u)$  with $A^{u'}=DA^{u}$. 
\item If $\gamma$
lies in $\Sp(v)$, then $\gamma'$ lies in $\Sp(v')$.
\item For each $x\in X$ we have $b_{u'}(x)=\sigma_b^{v'} (b_{u}(x))$, 
$a_{u'}(x)=\sigma_a^{v'} (a_{u}(x))$.
\item If $\gamma_1,\gamma_2$ are two paths from
$v$ to $u$, and  $\gamma'_1: v'\to u'$, $\gamma'_2: v'\to u''$ are the paths
constructed in (i), then $u'=u''$.
\end{enumerate}
\end{lem}
\begin{proof}
We proceed by induction on the length of $\gamma$ which we denote by $\ell(\gamma)$.
If $\ell(\gamma)=1$, then 
$\gamma=r_x$ and  the assertions (i)-(iv) follow from the formula for the reflexions.
For the induction step we 
present $\gamma$ as  $\gamma=r_x\circ \gamma_1$
where $r_x: u_1\to u$ is the last edge in $\gamma$. 
Since  $\gamma_1: v\to u_1$ has length $m-1$, the skeleton
contains a unique path $\gamma_1': v'\to u_1'$ and 
(ii), (iii), (iv) holds for $u_1'$. Since $u'_1\in\Sk^D(u_1)$ the reflexion
$r_x: u_1\to u$ gives rise to $r_x: u_1'\to u'$ and assertions (ii)--(iv) 
hold.

For (v) note that (iv) implies $\Sigma_{u'}=\Sigma_{u''}$.
By~(\ref{Sigmau=Sigmav}) this gives $u'=u''$.
\end{proof}

\subsubsection{}
By~\Lem{lemSkDv} each  $v'\in\Sk^D(v)$ 
defines an automorphism  of the skeleton (as a marked graph);
we denote this automorphism as $\sigma^{v'}$.

For $u_1,u_2\in \Sk^D(v)$ we set
$$u_1*u_2:=\sigma^{u_1}(u_2).$$
By~\Lem{lemSkDv}, $u_1* u_2\in Sk^D(v)$. Moreover, if $A^{u_i}=D_i A^v$,
then $A^{u_1* u_2}=D_2D_1 A^v$.
Clearly,
$\sigma^v=\Id$, so 
$$v*u=u=\sigma^u(v)=u*v.$$

\subsubsection{}
\begin{lem}{}
The operation $*$ endows $\Sk^D(v)$ with a group structure with the unit $v$.
The group $\Sk^D(v)$ contains $\Sp^D(v)$ as a subgroup.
\end{lem}
\begin{proof}
Let us check the associativity.

Take $u_1,u_2,u_3\in\Sk^D(v)$ and the paths $\gamma_i: v\to u_i$;
we denote by $M_i$ the ordered set of marks for  $\gamma_i$.

The skeleton contains a  path
$\sigma^{u_2}(\gamma_3): u_2\to (u_2*u_3)$ which is marked by $M_3$;
this gives  the  path $\sigma^{u_2}(\gamma_3)\circ \gamma_2 :v\to (u_2*u_3)$
marked by $M_2\coprod M_3$  and 
the  path 
$$\sigma^{u_1}( \sigma^{u_2}(\gamma_3)\circ \gamma_2): u_1\to u_1*(u_2*u_3)$$
marked by $M_2\coprod M_3$. The composed path
\begin{equation}\label{path1}
\sigma^{u_1}( \sigma^{u_2}(\gamma_3)\circ \gamma_2)\circ\gamma_1: v\to u_1*(u_2*u_3)
\end{equation}
is marked by $M_1\coprod M_2\coprod M_3$.

On the other hand, the skeleton contains a  path  
$\sigma^{u_1}(\gamma_2): u_1\to (u_1*u_2)$ which is marked by $M_2$ 
and a path $\sigma^{u_1*u_2}(\gamma_3): u_1*u_2\to (u_1*u_2)*u_3$
which is marked by $M_3$.
The composed path
\begin{equation}\label{path2}
\sigma^{u_1*u_2}(\gamma_3)\circ \sigma^{u_1}(\gamma_2)\circ \gamma_1:
v\to  (u_1*u_2)*u_3
\end{equation}
is marked by the ordered set $M_1\coprod M_2\coprod M_3$.
Since the paths~(\ref{path1}), (\ref{path2}) have the same marks and
start at the same vertex $v$, these paths coincide and 
$$ u_1*(u_2*u_3)=(u_1*u_2)*u_3.$$

Now take any $u\in\Sk^D(v)$
and fix a path $\gamma: u\to v$ marked by a set $M$.
Applying~\Lem{lemSkDv}  to
 $v\in\Sk^D(u)$ we conclude that
 the skeleton contains a  path $\gamma':v\to u'$ which is marked by the same set
 $M$ and $u'$ lies in $\Sk^D(v)$.
 Let us show that 
  $$u'*u=u*u'=v.$$ 
 Indeed, the  path $\sigma^u(\gamma'): u\to u*u'$ is 
 marked $M$, so  $\sigma^u(\gamma')=\gamma$ and $u*u'=v$.
 On the other hand, taking $\gamma^{-1}: v\to u$ we obtain 
 $\sigma^{u'} (\gamma^{-1}): u'\to u'*u$. The paths
 $\gamma^{-1}:v\to u$ and  $(\gamma')^{-1}: u'\to v$ are marked by the same 
 set (the set $M$ with the reverse order), so  $\sigma^{u'} (\gamma^{-1}): u'\to u'*u$
 and $(\gamma')^{-1}: u'\to v$ are marked by the same 
 set. Hence  $u'*u=v$ as required.
 
 By above, $(\Sk^D(v), *)$ is a group. By~\Lem{lemSkDv} one has
 $u_1*u_2\in\Sp^D(v)$ for $u_1,u_2\in\Sp^D(v)$. This completes the proof.
\end{proof}

\subsection{}
\begin{thm}{thm1}
The map $u\mapsto \sigma_b^u$ induces a group monomorphism $\Sk^D(v)\to GL(V_b)$.
The image lies in the group of automorphisms of the 
Dynkin diagram of $\Sigma_{\pr}$, i.e. 
$$\operatorname{Dynkin}(\Sigma_{\pr}):=\bigl\{
\begin{array}{l}\psi\in \Aut(\Sigma_{\pr})|\ p(\psi(\alpha))=p(\alpha),\\
\ \langle \alpha^{\vee},\beta\rangle=
\langle \psi(\alpha)^{\vee},\psi(\beta)\rangle \text{ for all }\alpha,\beta\in\Sigma_{\pr}\end{array}\bigr\}.$$ 
\end{thm}
\begin{proof}
By~\Lem{lemSkDv} (iii), (iv)  we have $\sigma_b^{u_1}\not=\Id$ if $u_1\not=\Id$ and 
\begin{equation}\label{sigmabu1u2}
\sigma_b^{u_1*u_2}=\sigma_b^{u_1}\circ \sigma_b^{u_2}.
\end{equation}
This gives a monomorphism $\Sk^D(v)\to GL(V_b)$.

Take $v'\in\Sp^D(v)$ and $\alpha\in\Sigma_{\pr}$. 
Set $\alpha':=\sigma_b^{v'}(\alpha)$.
Let $u\in \Sp(v)$ and $x\in X$ be such that $\alpha=b_u(x)$. Note that $a^u_{xx}\not=0$.
Using~\Lem{lemSkDv} (ii), (iii) and (iv)
 we obtain $u':=\sigma^{v'}(u)\in\Sp(v)$ with 
 $$\sigma_b^{v'}(b_u(y))=b_{u'}(y)\ \ \text{ for all }y\in X$$
and $a^{u'}_{xx}\not=0$. Therefore $\alpha'=b_{u'}(x)\in\Sigma_{\pr}$.
 Therefore the map $v'\mapsto \sigma_b^{v'}$ induces
a group homomorphism $\iota:\Sp^D(v)\to  \Aut (\Sigma_{\pr})$.

Take $\beta=\sum\limits_{y\in X} m_y b_u(y)\in V_b$
and set $\beta':=\sigma_b^{v'}(\beta)$. Then
$$\beta'=\sigma_b^{v'}(\beta)=\sum\limits_{y\in X} m_y b_{u'}(y).$$
Since $\alpha^{\vee}=\frac{2}{a^{u}_{xx}} a_{u}(x)$ we have
$$\langle \alpha^{\vee},\beta\rangle=\sum\limits_{y\in X} m_y \frac{2a^{u}_{xy}}{a^{u}_{xx}}.$$
Similarly, $(\alpha')^{\vee}=\frac{2}{a^{u'}_{xx}} a_{u'}(x)$, so
$$\langle (\alpha')^{\vee},\beta'\rangle=
\frac{2}{a^{u'}_{xx}}\langle a_{u'}(x),\beta'\rangle=\frac{2}{a^{u'}_{xx}}
\sum\limits_{y\in X} m_y \langle a_{u'}(x),b_{u'}(y)\rangle=\sum\limits_{y\in X} m_y \frac{2a^{u'}_{xy}}{a^{u'}_{xx}}.$$
By~\Lem{lemSkDv}  (ii), $\frac{2a^{u}_{xy}}{a^{u}_{xx}}=\frac{2a^{u'}_{xy}}{a^{u'}_{xx}}$.
Hence $\langle \alpha^{\vee},\beta\rangle=\langle (\alpha')^{\vee},\beta'\rangle$ as required.
\end{proof}

\subsubsection{Remark}\label{type2Aut}
Denote by $\Aut(\Sigma_{\pr})$ the group
of automorphism of the set $\Sigma_{\pr}$. By~\Thm{thm1} the map
$u\mapsto \sigma_b^u$ induces a group homomorphism
$\Sp^D(v)\to \Aut (\Sigma_{\pr})$ which is injective 
if   $\mathbb{C}\Sigma_{\pr}=V_b$.

\subsection{}
\begin{lem}{lemWSkDv}
\begin{enumerate}
\item For $w\in W$ one has $wv\in\Sk^D(v)$ and $\sigma^{wv}(u)=wu$. Moreover,
 $\sigma_b^{wv}\in GL(V_b)$ coincides with the restriction of $w\in GL(\fh^*)$ to $V_b$.
 \item For $w_1,w_2\in W$ we have $w_1v* w_2v=(w_1w_2)v$, so 
  $Wv$ forms a  subgroup of $\Sk^D(v)$. 
  \item Identifying $W$ with $Wv$ we have
   $W\cap \Sp^D(v)=\Id$ and $\Sk^D(v)=W (\Sp^D(v))$.
  \end{enumerate}
\end{lem}
\begin{proof}
By~\ref{l(w)}, 
for $w\in W$ and $u\in\Sk^D(v)$  we have $wu\in\Sk^D(v)$ and $b_{wu}(x)=wb_u(x)$. In particular,
$\sigma^{wv}_b(b_v(x))=wb_v(x)$, so $\sigma^{wv}_b$ is the restriction of $w\in GL(\fh^*)$ to $V_b$.
For $u\in \Sk(v)$ we have
$\sigma^{wv}(u)=u'$ where 
$$\Sigma_{u'}=\sigma^{wv}_b(\Sigma_u)=w\Sigma_u=\Sigma_{wu}.$$
Hence $\sigma^{wv}(u)=wu$. This gives (i). For (ii) take $w_1,w_2\in W$. Then
$$\Sigma_{w_1v*w_2v}=w_1\Sigma_{w_2v}=w_1w_2\Sigma_v=\Sigma_{w_1w_2v}.$$
Hence $w_1v* w_2v=(w_1w_2)v$. This gives (ii).
Now (iii) follows from the fact that each $W$-orbit in $\Sk(v)$ has a unique intersection with
$\Sp(v)$ (see~\ref{l(w)}).
\end{proof}

\subsubsection{}
By~\cite{GHS}, Proposition 5.4.7 (4) 
one has $\Sk^D(v)=Wv\rtimes \Sp^D(v)$.

\subsection{Examples}\label{examples7}
By~\cite{GHS}, the group  $\Sp^D(v)$ is trivial if $\Sk(v)$ is of type (Fin) and differ from
 $A(n|n)$. \footnote{Recall that  $A(m|n)$ is defined for
 $m,n\geq -1$; one has $A(m|-1)=A_m$.
 We have $\fg^{\ttC}=\fsl(m+1|n+1)$
  if  $m\not=n$ and $\fg^{\ttC}=\fgl(m+1|m+1)$ for $m=n$.}

 \subsubsection{}\label{Ann}
 Consider the component $A(n|n)$.
The space $\fh^*$ is spanned by $\vareps_1,\ldots,\vareps_{n+1}, \delta_1,\ldots,\delta_{n+1}$.
We fix $v$ with
$$\Sigma_v=\{\vareps_1-\vareps_2,\ldots,\vareps_n-\vareps_{n+1},\vareps_{n+1}-\delta_1, \delta_1-\delta_2,\ldots,\delta_n-\delta_{n+1}\}.$$
In this case $\Sp^D(v)=\{v,u\}$ with
$$\Sigma_u=\{\delta_1-\delta_2,\ldots,\delta_n-\delta_{n+1},\delta_{n+1}-\vareps_1,\vareps_1
-\vareps_2,\ldots,\vareps_n-\vareps_{n+1}\}$$
and $A_u=-A_v$. 
Thus 
$$\Sp^D(v)=\{v,u\}\cong \mathbb{Z}_2.$$

The group $W$ identifies with $S_{n+1}\times S_{n+1}\subset GL(\fh^*)$, where
the first copy of $S_{n+1}$ 
stands for  the group of permutations of $\vareps_1,\ldots,\vareps_{n+1}$
and the  second copy of $S_{n+1}$ 
stands for  the group of permutations of $\vareps_1,\ldots,\vareps_{n+1}$.

Consider the involution $\xi\in GL(\fh^*)$ given by $\xi(\vareps_i)=\delta_i$.
 The group $\Sk^D(v)$ can be identified with the group $W\rtimes \{\xi,\Id\}\subset GL(\fh^*)$
 where $\sigma^u_b=\xi$.
 Note that the semidirect product $W\rtimes \{\xi,\Id\}$ is not direct.

 \subsubsection{}\label{Amnaff}
Take $n>0$.
 By~\cite{GHS}, 10.3 the skeleta of the affine Kac-Moody components $A_n^{(1)}$
  and $A(k|n-1-k)^{(1)}$ 
identify with the Cayley graph of the affine Weyl group
$W(A_n^{(1)})$.  For $A(k|n-1-k)^{(1)}$ the Weyl group
is $W(A_k^{(1)})\times W(A_{n-k-1}^{(1)})$. 

For $A(k|n-1-k)^{(1)}$  the group 
$\Sk^D(v)$ is described in~\cite{GHS}, 10.3.5. It is shown that for $k\not=-1,n$, the group
$W$ has an infinite index in $\Sk^D(v)$,
so $\Sp^D(v)$ is infinite.  Therefore $\Sp(v)$ is infinite.
Another proof of the fact $\Sp(v)$ is infinite appears in~\Cor{corinf}.
\footnote{Another description 
 of the spine for  $A(m|n)^{(1)}$ appears in~\cite{Mus}.}

\subsubsection{}
For the cases  $G(3)^{(1)}$ and $F(4)^{(1)}$ the group $\Sp^D(v)$ is trivial since in 
each case we have a Dynkin diagram appearing only once (the spines are given in Appendix).

\subsection{Remark}\label{Sk'}
Let $\cR_0$ be an admissible component of the root groupoid. 
Take a subset $X'\subset X$. For a vertex $v\in\cR_0$ we consider 
a vertex $v'$ which corresponds to the quadruple $(\fh,a',b',p')$
where  $a', b', p'$ are the restrictions of $a,b,p$ to $X'$.
Denote by $\cR'_0$ the component corresponding to $(\fh,a',b',p')$.
It is easy to see that $\cR'_0$ is admissible and 
 $\sum\limits_{\alpha\in \mathbb{Z}\Sigma_{v'}}\fg^{\ttC}_{\alpha}$ is the contragredient
algebra for $\cR'_0$. In particular, $\Delta'\subset \Delta$. 

If $\cR$ is fully reflectable, then $W'$ is a subgroup of $W$ and
$\Sk(v')$, $\Sp(v')$ are subgraphs of $\Sk(v)$ and $\Sp(v)$ respectively.
In general, $\Sp(v')$ is a subraph of $\Sp(v)$ (since $a_{xx}=0$ and $p(x)=\ol{1}$
implies reflectability of $x$ at $v$), but $\Sk(v')$ 
may not be a subgraph of $\Sk(v)$ and $W'$ may not be a subgroup of $W$,
see example~\ref{infinitepiS} below.

\subsection{Spine of a Kac-Moody component}
In the rest of this section we show how~\Thm{corintrospine} (iv)  can be deduced from~\Thm{thm1}.

Let $\Sp(v)$ be  an  indecomposable   Kac-Moody component of $\Sp$.
 Recall that  $\Delta$ is of  type  I  if $\mathbb{Q}\Delta_{\ol{0}}\not=\mathbb{Q}\Delta$. 
 By~\Cor{corinf} the spine is infinite in the type (Aff) I.

 \subsubsection{}
 \begin{lem}{lemspinefin}
 Let  $\Sp(v)$ be an  indecomposable  affine Kac-Moody component.
 \begin{enumerate}
 \item If $\Sp(v)$ is not of type $S(1|2;b)$,
 then the set $\mathcal{A}:=\{A_u\}_{u\in \Sp(v)}$ is finite.
 \item Assume  that $\mathcal{A}$ is finite. Then $\Sp^D(v)$ is finite if and only if $\Sp(v)$ is finite.
 \item  If $\Sp(v)$ is of type II, then the spine is finite.
 \end{enumerate}
\end{lem}
 \begin{proof}
Let $V\subset \fh^*$ be the $\mathbb{R}$-span of $\Delta$.
 Consider the symmetrizable case. By~\ref{bilinearform} for 
$u\in \Sk(v)$ the Cartan matrix $A_u$ is the Gram matrix of $\Sigma_u$.
In this case $(\Delta,\delta)=0$. By above,  the image of $\Delta$
 in $V/\mathbb{R}\delta$ is finite, so 
 the set of various Gram matrices of a given size which correspond to
 the vectors in $\Delta$ is finite. Thus
 the set  $\mathcal{A}$  is finite.
 
 Consider the case $q_n^{(2)}$. By~\cite{S}, Section 6
 we can choose $v$ in such a way that $a_{xy}\in \{0;\pm1 \}$ for $x\not=y$
 and $\sum_{y\in X} a_{xy}=0$ for all $x\in X$. Then all 
Cartan matrices $A_u$ for $u\in \Sk(v)$ satisfy the same conditions
($A_u=(a'_{xy})$ with  $a'_{xy}\in \{0;\pm1 \}$ for $x\not=y$
 and $\sum_{y\in X} a'_{xy}=0$ for all $x\in X$).  Clearly, the set of such matrices is finite.
 Since $S(1|2;b)$ and $q_n^{(2)}$ are the only non-symmetrizable affine components,
 this completes the proof of (i).

 For (ii) the part ``if'' follows from the definition of $\Sp^D(v)$.
 Assume that  $\mathcal{A}$  is finite.
 Consider the equivalence relation on $\Sp(v)$ given by $u_1\sim u_2$ if
 $u_2\in \Sp^D(u_1)$. The number of equivalence classes is finite since
 $\mathcal{A}$ is finite and the set of possible parity function $p:X\to \{\ol{0},\ol{1}\}$
 is also finite. If $\Sp(v)$ is infinite, then one of the equivalence classes is infinite,
 i.e.,  $\Sp^D(u)$ is infinite for some $u$.
By~\Lem{lemSkDv},   $\Sp^D(u)$ acts faithfully on $\Sp^D(v)$, so $\Sp^D(v)$ is infinite.

 For (iii)  assume that $\Sp(v)$ is of type II.  Combining~\ref{type2Aut}
 and~\ref{KMtypeIandII},  we obtain an embedding
  $\Sp^D(v)\to \Aut(\Sigma_{\pr})$.
 By~\Lem{Sigmapraff} the set  $\Sigma_{\pr}$ is finite. Thus  $\Sp^D(v)$ is finite. 
 Combining (ii), (iii)  and the fact $S(1|2;b)$ is of type I, we obtain that
 $\Sp(v)$ is finite in the type II.
 \end{proof}

\subsubsection{Remark}\label{remS12b}
For the case $S(1|2;b)$ the spine is infinite and $\Sp^D(v)$ is finite, 
see~\cite{GHS}, 10.4.

\subsubsection{}
\begin{cor}{}
Let    $\Sk(v)$ be an indecomposable Kac-Moody component. The set $\Sigma_{\pr}$ is finite. The spine
 is infinite if and only if $\Delta$ is infinite and $\mathbb{C}\Sigma_{\pr}\not=\mathbb{C}\Delta$.
 The group $\Sp^D(v)$ is infinite only
in the following cases: $A(m|n)^{(1)}$ for $m,n\geq 0$, $(m,n)\not=(0,0)$, and  $C(n+1)^{(1)}$ for $n\geq 1$.
\end{cor}
 \begin{proof}
In the type (Fin) the  set of roots is finite, so  $\Sigma_{\pr}$
  and the spine are finite (see~\ref{346}). In the type (Aff) the assertions  follow from~\Lem{Sigmapraff},  \Cor{corinf}, \Lem{lemspinefin} and~\Rem{remS12b}. 
For the purely anisotropic component   the spine is trivial and $\Sigma_{\pr}=\Sigma$.
In the remaining case $Q^{\pm}(m,n,t)$ the spine and  $\Sigma_{\pr}$ are finite,
see~\ref{Qmnpcon}.  
 \end{proof}

\section{Examples}
In this section we consider several  examples.
In~\ref{spineFin} we briefly describe the spines in type (Fin).
In~\ref{Cn+1spine} and~\ref{spineCn+11} we describe the spines for $C(n+1)$ and for $C(n+1)^{(1)}$.  
In~\ref{G31} we show that in van de Leur classification one has
$G(3)^{(1)}\cong G(3)^{(2)}$. In~\ref{infinitepiS} we describe an interesting example which is not fully reflectable.

We describe the root systems using the standard notation of~\cite{Ksuper}.

\subsection{Spines in the type (Fin)}\label{spineFin}
For the series $A(m|n)$, $B(m|n)$ and $D(m|n)$ the spine
can be described in terms Young diagrams, see~\cite{Mus}  and~\cite{BN}.
Below we give another description in the terms of $\varepsilon$-$\delta$ language which was used
in several papers for character formulae and computations of $\DS$-functors. Our  description 
easily follows from the descrption of sets of simple roots in~\cite{Ksuper}. 

For  $D(2|1;a)$ the spine and the skeleta coincide with the spine  and the skeleta of $D(2|1)$.
In Appendix we give the graphs of the spines for $G(3)^{(1)}$ and $F(4)^{(1)}$;
the spines of $G(3)$ and $F(4)$ are the subgraphs obtained by deleting the vertex marked by
$\Sigma_0$ and the edge marked by $r_{x_0}$ (the set of simple roots for $G(3)$ and $F(4)$ 
 can be obtained from
the sets of simple roots  for $G(3)^{(1)}$ and $F(4)^{(1)}$, presented in Appendix, 
 by deleting the first root in each set).

\subsubsection{Types $A(m-1|n-1)$, $B(m|n)$}\label{spineAB}
In the types $A(m-1|n-1)$ and $B(m|n)$ the spines can be described in the following way: the vertices
can be encoded by the words consisting of $m$ letters $\varepsilon$ and $n$ letters $\delta$, 
the edges correspond to permutations of the neighbouring $\varepsilon$ and $\delta$ (see, 
for example, \cite{Gdenomfin}); the edge is marked by $r_{x_i}$ when we permute
the letter number $i$ with the letter number $i+1$. For example, for $m=2$ and $n=1$, the spine contains an
 edge $v=\vareps\vareps\delta\overset{r_2}{\longleftrightarrow} u=\vareps\delta\vareps$,
 where the ordered sets $\Sigma_v,\Sigma_u$ are given by the following formulae:
with $\Sigma_v=\{\vareps_1-\vareps_2,\vareps_2-\delta_1\}$,
$\Sigma_u=\{\vareps_1-\delta_1,\delta_1-\vareps_2\}$ for $A(1|0)$ and
$\Sigma_v=\{\vareps_1-\vareps_2,\vareps_2-\delta_1,\delta_1\}$,
$\Sigma_u=\{\vareps_1-\delta_1,\delta_1-\vareps_2,\vareps_2\}$ for $B(2|1)$.

Note that the pair $A(m-1|n-1)$, $B(m|n)$ is an example of the situation 
described in~\ref{Sk'} (in the above example $m=2$ and $n=1$ we have
$X:=\{x_1,x_2,x_3\}$ for $B(2|1)$ and $X':=\{x_1,x_2\}$ for $A(1|0)$). In this case $\Sp(v')=\Sp(v)$
(since the last root in $\Sigma_u$  is anisotropic for any $u\in\Sp(v)$).

The skeleta is naturally isomorphic to the Cayley graph
of the Weyl group of $A_{m+n-1}$  and $B_{m+n}$ respectively,
see~\cite{GHS} and~\cite{Shay}.
Recall that the Weyl group acts faithfully on the 
skeleton and each orbit has a unique representative
in the spine (for $A_{m+n-1}$ the skeleton contains $(m+n)!$ vertices 
and the order of the Weyl group is $m!n!$; for $B(m|n)$ the skeleton contains $2^{m+n}(m+n)!$ vertices and the order of the Weyl group is $2^mm! 2^nn!$; 
in both cases the spine contains 
$\binom{m+n}{m}$ vertices).

\subsubsection{Type $D(m|n)$}\label{spineDmn}
In this case the  vertices in the spine can be encoded by the words consisting of $m$ letters 
$\varepsilon$ and $n$ letters $\delta$ with additional sign $+$ or $-$
if the last letter is $\delta$ (there are $\binom{m+n}{m}+\binom{m+n-1}{m}$ vertices). In this case the Weyl group is $W(D_m)\times W(C_n)$ and has the order
$2^{m+n-1}m!n!$, so the number of vertices in the skeleton is equal to
$$2^{m+n-1}m!n!\bigl(\binom{m+n}{m}+\binom{m+n-1}{m}\bigr)=2^{m+n-1}(m+2n)(m+n-1)!.$$

The edges in the spine correspond to permutations of the neighbouring $\varepsilon$ and $\delta$ with the following 
``sign rule'': there are no edges between vertices with different signs. 

For example,
 for $D(2|1)$ the spine contains $4$ vertices  $v_1=\delta\vareps\vareps$, 
 $v_2=+\vareps\vareps\delta$, $v_3=-\vareps\vareps\delta$ and $v_4=\vareps\delta\vareps$
 with $3$ edges $v_i \overset{r_i}{\longleftrightarrow} v_4$ for $i=1,2,3$; 
 for $D(1|2)=C(3)$ the spine is of the form
$$+\vareps\vareps\delta \overset{r_1}{\longleftrightarrow} +\delta\vareps\delta \overset{r_2}{\longleftrightarrow}\delta\delta\vareps 
\overset{r_3}{\longleftrightarrow}-\delta\vareps\delta \overset{r_1}{\longleftrightarrow}-\vareps\vareps\delta$$ 
(see~\ref{C3} for $\Sigma_u$).

 The marking on the edges is the following: an edge 
$v\overset{r_i}{\longleftrightarrow} u$ corresponds to the case when 
the word of $u$ is obtained from the word of $v$ by permuting the letter number $i$ with the letter number $i+1$,  except for the case when one of the words $u,v$ has sign $-$ and another does not have
sign (this means that we permute $\epsilon$ and $\delta$ which occupied the last two positions);
in the latter case the edge is marked by $r_{x_{m+n}}$, see the above examples.

Note that the pair $A(m-1|n-1)$, $D(m|n)$ is an example of the situation 
described in~\ref{Sk'}. In this case $\Sp(v')$
(the spine of $A(m-1|n-1)$)
 is a proper subgraph of $\Sp(v)$.

The spine admits an involutive automorphism which acts in the following way:
\begin{itemize}
\item
preserves the vertices encoded by unsigned words and interchanges 
the vertices encoded by the same word with different signs;
\item interchanges the edges marked by $r_{m+n}$ with the edges marked by $r_{m+n-1}$
and preserves the markings $r_i$ for $i<m+n-1$.
\end{itemize}
The image of a vertex $u$  has a set of simple roots equal to $s_{\vareps_n}\Sigma_u$
up to a permutation of elements. Each orbit of this involutive automorphism
contains a unique representative in $\Sp(v')$ (the spine of $A(m-1|n-1)$).

By~\cite{Shay}, Section 10
the skeleton for $D(m|n)$ is not the Cayley graph of a Coxeter group except  
for the case $D(1|1)=C(2)=A(1|0)$.
For example, for $D(2|1)$ the graph of the skeleton is isomorphic to a  so-called camfered cube
(vertices and  edges in the skeleton correspond to  vertices and edges in the camfered cube).

\subsection{Case $C(n+1)$}\label{Cn+1spine}
We fix $X=\{x_1,\ldots,x_n, x_{n+1}\}$.   In this case $\fh^*$
is spanned by $\vareps_1;\delta_1,\ldots,\delta_n$.
We fix the spine in such a way that 
\begin{itemize}
\item
the Cartan matrix is symmetric and 
the  bilinear form  described in~\ref{bilinearform} is such that 
$(\vareps_1,\vareps_1)=1$. Then  $(\delta_i,\delta_j)=-\delta_{ij}$ for $i,j=1,\ldots,n$, 
 and $(\delta_i,\vareps_1)=0$;
\item 
$\Sigma_{{pr}}=\pi=\{\delta_1-\delta_2,\ldots,\delta_{n-1}-\delta_n,2\delta_n\}$.
\end{itemize}

We denote by  $\sigma_{i;j}$ for the involution of $\mathbb{Z}$
which permutes $i$ and $j$.

\subsubsection{Case $C(2)$}
The spine takes the form
$v_{-1}\overset{r_1}\longleftrightarrow v_0 \overset{r_2}\longleftrightarrow v_1$
with
$$\begin{array}{lll}
\Sigma_{v_{-1}}=\{\vareps_1-\delta_1, 2\delta_1\} &\ \ 
 \Sigma_{v_0}=\{\delta_1-\vareps_1, \vareps_1+\delta_1\} &\ \ 
 \Sigma_{v_1}=\{ 2\delta_1, -\vareps_1-\delta_1\}.
\end{array}$$

\subsubsection{Case $C(3)$}\label{C3}
The spine  takes the form
$$v_{-2}\overset{r_1}{\longleftrightarrow} v_{-1}\overset{r_2}\longleftrightarrow v_0 \overset{r_3}\longleftrightarrow v_1 \overset{r_1}\longleftrightarrow v_2$$
with  the  following ordered sets $\Sigma_v:=\{b_v(x_i)\}_{i=1}^3$ (we give in
the parentheses the corresponding $\vareps$-$\delta$ word, see~\ref{spineFin})
$$\begin{array}{rlllcr}
\Sigma_{v_{-2}}=&\{\vareps_1-\delta_1,& \delta_1-\delta_2,& 2\delta_2\}\ &  & +\vareps\delta\delta\\
\Sigma_{v_{-1}}=&\{\delta_1-\vareps_1,& \vareps_1-\delta_2,& 2\delta_2\} & & +\delta\vareps\delta\\
 \Sigma_{v_0}=&\{\delta_1-\delta_2,&\delta_2-\vareps_1,& \vareps_1+\delta_2\}   & &\ \ \delta\delta \vareps \\
 \Sigma_{v_1}=&\{\vareps_1+\delta_1,&  2\delta_2, &-\vareps_1-\delta_2\}& & -\delta\vareps\delta\\
 \Sigma_{v_2}=&\{\ -\vareps_1-\delta_1,&  2\delta_2, & \delta_1-\delta_2\}   & & -\vareps\delta\delta\\
\end{array}$$

\subsubsection{}\label{Phi1}
Let $\ol{\alpha}:=\{\alpha_1,\ldots,\alpha_{n+1}\}$ be an ordered subset of $\Delta$.
We denote by $\tilde{\Phi}_1(\ol{\alpha})$ the following ordered subset of $\Delta$:
$$\tilde{\Phi}_1(\ol{\alpha}):= (s_{\vareps_1}\circ \sigma_{n;n+1})(\ol{\alpha})=\{s_{\vareps_1}(\alpha_1);s_{\vareps_1}(\alpha_{2});\ldots;s_{\vareps_1}(\alpha_{n+1});
s_{\vareps_1}(\alpha_n)\}.$$
where $s_{\vareps_1}$ is defined in~\ref{bilinearform}.

Observe  that  $\Sigma_{v_{-i}}=\tilde{\Phi}_1(\Sigma_{v_i})$.

\subsubsection{General case}\label{spineCn+1}
 Set
\begin{equation}\label{SigmaCn+1}
\Sigma_{v_0}:=\{\delta_1-\delta_2,\ldots,  \delta_{n-2}-\delta_{n-1},\delta_{n-1}-\delta_{n},\delta_n-\vareps_1,\delta_n+\vareps_1\}.\end{equation}
Noticing that $\Sigma_{v_0}$ is invariant under  $\tilde{\Phi}_1$ we obtain by induction on $i=1,\ldots,n$
\begin{enumerate}
\item  in the spine we have the arrows $v_{-i}\overset{r_s}{\longleftrightarrow} v_{-i+1}$ 
 and
$v_{i-1}\overset{r_{\sigma_{n;n+1}(s)}}{\longleftrightarrow} v_{i}$ for $s:=n+1-i$;
\item $\Sigma_{v_{-i}}=\tilde{\Phi}_1(\Sigma_{v_{-i}})$.
\end{enumerate}
Thus the spine contains a segment 
\begin{equation}\label{eqspineCn+1}
v_{-n}\overset{r_1}{\longleftrightarrow} v_{-n+1}\overset{r_2}\longleftrightarrow \ldots \overset{r_n}\longleftrightarrow 
v_0 \overset{r_{n+1}}\longleftrightarrow v_1 \overset{r_{n-1}}\longleftrightarrow v_2 \ldots \overset{r_1}\longleftrightarrow v_n
\end{equation}
with  the  following ordered sets $\Sigma_v:=\{b_v(x_i)\}_{i=1}^{n+1}$: 
$$\begin{array}{rl}
\Sigma_{v_{-n}}=&\{\vareps_1-\delta_1, \delta_1-\delta_2, \ldots, \delta_{n-1}-\delta_{n}, 2\delta_n\}\\
\Sigma_{v_{-i}}=&\{\delta_1-\delta_2,\ldots, \delta_{i-1}-\delta_{i},\delta_{i}-\vareps_1,\vareps_1-\delta_{i+1},\delta_{i+1}-\delta_{i+2},\ldots, 2\delta_n\}\\
\Sigma_{v_{-1}}=&\{\delta_1-\delta_2,\ldots, \delta_{n-2}-\delta_{n-1},\delta_{n-1}-\vareps_1,\vareps_1-\delta_n,2\delta_n\}\\
\Sigma_{v_{0}}=&\{\delta_1-\delta_2,\ldots,  \delta_{n-2}-\delta_{n-1},\delta_{n-1}-\delta_{n},\delta_n-\vareps_1,\delta_n+\vareps_1\}\\
\Sigma_{v_{1}}=&\{\delta_1-\delta_2,\ldots, \delta_{n-2}-\delta_{n-1},\delta_{n-1}-\vareps_1,2\delta_n,-\delta_n-\vareps_1\}\\
\Sigma_{v_{i}}=&\{\delta_1-\delta_2,\ldots, \delta_{i-1}-\delta_{i},\delta_{i}+\vareps_1,-\vareps_1-\delta_{i+1},\delta_{i+1}-\delta_{i+2},\ldots, 2\delta_n,\delta_{n-1}-\delta_{n}\}\\
\Sigma_{v_{n}}=&\{-\vareps_1-\delta_1, \delta_1-\delta_2, \ldots, , 2\delta_n, \delta_{n-1}-\delta_{n}\}
\end{array}
$$
where  $i=2,\ldots,n-1$. 
Note that $\Sigma_{v_{\pm i}}$ contains at most two isotropic roots and 
contains only one  isotropic root for $i=\pm n$. Thus a vertex $v_i$ has  degree two if $i=0,\pm 1,\ldots,\pm (n-1)$
and degree one if $i=\pm n$. Hence
 the spine of $C(n+1)$ is given by~(\ref{eqspineCn+1}).

\subsection{Case $C(n+1)^{(1)}$}\label{CaseCn+11}
We fix $X=\{x_0,x_1,\ldots,x_n, x_{n+1}\}$  for $C(n+1)^{(1)}$.
In this case $\fh^*$
is spanned by the elements $\vareps_1;\delta_1,\ldots,\delta_n;\delta$ and $\Lambda_0$
where $\Delta^{\ima}=\{j\delta\}_{j\not=0}$.
We  fix the spine in such a way that 
\begin{itemize}
\item
the Cartan matrix is symmetric and
the corresponding bilinear form on $\fh^*$ is such that
$(\vareps_1,\vareps_1)=1$. One has  $(\delta_i,\delta_j)=-\delta_{ij}$ for $i,j=1,\ldots,n$, 
 and $(\delta_i,\vareps_1)=(\delta,\vareps_1)=(\delta,\delta_i)=0$.
 
\item 
$\Sigma_{{pr}}=\pi=\{\delta-2\delta_1,\delta_1-\delta_2,\ldots,\delta_{n-1}-\delta_n,2\delta_n\}$.
\end{itemize}

Then $(\Delta^{\ima})^+=\mathbb{Z}_{>0}\delta$.

Similarly to~\ref{Phi1}, if 
 $\ol{\alpha}:=\{\alpha_0,\alpha_1,\ldots,\alpha_m\}$ is an ordered subset of $\Delta$,
 we set $\tilde{\Phi}_1(\ol{\alpha}):= (s_{\vareps_1}\circ \sigma_{n;n+1})(\ol{\alpha})$ and
$$\tilde{\Phi}_0(\ol{\alpha}):= (s_{2\delta-\vareps_1}\circ \sigma_{0;1})(\ol{\alpha})=\{s_{2\delta-\vareps_1}(\alpha_1);s_{2\delta-\vareps_1}(\alpha_0);s_{2\delta-\vareps_1}(\alpha_2);\ldots;
s_{2\delta-\vareps_1}(\alpha_m)\}.$$

Note that the reflections $s_{2\delta-\vareps_1},s_{\vareps_1}$ generate the infinite dihedral subgroup in $GL(\fh^*)$; 
this subgroup acts trivially on the even roots and stabilizes the set of odd roots.

\subsubsection{Remark}
	\label{autofCartan}
	If  the spine contains $v \overset{r_1}{\longleftrightarrow} u$ 
	(resp., $v \overset{r_n}{\longleftrightarrow} u$) 
	and $\Sigma_v$ 
	(as an ordered set) is $\tilde{\Phi}_0$-invariant (resp.,  $\tilde{\Phi}_1$-invariant), then 
	the spine contains $v \overset{r_0} {\longleftrightarrow}u'$ with $\Sigma_{u'}=\tilde{\Phi}_0(\Sigma_u)$
	(resp., $v \overset{r_{n-1}} {\longleftrightarrow}u'$
	with $\Sigma_{u'}=\tilde{\Phi}_1(\Sigma_u)$).

\subsubsection{Case $C(n+1)^{(1)}$,  $n>1$}\label{spineCn+11}
Combining~\ref{Sk'} and~\ref{spineCn+1}, we conclude that the spine contains~(\ref{eqspineCn+1}) as a subgraph with  the  following ordered sets $\Sigma_v:=\{b_v(x_i)\}_{i=0}^{n+1}$ :
$$\begin{array}{rl}
\Sigma_{v_{-n}}=&\{\delta-\vareps_1-\delta_1,\vareps_1-\delta_1, \delta_1-\delta_2, \ldots, \delta_{n-1}-\delta_{n}, 2\delta_n\}\\
\Sigma_{v_{-i}}=&\{\delta-2\delta_1,\delta_1-\delta_2,\ldots, \delta_{i-1}-\delta_{i},\delta_{i}-\vareps_1,\vareps_1-\delta_{i+1},\delta_{i+1}-\delta_{i+2},\ldots, 2\delta_n\}\\
\Sigma_{v_{-1}}=&\{\delta-2\delta_1,\delta_1-\delta_2,\ldots, \delta_{n-2}-\delta_{n-1},\delta_{n-1}-\vareps_1,\vareps_1-\delta_n,2\delta_n\}\\
\Sigma_{v_{0}}=&\{\delta-2\delta_1,\delta_1-\delta_2,\ldots,  \delta_{n-2}-\delta_{n-1},\delta_{n-1}-\delta_{n},\delta_n-\vareps_1,\delta_n+\vareps_1\}\\
\Sigma_{v_{1}}=&\{\delta-2\delta_1,\delta_1-\delta_2,\ldots, \delta_{n-2}-\delta_{n-1},\delta_{n-1}-\vareps_1,2\delta_n,-\delta_n-\vareps_1\}\\
\Sigma_{v_{i}}=&\{\delta-2\delta_1,\delta_1-\delta_2,\ldots, \delta_{i-1}-\delta_{i},\delta_{i}+\vareps_1,-\vareps_1-\delta_{i+1},\delta_{i+1}-\delta_{i+2},\ldots, 2\delta_n,\delta_{n-1}-\delta_{n}\}\\
\Sigma_{v_{n}}=&\{\delta+\vareps_1-\delta_1,-\vareps_1-\delta_1, \delta_1-\delta_2, \ldots, , 2\delta_n, \delta_{n-1}-\delta_{n}\}
\end{array}
$$
where  $i=2,\ldots,n-1$.  
Noticing that $\Sigma_{v_{-n}}$ is invariant under  $\tilde{\Phi}_0$ and $\Sigma_{v_{n}}$ is invariant under  
$\tilde{\Phi}_1\tilde{\Phi}_0$ (and using~\ref{autofCartan})
we obtain by induction for $i=1,\ldots,n$ that the spine contains the following subgraphs
\begin{equation}\label{spineCn+1aff}\begin{array}{ll}
v_{-2n}\overset{r_n}{\longleftrightarrow} v_{-2n+1}\overset{r_{n-1}}{\longleftrightarrow}\ldots
v_{-n-1} \overset{r_2}{\longleftrightarrow} 
v_{-n-1}\overset{r_0}{\longleftrightarrow} v_{-n}\overset{r_1}{\longleftrightarrow} v_{-n+1}\overset{r_2}{\longleftrightarrow} \ldots \overset{r_n}{\longleftrightarrow} 
v_0 \\
v_0\overset{r_{n+1}}{\longleftrightarrow} v_1 \overset{r_{n-1}}{\longleftrightarrow} v_2 \ldots 
\overset{r_1}{\longleftrightarrow} v_n
\overset{r_0}{\longleftrightarrow} v_{n+1}\overset{r_2}{\longleftrightarrow}\ldots \overset{r_{n-2}}{\longleftrightarrow} 
v_{2n-2}\overset{r_{n-1}}{\longleftrightarrow}
v_{2n-1}\overset{r_{n+1}}{\longleftrightarrow} v_{2n}\end{array}
 \end{equation}
where for $i=1,\ldots,n$ we have
 $\Sigma_{v_{-n-i}}=\tilde{\Phi}_0(\Sigma_{v_{-n+i}})$ and  $\Sigma_{v_{n+i}}=\tilde{\Phi}_1\tilde{\Phi}_0(\Sigma_{v_{n-i}})$.
 In particular,
$$\begin{array}{rl}
\Sigma_{v_{2n}}&=\{\delta_1-\delta_2,\delta-2\delta_1, \delta_2-\delta_3,  \ldots, \delta_{n-1}-\delta_{n},\delta+\delta_n+\vareps_1;-\delta+
\delta_n-\vareps_1\}\\
\Sigma_{v_{-2n}}&=\{\delta_1-\delta_2,\delta-2\delta_1, \delta_2-\delta_3,  \ldots, \delta_{n-1}-\delta_{n},-\delta+\delta_n+\vareps_1;\delta+
\delta_n-\vareps_1\}
\end{array}$$

The Cartan data for $v_{2n}$ and $v_{-2n}$ are equal. One has
$\Sigma_{v_{2n}}=(s_{\vareps_1}s_{2\delta-\vareps_1})^2 (\Sigma_{v_{-2n}})$.
This allows to construct $v_s$ for each $s\in\mathbb{Z}$ through the formula
 $v_{j+4mn}:=(s_{\vareps_1}s_{2\delta-\vareps_1})^{2m}(v_j)$ for $j=0,\pm 1,\ldots,\pm 2n$ and any integral $m$. Note that for these values of $j$ the spine contains
 the arrows $v_j\longleftrightarrow v_{j+1}$. Therefore the spine contains
 the arrows $v_s\longleftrightarrow v_{s+1}$ for all $s\in\mathbb{Z}$. 
 Since each $\Sigma_{v_s}$ contains exactly two isotropic roots, each vertex $v_i$ has  degree two. Hence
 the spine of $C(n+1)^{(1)}$ takes the form
 $$\ldots\overset{r_n}\longleftrightarrow v_0\overset{r_{n+1}}{\longleftrightarrow} v_1 \overset{r_{n-1}}{\longleftrightarrow} v_2 \ldots 
\overset{r_1}{\longleftrightarrow} v_n
\overset{r_0}{\longleftrightarrow} v_{n+1}\overset{r_2}{\longleftrightarrow}\ldots\overset{r_{n-1}}{\longleftrightarrow}
v_{2n-1}\overset{r_{n+1}}{\longleftrightarrow} v_{2n}\overset{r_{n}}{\longleftrightarrow}\ldots$$
The  marks over the arrows  are determined by~(\ref{spineCn+1aff}) 
and the fact that $v_s\overset{r_{i(s)}}{\longleftrightarrow} v_{s+1}$ implies
$v_{s+4n}\overset{r_{i(s)}}{\longleftrightarrow} v_{s+4n+1}$
(i.e., the marks  have period $4n$).

\subsection{}
\begin{cor}{corCnSp}
\begin{enumerate}
\item
The group $\Sp^D(v)$ is trivial for $C(n+1)$. 
\item
For $C(n+1)^{(1)}$ with   $n>1$
we let $\Psi$ be the restiction of $(s_{\vareps_1}s_{2\delta-\vareps_1})^2$ to $V_b$.
The group $\Sp^D(v)$ is an infinite cyclic group generated by    $u$ with  $\sigma_b^u=\Psi$
and $\sigma^u(v_i)=v_{i+4n}$.
\end{enumerate}
\end{cor} 
\begin{proof}
Any element of $\Sp^D(v)$ induces an automorphism of the spine as a marked graph. 
In particular,  for $C(n+1)$ the group $\Sp^D(v)$ is trivial since the graph presented
in~(\ref{spineCn+1}) does not admit non-trivial automorphisms.

Consider the case $C(n+1)^{(1)}$ with   $n>1$. By~\ref{spineCn+11}, the group $\Sp^D(v)$ contains
$\Psi$ and $\Psi^{m}(v_j)=v_{j+4mn}$.  Assume that  $\Psi$ generates a proper subgroup of $\Sp^D(v)$.
Then $\Sp^D(v)$
contains  an automorphism which maps
$v_{-2n+1}$ to $v_i$ with $-2n+1<i\leq 2n$. Note that $v_{-2n+1}$ and $v_i$ belong to the part
of the spine presented in  (\ref{spineCn+1aff}). Since $v_{-2n+1}$ is the only vertex
with the arrows  marked by $r_n$ and $r_{n-1}$, we have $i=-2n+1$, a contradiction.
\end{proof}

%
%
%

 \subsection{Example: $G(3)^{(1)}$}\label{G31}
	In van de Leur classification of symmetrizable affine Lie superalgebras the algebra
	$G(3)^{(1)}$ appears twise: as $G(3)^{(1)}$  and as $G(3)^{(2)}$.  The automorphism $\sigma$ of $G(3)$ which used
	to construct  $G(3)^{(2)}$ is given in Table 4 of~\cite{vdL}. One readily sees that $\sigma$ is 
	the inner automorphism $\exp(h)$ where $h\in\fh$ is such that $\langle h,\alpha_1\rangle= \pi i$
	and $\langle h,\alpha_2\rangle= \langle h,\alpha_3\rangle=0$, so the isomorphism
	 $G(3)^{(1)}\cong G(3)^{(2)}$ which we present in~\Rem{G321} below is not suprising.
	
	Using  a well-known presentation of the root system of $G_2$ (in terms 
	of $\vareps_i, i=1,2,3$ with $\sum\limits_{i=1}^3 \vareps_i=0$ and
	$(\vareps_i,\vareps_i)=-2,\ (\vareps_i,\vareps_j)=1$)  we express the root system of $G(3)$ in 
	terms of $\vareps_i, i=1,2,3$ and $\delta_1$ with 
	$$(\vareps_i,\delta_1)=0,\ \ (\delta_1,\delta_1)=-1$$
	with the parity function given by $p(\vareps_i)=\ol{0}$, $p(\delta_1)=\ol{1}$.
	
	\subsubsection{}\label{G321}
	Let us check that  $G(3)^{(1)}\cong G(3)^{(2)}$. It is enough to verify that these superalgebras
	admits $D$-equivalent Cartan data.

	We consider a base for $G(3)$ consisting of $\vareps_2,\vareps_3-\delta_1,\delta_1-\vareps_2$; this gives
	the following base for $G(3)^{(1)}$: 
	$$\{\vareps_3-\delta_1, \delta+\vareps_1-\vareps_3, \vareps_2,\delta_1-\vareps_2\}$$
	(this base is obtained by reordering of $\Sigma_2$ in Appendix~\ref{appendix:a}). The Cartan matrix
	$$A=\begin{pmatrix}
0 & -3/2 &-1/2 & 3/2\\
-3/2 & 3& 0 & 0\\
-1/2 & 0 & 1 & -1\\
3/2 & 0  &-1 & 0\end {pmatrix}$$

	In~\cite{REIF},  6.2 the roots system of $G(3)^{(2)}$ in given in the terms of  an orthogonal basis
	$\delta,\epsilon_1,\epsilon_2,\epsilon_3$\footnote{ \cite{REIF}, 6.2 contains a misprint: 
	the minimal imaginary root is $2\delta$.} with the form
	$$(\delta',\delta')=0,\ \ (\epsilon_1,\epsilon_1)=\frac{3}{2},\ \  (\epsilon_2,\epsilon_2)=\frac{1}{2},
	(\epsilon_3,\epsilon_3)=-2.$$
	 The parity function is given by $p'(\delta)=p'(\epsilon_2)=p'(\epsilon_3)=\ol{0}$
	and $p'(\epsilon_3)=\ol{1}$.
	A  set of simple roots in~\cite{REIF}, Table 3  is $\{\epsilon_3-\epsilon_2-\epsilon_1, 2\epsilon_1,2\epsilon_2,
	\delta-(\epsilon_3+2\epsilon_2)\}$; the 
	corresponding Cartan matrix is
	$$A'=\begin{pmatrix}
0 & -3 &-1 & 3\\
-3 & 6  & 0 & 0\\
-1 & 0 & 2 & -2\\
3 & 0 & -2 & 0\end {pmatrix}$$	
	and the only odd root is the first one. Since $A'=2A$ and $p'=p$, we have $G(3)^{(1)}\cong G(3)^{(2)}$.
	
	Note that in~\cite{vdL} notation used in Table 5, 
	an isomorphism between root systems of $G(3)^{(1)}$ and $G(3)^{(2)}$
	is given, for example, by the formula $\delta\mapsto 2\delta$, $\vareps_i\mapsto \delta+\vareps_i$
	for $i=2,3$, $\vareps_1\mapsto -2\delta+\vareps_1$, $\delta_1\mapsto \delta+\delta_1$.

\subsection{Example with an infinite set $\pi_S$}\label{infinitepiS}
Take $X=\{x_1,x_2,x_3,x_4\}$ and $v_0$ with the Cartan datum
$$A=\begin{pmatrix}
2 & -1& -1& 0\\
-1 & 0 & 1 & -t\\
-1 & 1& 0& t\\
0 & -t& t &2
\end{pmatrix},\ \ \  p(x_1)=p(x_4)=\ol{0},\ \ p(x_2)=p(x_3)=\ol{1}$$
with $t\in\mathbb{R}\setminus\mathbb{Q}$. 
Observe that  $x_i$ is reflectable at $v$ if and only if  $i\not=4$. 
Note that for $X':=\{x_1,x_2,x_3\}$ the corresponding Cartan datum
 is of the type $A(1|0)^{(1)}$.  
 
 One has  $\rank A=3$, so $\dim \fh$ must be at least $5$.
 We take $\fh$ of the dimension $5$ and
  $$\Sigma_{v_0}=\{\vareps_1-\vareps_2,\vareps_2-\delta_1,\delta-\vareps_1+\delta_1,\beta\}$$
 where $\vareps_1,\vareps_2,\delta_1,\delta,\beta$ is a basis of $\fh^*$. We can choose the bilinear form
 by setting
 $$(\vareps_i,\vareps_j)=\delta_{ij}=-(\delta_1,\delta_1),\ \ (\beta,\delta_1)=t$$
 and defining all other products between bases elements equal to zero.
 
Let $r_{x_2}: v_0\to v_1$ be the isotropic reflection. Then
$$\Sigma_{v_1}=\{\vareps_1-\delta_1,\delta_1-\vareps_2,\delta-\vareps_1+\vareps_2, \beta+\vareps_2-\delta_1\}$$
(as the ordered set) and the Cartan matrix is
$$\begin{pmatrix}
0 & 1& -1& -t-1\\
1 & 0 & -1 & t+1\\
-1 & -1& 2& 1\\
-t-1 & t+1&1& 2-2t 
\end{pmatrix}$$
Note that  $\delta-\vareps_1+\vareps_2$ is a non-reflectable root, so the skeleton does not contain
the skeleton of $A(1|0)^{(1)}$. The spine contains the spine of $A(1|0)^{(1)}$ 
which takes the form
\begin{equation}\label{interestingspine}
\ldots \overset{r_{x_2}}{\longleftrightarrow} v_{-2} \overset{r_{x_1}}{\longleftrightarrow} v_{-1}\overset{r_{x_3}}{\longleftrightarrow} v_0\overset{r_{x_2}}{\longleftrightarrow} v_1 \overset{r_{x_1}}{\longleftrightarrow} v_2 \overset{r_{x_3}}{\longleftrightarrow} v_3 \overset{r_{x_2}}{\longleftrightarrow} v_4\overset{r_{x_1}}{\longleftrightarrow}   \ldots    \end{equation}
with the following unordered sets of simple roots 
 $$\begin{array}{ll}
\{\vareps_1-\vareps_2, i\delta+\vareps_2-\delta_1, (1-i)\delta-\vareps_1+\delta_1,
 \beta+i(i-1)\delta+i\str\}\ &\text{ for }v_{2i}\\
\{\delta-\vareps_1+\vareps_2,i\delta+\vareps_1-\delta_1, -i\delta-\vareps_2+\delta_1,
 \beta+i^2\delta+i\str+ (\vareps_2-\delta_1)\}\ & \text{ for }v_{2i+1}\end{array}$$
 where $\str:=\vareps_1+\vareps_2-2\delta_1$.
 In all cases the last root is $b_{v_j}(x_4)$ and $x_4$ is non-reflectable at $v_j$.
Therefore the spine coincides with  the graph~(\ref{interestingspine}).
 One has
 $\pi=\{\vareps_1-\vareps_2\}$, so $W\cong \mathbb{Z}_2$ and the skeleton is 
 $$  
\xymatrix{& \ldots\ar^{r_{x_2}}[r] & v_{-2}\ar^{r_{x_3}}[d]\ar[l]\ar^{r_{x_1}}[r]& v_{-1}\ar^{r_{x_3}}[r]\ar[l] & v_0\ar^{r_{x_2}}[r]\ar^{r_{x_1}}[d]\ar[l]&v_1\ar^{r_{x_1}}[r]\ar[l] &v_2\ar^{r_{x_3}}[r]\ar^{r_{x_2}}[d]\ar[l]&\ldots\ar[l]& \\
&\ldots\ar^{r_{x_2}}[r] & v'_{-2}\ar[u]\ar[l]\ar^{r_{x_1}}[r]& v'_{-1}\ar^{r_{x_3}}[r] \ar[l]& v'_0\ar^{r_{x_2}}[r]\ar[u]\ar[l]&v'_1\ar^{r_{x_1}}[r]\ar[l] &v'_2
\ar^{r_{x_3}}[r]\ar[u]\ar[l]&\ldots\ar[l]& }$$
where $v'_i=s_{\vareps_1-\vareps_2} v_i$ and $\Sigma_{v'_i}=s_{\vareps_1-\vareps_2} \Sigma_{v_i}$.

The set of $S$-principle elements introduced in~\ref{defnpiS} is infinite:
 $$\pi_S=\{\vareps_1-\vareps_2,\ \delta-\vareps_1+\vareps_2;
  \beta+i(i-1)\delta+i\str ,\ 
 2(\beta+i^2\delta+i\str+ (\vareps_2-\delta_1))\}_{i\in\mathbb{Z}}.$$
 Note that the imaginary ray $\mathbb{R}_{\geq 0}\delta$ is the limit ray
 of the rays $\mathbb{R}_{\geq 0}\alpha$ for $\alpha\in\pi_S$.

 \section{The group $\Sk^D(v)$ in the type (Aff)}
 In this section we describe $\Sk^D(v)$ in the case (Aff) and establish the formula~(\ref{eqintroSkDv}).

 \subsection{The group $\tau(V_b)$}
 Recall that indecomposable Kac-Moody superalgebras except for  $S(1|2;b)$ can be realized as twisted or untwisted affinizations of finite-dimensional Kac-Moody superalgebras or a simple $Q$-type superalgebra.
 Let $\Sk(v)$ be an indecomposable  component of the type (Aff) which is not $S(1|2;b)$.
 We fix a bilinear form $(-,-)$ on $\fh^*$ in the following way:
 if the component is symmetrizable, we denote $(-,-)$ a non-degenerate invariant bilinear form;
 in the remaining case $\fq_n^{(2)}$ we denote by $(-,-)$ a non-degenerate invariant bilinear form on $\fg_{\ol{0}}\cong \fsl_n^{(1)}$.  Retain notation of~\ref{VaVb}. We have $(V_b,\delta)=0$.
 We denote by $O(\fh^*)$
the subgroup of $GL(\fh^*)$ which preserves $(-,-)$ and set 
$$G:=\{\phi\in O(\fh^*)|\ \phi(V_b)=V_b, \ \phi(\delta)=\delta\}.$$

\subsubsection{}
The form $(-,-)$ induces a bilinear form on the quotient space $\dot{V}_b:=V_b/\mathbb{C}\delta$
(this form is non-degenerate except for the case $A(n|n)^{(1)}$). Following the notation of~\cite{Kbook},
Section 6, we denote by $\ol{\lambda}$ the image of $\lambda$ in $\dot{V}_b$.
We have the natural homomorphism  $G\to O(\dot{V}_b)$.

\subsubsection{}
We set 
$$\dot{\Delta}:=\{\ol{\alpha}|\ \alpha\in\Delta\}\setminus\{0\},\ \ \ 
\dot{\Delta}_{\pi}:=\{\ol{\alpha}|\ \alpha\in\Delta_{\ol{0}}\}\setminus\{0\}.$$
By~(\ref{affmodulodelta}), $\dot{\Delta}$ is finite.
The set $\dot{\Delta}_{\pi}\subset \dot{V}_b$ 
satisfies the standard axioms of a (finite) root system.

\subsubsection{}\label{dotpi}
For example, if $\fg=\dot{\fg}^{(1)}$ is an untwisted affinization, then 
$\dot{\Delta}$ is the root system of
$\dot{\fg}$ and $\dot{\Delta}_{\pi}=\dot{\Delta}_{\ol{0}}$. 
We denote by $\dot{v}$ a vertex in the skeleton of $\dot{\fg}$
and let $v$ be the corresponding vertex in the skeleton of $\fg$; one has 
 $\Sigma_v=\Sigma_{\dot{v}}\cup\{\delta-\theta\}$ where $\theta$ is the highest root
 in $\Delta^+_{\dot{v}}\subset \dot{\Delta}$.

Another example is $\fq_n^{(2)}$. In this case $\dot{\Delta}=\dot{\Delta}_{\pi}$ is the root system of $\fsl_n$. We denote by $\dot{v}$ a vertex in the skeleton of 
$\dot{\fg}$
and let $v$ be the corresponding vertex in the skeleton of $\fg$; one has 
 $\Sigma_v=\Sigma_{\dot{v}}\cup\{\delta-\theta\}$ where $\theta$ is as above
(in this case $p(\delta)=\ol{1}$).

In both cases the set 
$$\dot{\pi}=\pi\cap \dot{\Delta}$$
 is the set of simple roots for $\dot{\Delta}_{\pi}$.

\subsubsection{Remark}
In general, it is not hard to show
 (see, for example,~\cite{GSh})  that $\dot{\Delta}$
 is an indecomposable  weak generalized root system
 (WGRS)  in the sense of~\cite{SGRS}.   
From the description of the root systems given in~\cite{vdL}
we see that $\dot{\Delta}$ is not a root system of a Kac-Moody superalgebras
if and only if $\fg$ is a twisted affinization of $A(m|n)$
which is not isomorphic to $A(m|n)^{(1)}$.

\subsubsection{}
 For each $\nu\in V_b$  we introduce
$t_{\nu}\in GL(\fh^*)$ by the formula
\begin{equation}\label{tnu}
t_{\nu}(\lambda):=\lambda+k\nu-((\lambda,\nu)+\frac{k}{2}
(\nu,\nu))\delta\ \ \text{ where } k:=(\lambda,\delta).\end{equation}
Since $(V_b,\delta)=0$  we have  $(t_{\nu}(\lambda),t_{\nu}(\lambda))=(\lambda,\lambda)$ so $t_{\nu}\in O(\fh^*)$. The following properties can be easily verified
\begin{equation}\label{tnubeta}\begin{array}{lcl}
(i) & & t_{\nu} t_{\mu}=t_{\mu+\nu};\\
(ii) & & t_{\nu}(\beta)=\beta-(\beta|\nu)\delta\  \text{ for }\ \beta\in \Delta;\\
(iii) & & t_{\nu}(\delta)=\delta;\ \ \ \  t_{c\delta}=\Id\ \ \text{ for  } c\in\mathbb{C};\\
(iv) & & \phi t_{\nu} \phi^{-1}=t_{\phi(\nu)}\ \ \text{ for } \phi\in G.
\end{array}
\end{equation}
Thus the  assignment $\nu\mapsto t_{\nu}$ defined a homomorphism $\tau: {V}_b\to G$
with $\delta\in \ker\tau$ (and $\ker\tau=\mathbb{C}\delta$ apart for the case
$A(n|n)^{(1)}$).
 The image   of $\tau$ is a normal subgroup of $G$ which is isomorphic to
$\dot{V}_b$ except for the case $A(n|n)^{(1)}$. Note that (iv) gives
$$wt_{\nu}w^{-1}=t_{w\nu}\  \text{ for }w\in W.$$

\subsubsection{}\begin{lem}{}
The Weyl group  $W$ is a subgroup of $G$.
\end{lem}
\begin{proof}
The proof is the same as in~\cite{Kbook}, Proposition 3.9. 
\end{proof}

\subsubsection{}\label{tauN}
Since $\tau(V_b)\cong \dot{V}_b$ is a  torsion-free subgroup in $G$ one has
$ \tau(V_b)\cap H=1$ if $H$ is a finite subgroup of $G$. Using the formula~(\ref{tnubeta}) (iv) we conclude that
for any finite subgroup $H$ of $G$ and  an $H$-invariant
subgroup $N$ of $V_b$ we have
$H\tau(N)=H\ltimes \tau(N)$.

In light of~\cite{Kbook} 6.3 and 6.5 we have
\begin{equation}\label{Weyldeco}
W=\dot{W}\ltimes \tau(M_{\pi})
\end{equation}
where $\dot{W}\subset G$ is  isomorphic to the Weyl group of the  root system
$\dot{\Delta}_{\pi}$ and $M_{\pi}\cong Q(\pi)$ is
a $W$-invariant lattice   in $\mathbb{Q}\pi$ given by the formula (6.5.8) in~\cite{Kbook}; the group
$\tau(M_{\pi})\cong \tau(Q(\pi))$ is a free abelian group. 
In $\fg=\dot{\fg}^{(1)}$, then 
 $\tau(M_{\pi})=\tau(Q(\dot{\pi}))\cong Q(\dot{\pi})$.

The decomposition~(\ref{Weyldeco}) 
implies the formula  $W\tau(N)=\dot{W}\ltimes \tau(N)$ for any 
$W$-invariant subgroup $N\subset V_b$ which contains $M_{\pi}$.

\subsection{Remark}\label{remaff}
In~\cite{Kbook}, the root system $\dot{\Delta}$ for the Kac-Moody algebra
is constructed as a ``finite part'' of $\Delta$ (this means that the set of simple of roots $\dot{\Sigma}$ 
for $\dot{\Delta}$ is a subset of $\Sigma$,  the Dynkin diagram of $\dot{\Sigma}$ is connected and 
$\Sigma\setminus \dot{\Sigma}$
is of cardinality one). In order to generalize this construction to 
 superalgebras, one has to
fix a finite part $\dot{\Delta}'$ in $\Delta$ in such a way that $W=\dot{W}'\ltimes \tau(M')$,
where $\dot{W}'$ is the Weyl group of $\dot{\Delta}'$ and $M'$ is a certain $W$-invariant lattice in $V_b$.

Observe that $\dot{W}'$ is the set of elements of the finite order in $W$, so the decomposition~(\ref{Weyldeco})
implies $\dot{W}'=\dot{W}$. As a result, the existence of the decomposition $W=\dot{W}'\ltimes \tau(M')$
is equivalent to the existence of  $\dot{\Delta}'$  with the Weyl group equal to $\dot{W}$.
The finite parts of symmetrizable Kac-Moody superalgebras are listed in~\cite{GKadm}, 13.2.1.
All finite parts of $S(1|2;b)$ are of the type $A(-1|0)$ (corresponding to $\fsl(1|2)$), and
the finite parts of $\fq_{n}^{(2)}$ are $A(k-1|n-1-k)$  (corresponding to $\fsl(k|n-k)$  if $2k\not=n$
and to $\fgl(k|k)$ if $n=2k$). Analyzing  these descriptions we  see that in all cases, except for $A(2m-1|2n-1)^{(2)}$,
it is possible to choose $\dot{\Delta}'$ in such a way that $\dot{W}$ equals to the Weyl group of $\dot{\Delta}'$.
For the remaining case $A(2m-1|2n-1)^{(2)}$ a suitable finite part does not exist, since
all finite parts are of the types $D(m|n)$ or $D(n|m)$, whereas
$\dot{W}$ is the Weyl group of $C_m\times C_n$.

\subsection{Case $A(m|n)^{(1)}$}
This case is treated in~\cite{GHS}, 10.3. We recall the results below.

In this case $\fg=\dot{\fg}^{(1)}$ where $\dot{\fg}$ is of the type $A(m|n)$. 
We normalize the bilinear form in the usual way (i.e., $(\alpha,\alpha)=\pm 2$ for $\alpha\in\pi$).
One has
$W=\dot{W}\ltimes \tau(Q(\pi))$
where $\dot{W}\cong S_m\times S_n$ is the Weyl group
of $\dot{\Delta}$. One has $ \tau(Q(\pi))=\tau(Q(\dot{\pi}))\cong Q(\dot{\pi})$
(since $Q(\dot{\pi})\cap \ker\tau=0$).

For $m\not=n$ and $m,n\geq 0$, \cite{GHS}, 10.3 gives $Sk^D(v)=W\tau(Q_v)$. By~\ref{tauN} this implies
$\Sk^D(v)=\dot{W}\ltimes \tau(Q_v)$
and $\tau(Q_v)=\tau(Q_{\dot{v}})\cong Q_{\dot{v}}$. 
$$\Sp^D(v)\cong \Sk^D(v)/W=
(\dot{W}\ltimes \tau(Q_v))/(\dot{W}\ltimes \tau(Q(\pi)))\cong Q_v/Q(\pi)\cong \mathbb{Z}.$$

For the case $m=n$ we retain notation of~\ref{Ann}.  
It is easy to see that the  involution $\xi$ (which sends $\vareps_i$ to $\delta_i$)  
can be uniquely extended to an element of $G$
(by setting $\xi(\delta)=\delta$ and $\xi(\Lambda_0)=\Lambda_0$). By~\cite{GHS}, 10.3 we have
$\Sk^D(v)=(\dot{W}\ltimes \{\xi,1\})\tau(Q_v)$. One has $\xi (Q_v)=Q_v$.
By~\ref{tauN}  we have
$\Sk^D(v)=
(\dot{W}\ltimes \{\xi,1\})\ltimes \tau(Q_v)$
and 
$$\Sp^D(v)\cong ((\dot{W}\rtimes \{\xi,1\}) \ltimes\tau(Q_v))/(\dot{W}\ltimes \tau(Q(\pi)))\cong 
(\{\xi,1\})\ltimes \tau(Q_v))/\tau(Q(\pi)).$$
The group  $Q/Q(\pi)$ is generated by  the element
$\str:=\sum\limits_{i=1}^n (\vareps_i-\delta_i)$. Since
$\xi (\str)=-\str$ we have
$\Sp^D(v)\cong  \mathbb{Z}_2\ltimes \mathbb{Z}$
and $\Sp^D(v)\not\cong  \mathbb{Z}_2\times \mathbb{Z}$.

\subsubsection{}
Let $\Sk(\dot{v})$ be the skeleton of $A(m|n)$. By~\ref{Sk'}, $\Sk(\dot{v})$ can be viewed as a subgraph
of $\Sk(v)$. One has
$\Sk^D(\dot{v})=\dot{W}$ for $m\not=n$ and $\Sk^D(v)=(\dot{W}\ltimes \{\xi,1\})$ for $m=n$.
This gives
$\Sk^D(v)=\Sk^D(\dot{v})\ltimes \tau(Q_v)$.
 
 \subsection{Case $C(n+1)^{(1)}$} 
 In this case $\dot{\Delta}=C(n+1)$. We retain the notation of~\ref{CaseCn+11}.
 We fix $v$ with $\Sigma_v$
 equal to $\Sigma_{v_0}$ described in~\ref{spineCn+11} and $\dot{v}$
 with $\Sigma_{\dot{v}}$ given by the formula~(\ref{SigmaCn+1}).
 Then $\Sk(\dot{v})$ is the skeleton of $C(n+1)$ and 
 $\Sk({v})$ is the skeleton of $C(n+1)^{(1)}$.

\subsubsection{}
\begin{cor}{}
Take $C(n+1)^{(1)}$ with $n>1$.
\begin{enumerate}
\item One has  $\Sk^D(v)=\mathbb{Z}u\times W$ where $u$ as in~\Cor{corCnSp}.

\item The embedding $W\to GL(\fh^*)$ can be extended to an embedding 
$\Sk^D(v)\to  GL(\fh^*)$
via the formula 
$u\mapsto (s_{\vareps_1}s_{2\delta-\vareps_1})^2=t_{-\vareps_1}$.

\item One has
$\Sk^D(v)={\Sk}^D(\dot{v})\ltimes\tau(Q_v)$ and ${\Sk}^D(\dot{v})=\dot{W}$.
\end{enumerate}
\end{cor}
\begin{proof}
Recall that the bilinear form is normalized by the condition
$(\delta_i,\delta_j)=-\delta_{ij}$. For this normalization
 the decomposition~(\ref{Weyldeco}) takes the form
 $W=\dot{W}\ltimes \tau(M_{\pi})$,
where $\dot{W}$ is the group of signed permutations of $\delta_1,\ldots,\delta_n$ and
$M_{\pi}$ is the span of  $\delta_1,\ldots,\delta_n$. 

By~\Cor{corCnSp},  $\Sp^D(v)$ is the infinite cyclic group generated by $u$ where $\sigma_b^u$
is the restriction of $(s_{\vareps_1}s_{2\delta-\vareps_1})^2=t_{-\vareps_1}$ to $V_b$. For all $w\in W$ one has $w  t_{\vareps_1}= t_{\vareps_1} w$.
In particular,  for any $\mu\in V_b$ one has 
\begin{equation}\label{wtmu}
(w\circ \sigma_b^u)(\mu)=( \sigma_b^u\circ w)(\mu).
\end{equation}
Using~\Lem{lemWSkDv} (iii) we obtain $\Sk^D(v)=W\times \Sp^D(v)$.
Since the group $W\subset GL(\fh^*)$ commutes with $t_{\vareps_1}$
the assignment $\phi_u:=t_{-\vareps_1}$ induces 
a homomorphism from $W\times \Sp^D(v)$ to $GL(\fh^*)$. 
Since $W\cap \{t_{m\vareps_1}\}_{m\in\mathbb{Z}}=\Id$ this homomorphism is injective.
This gives an embedding $\Sk^D(v)=W\times \Sp^D(v)\to GL(\fh^*)$
and establishes (ii).
For (iii) note that~\ref{tauN} gives $\Sk^D(v)={\Sk}^D(\dot{v})\ltimes\tau(N)$
for $N$ is spanned by $M_{\pi}$ and $\vareps_1$. Hence $N=Q_v$ as required.
\end{proof}

\subsection{Case $\fq_{n}^{(2)}$} 
In this case  $\dot{\Delta}$ is of the type $A(n-1)$. 
We normalize the bilinear form by the condition $(\alpha,\alpha)=2$ for 
$\alpha\in\pi$.
One has
$W=\dot{W}\ltimes \tau(M_{\pi})$
where $\dot{W}\cong S_n$ is the Weyl group
of $\dot{\Delta}$ and $M_{\pi}$ is the lattice spanned by the elements
$2\alpha$, $\alpha\in \dot{\Delta}$.  One has $\tau(M_{\pi})=\tau(\dot{M})\cong \dot{M}$ where $\dot{M}\subset \dot{V}_b$ is spanned by 
the elements $2\alpha$ with $\alpha\in\dot{\pi}$; these elements
are linearly independent, so $\dot{M}\cong Q(\dot{\pi})$.

By~\cite{GHS}, 10.3 one has $\Sk^D(v)=W$ for odd $n$. 

Consider the case  $n=2k$. 
One has $\dot{\pi}=\{\vareps_i-\vareps_{i+1}\}_{i=1}^{2k-1}$. By~\cite{GHS}, 10.3
$$Sk^D(v)=\dot{W}\ltimes \tau(M'),$$ 
where 
$M'\subset \dot{V}_b$ is spanned by 
$\{ 2(\vareps_i-\vareps_{i+1})\}_{i=1}^{2k-2}\cup\{\beta\}$ for
 $\beta:=\sum\limits_{i=1}^k (\vareps_{2i-1}-\vareps_{2i})$.
 Note that $M'\cong Q(\dot{\pi})$.
 Since  $M'\subset \dot{V}_b$ we have
 $\tau(M')\cong {M}'$. This gives 
$$\Sp^D(v)=Sk^D(v)/W=M'/\dot{M}=\mathbb{Z}_2.$$

\newpage
\appendix
\includepdf[pages={1},scale=.90,pagecommand=\section*{Appendix A: Spine of \(G(3)^{(1)}\)}\label{appendix:a}]{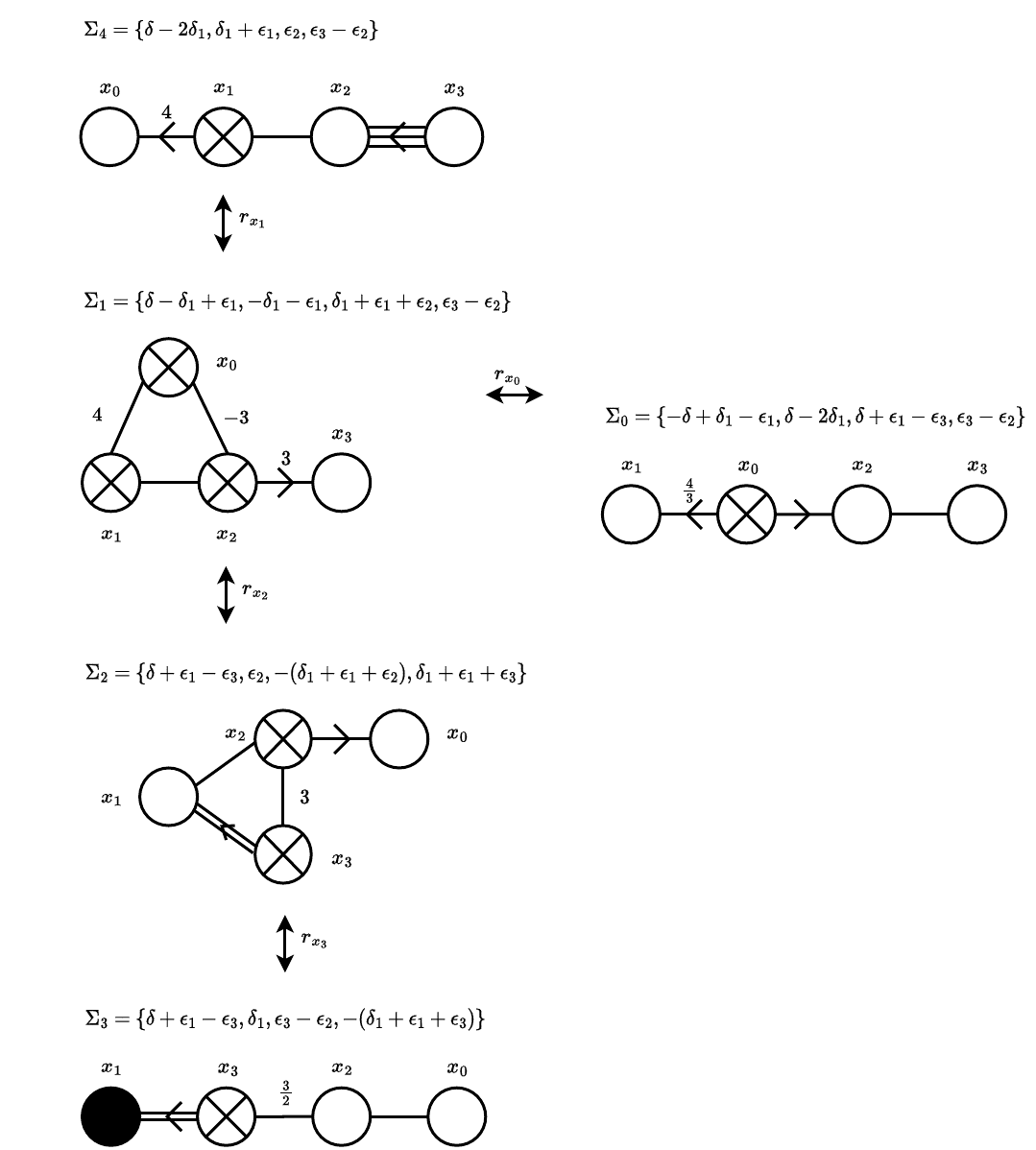}
	
\includepdf[pages={1},scale=.80,pagecommand=\section*{Appendix B: Spine of \(F(4)^{(1)}\)}\label{appendix:b}]{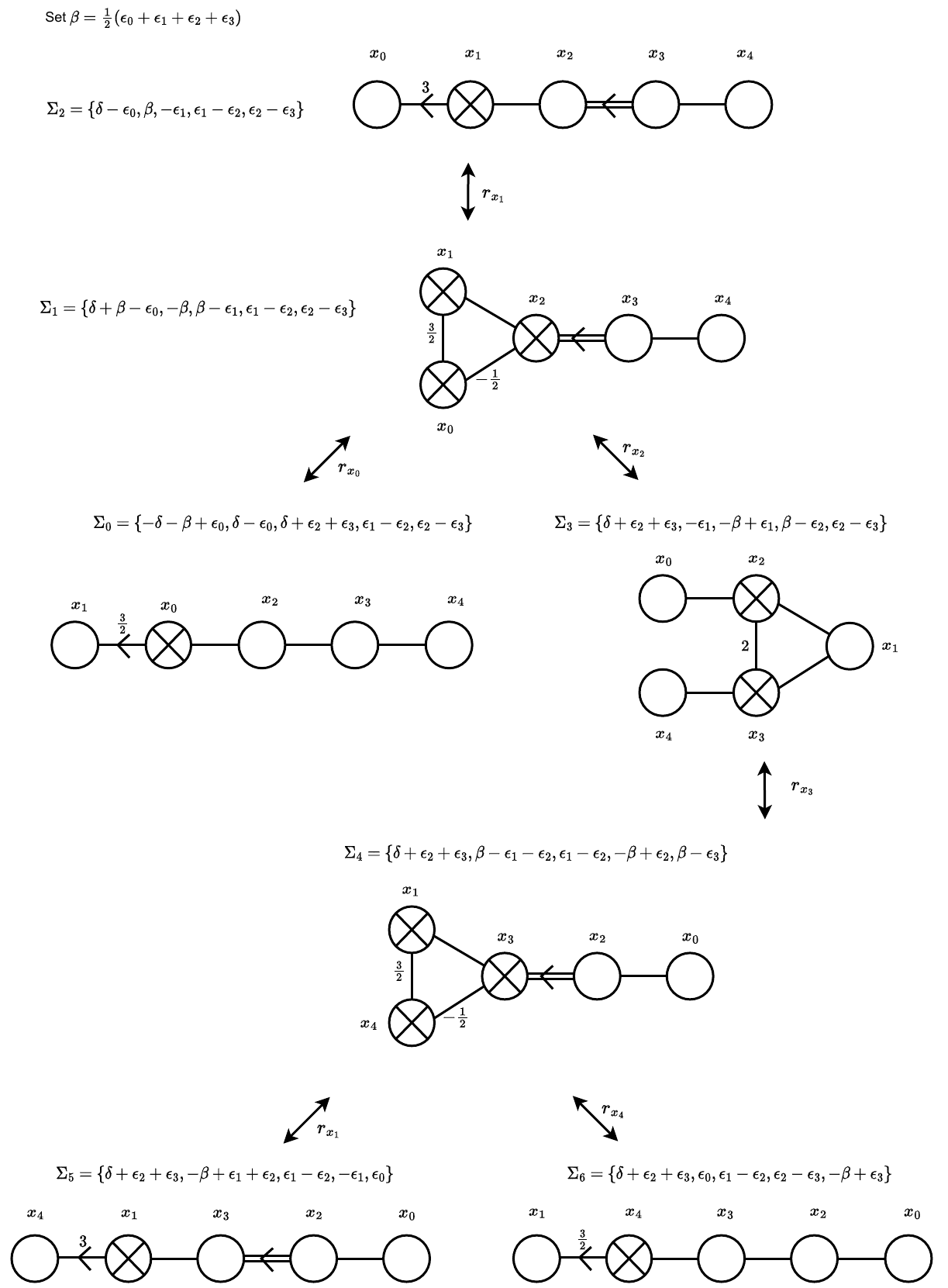}

		\section*{}	
		
%

	\end{document}